\documentclass[a4paper, 11pt]{article}
\usepackage{authblk}
\usepackage{amsmath}
\usepackage{amsthm}
\numberwithin{equation}{section}
\usepackage{xfrac}
\usepackage{amssymb}
\usepackage{hyperref}
\usepackage{graphicx}
\usepackage[percent]{overpic}
\usepackage{bm}
\usepackage{gensymb} 
\usepackage{subcaption}
\usepackage{cases}
\usepackage{geometry}
\usepackage{xcolor}
\usepackage[utf8]{inputenc}
\usepackage{enumitem}
\usepackage{algorithmic}
\usepackage{algorithm}
\usepackage{tikz}
\usepackage{booktabs}
\graphicspath{{figures/}} 
 
\usepackage{nomencl}
\makenomenclature

\usepackage{cancel} 
 
\setlist[enumerate,1]{label={(\arabic*)}}
\setlist[enumerate,2]{label={(\roman*)}}

\renewcommand{\Im}{\text{Im}}
\renewcommand{\Re}{\text{Re}}

\newtheorem{lemma}{Lemma}[section]
\newtheorem{theo}{Theorem}[section]

\title{Robust optimization of control parameters for WEC arrays using stochastic methods}

\author[1]{Marco Gambarini}
\author[1]{Gabriele Ciaramella}
\author[1]{Edie Miglio}
\author[2]{Tommaso Vanzan}

\affil[1]{MOX, Dipartimento di Matematica, Politecnico di Milano, Piazza Leonardo da Vinci 32, 20133 Milano, Italy}
\affil[2]{CSQI Chair, Institute of Mathematics, Ecole Polytecnique Fédérale de Lausanne, Lausanne, 1015, Switzerland}

\date{}

\begin{document}

\maketitle 

\begin{abstract}
    This work presents a new computational optimization framework for the robust control  of parks of Wave Energy Converters (WEC) in irregular waves. The power of WEC parks is maximized with respect to the individual control damping and stiffness coefficients of each device. The results are robust with respect to the incident wave direction, which is treated as a random variable. Hydrodynamic properties are computed using the linear potential model, and the dynamics of the system is computed in the frequency domain. A slamming constraint is enforced to ensure that the results are physically realistic. We show that the stochastic optimization problem is well posed. Two optimization approaches for dealing with stochasticity are then considered: stochastic approximation and sample average approximation.
    The outcomes of the above mentioned methods in terms of accuracy and computational time are presented. The results of the optimization for complex and realistic array configurations of possible engineering interest are then discussed. Results of extensive numerical experiments demonstrate the efficiency of the proposed computational framework.
\end{abstract}

\section{Introduction}

Ocean waves and tides are a renewable source of energy with large predicted growth over the next years and decades. Proposals of the European Commission \cite{eurocomm2020} regard as ``realistic and achievable'' a target of 1 GW installed capacity by 2030 and of 40 GW by 2050 for ocean energy.\footnote{For comparison, the installed wind energy capacity in Europe was 236 GW as of 2021 \cite{windeurope2022}.} A massive deployment of wave energy is expected to drive costs down and make this choice competitive, especially for islands and offshore applications \cite{oceaneuroreport}. Since single devices for wave energy conversion have small production capabilities and a limited possibility of scaling up compared to other technologies such as wind turbines \cite{Babarit2012}, it is crucial from the economical point of view to install them in arrays. 
The relatively small distance between devices in arrays implies the appearance of mutual interactions. The result of these interactions is known as the park effect \cite{Babarit2013}, and it has to be carefully taken into account in the design process. 

In order to better utilize the energy resource and improve survivability in extreme conditions, control strategies of devices in the array need to be appropriately designed. In particular, hydrodynamic interactions in the controlled system should be favourable in terms of extracted power. The most common control techniques for wave energy converters are presented in \cite{arenas2019}, reviewing constant and time-varying damping control, reactive control, latching and model predictive control. Classical results obtained through frequency domain analysis for regular waves with control parameters kept constant in time are presented in \cite{Falnes2020}. Strategies for irregular waves (in time domain) and related results are discussed in \cite{Korde2016}.
In an array, the same control parameters may be assigned to all devices, or each device may be assigned its own. The latter approach is referred to as individual optimization and it has been performed in \cite{Backer2010} and \cite{Sinha2016} using SQP (Sequential Quadratic Programming), optimizing for a single wave direction and then checking the performance for different directions.
Another aspect of array design and optimization, that will not be treated in this work, is the layout of devices. Layout optimization can be performed either before or together with control optimization. For more details, see the reviews \cite{Yang2022, Goeteman2020}. 

In this work, we present a numerical strategy for the robust optimization (maximization) of the power production of arrays with reactive control: damping and stiffness coefficients of the generators are optimized. We allow each body to have its own control coefficients and assume that the corresponding parameters are constant in time. As shown in \cite{Eriksson2007}, damping is actually a nonlinear function of the velocity for real electrical generators. Our assumptions, though, allow the use of a linear framework, and the results can still be a guide for the design process. We also allow control stiffnesses to assume negative values, meaning that the control action must be able to counteract the hydrostatic restoring force, that is the force acting on a floating body when it is displaced from its buoyancy equilibrium position \cite[Sec. 5.9]{Falnes2020}. Even though a negative stiffness could seem counterintuitive or unphysical, some technologies are in fact capable of achieving this effect: see, e.g., \cite{Todalshaug2016,Zhang2019}. For this reason, we do not exclude apriori this case from the set of admissible solutions. 
Moreover, the optimal solutions must fulfill the slamming constraint, meaning that devices must not leave the water, and in some cases a constraint of positive stiffness is added, to assess the benefit of using negative stiffness technologies.

To obtain good performances in real scenarios, the result of optimization needs to be robust against the uncertainty in the incident waves direction. 
To the best of our knowledge, this problem has never been addressed in the literature.
We formulate a stochastic (robust) optimal control problem governed by linear hydrodynamic equations, describing the physics of the system, and characterized by control and state constraints.
The numerical solution of this optimization problem requires an adequate treatment of different issues.
A reduced approach based on the adjoint equation allows us to work in the space of solutions of the hydrodynamic equations and consider the control as the only optimization variable. Stochasticity is addressed by employing two
different approaches suitable for stochastic optimization problems: SAA (Sample Average Approximation) and SA (Stochastic Approximation) \cite{Shapiro2014}.
The slamming (state) constraint is treated by using a penalty approach, while the stiffness (control) constraint is enforced by projection.
Therefore, our proposed numerical computation framework is obtained by two nested iterative processes. The outer iteration is essentially a quadratic penalty method increasing at each iteration the weight of the penalty term for the slamming constraint. 
At each penalty iteration a stochastic optimal control problem is solved by a projection method based either on SA or SAA.
The performances obtained by individually optimizing the coefficients of each device are compared with those obtained by setting the coefficients of all devices equal to the tuning parameters of the isolated case. We show that sets of parameters corresponding to feasible solutions for isolated devices can lead to unfeasible (incompatible with the constraints) solutions when the devices are installed in arrays, and that the optimization process leads to sensible increases of the power output in most cases.

The paper is organized as follows. In Section~\ref{sec:modeling}, we detail the hydrodynamic and mechanical models used and discuss the constraints. In Section~\ref{sec:optmethods}, we study the well-posedness of the problem and present optimization methods and algorithms. Sections~\ref{sec:numexp} and \ref{sec:2peak} contain numerical experiments performed with the proposed algorithms. In Section~\ref{sec:concl} the obtained results are commented and conclusions are drawn. 

\section{Modeling of WEC}\label{sec:modeling}
\subsection{Hydrodynamic modeling} \label{sec:hydro}

To model the hydrodynamic behaviour of WEC (Wave Energy Converter) arrays, we use linear potential wave theory \cite{Chakrabarti1987}.
In particular, we make the kinematic assumption of irrotational flow (zero vorticity), which implies the existence of a scalar potential $\phi$ whose gradient is the flow velocity. This assumption is well verified for external flows, outside boundary layers \cite{Kundu2012}.
Under the assumption that the displacements of the free surface and of the surfaces of floating bodies are small, boundary conditions can be linearized and enforced at the reference surfaces. This avoids the difficulty of dealing with moving boundaries and permits solving the problem in the frequency domain.
Finally, we assume that viscous effects can be neglected.

\begin{figure}[h]
    \centering
    \begin{tikzpicture} 
    \draw (0, 0) -- (2, 0);
    \node (dev1) at (2.5, -0.8) {$\Gamma_{d,1}$};
    \draw[very thick] (2, 0) -- (2, -0.5) -- (3, -0.5) -- (3, 0);
    \draw (3, 0) -- (4.5, 0);
    \node (dev1) at (5., -0.8) {$\Gamma_{d,2}$};
    \draw[very thick] (4.5, 0) -- (4.5, -0.5) -- (5.5, -0.5) -- (5.5, 0);
    \draw (5.5, 0) -- (6, 0);
    \node (dots) at (6.6, -0.8) {$\dots$};
    \draw (7, 0) -- (7.5, 0);
    \node (dev1) at (8., -0.8) {$\Gamma_{d,N_b}$};
    \draw[very thick] (7.5, 0) -- (7.5, -0.5) -- (8.5, -0.5) -- (8.5, 0);
    \draw (8.5, 0) -- (10.5, 0);
    \draw[dashed] (10.5, 0) -- (10.5, -3);
    \draw (10.5, -3) -- (0, -3);
    \draw[dashed] (0, -3) -- (0,  0);
    \node (omega) at (5, -2) {$\Omega$};
    \node (surf) at (1., 0.2) {$\Gamma_s$};
    \node (bottom) at (1., -3.3) {$\Gamma_b$};
    \end{tikzpicture}
    \caption{Domain of the hydrodynamic problem (slice)}
    \label{fig:domain}
\end{figure}
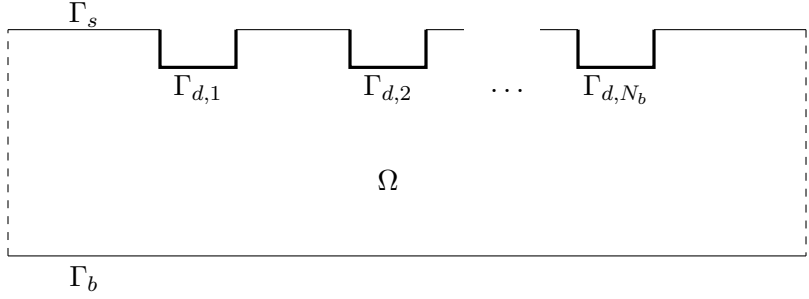

For a single frequency $\omega$, the potential is written as $\phi(\bm{x}, t) = \Re\left[\hat{\phi}(\bm{x}) \exp(-i\omega t) \right]$ and $\hat{\phi}$ is the solution of the hydrodynamic problem 

\begin{subnumcases}{}
\Delta \hat{\phi} = 0 & \text{in $\Omega$}\\[.5em]
\dfrac{\partial \hat{\phi}}{\partial n}  =0 & \text{on $\Gamma_b$}\\[.5em]
\dfrac{\partial \hat{\phi}}{\partial z} - \dfrac{\omega^2}{g} \hat{\phi} = 0 & \text{on $\Gamma_s$}\\[.5em]
\dfrac{\partial \hat{\phi}}{\partial n} = \hat{\bm{v}}_\ell \cdot \bm{n} & \text{on $\Gamma_{d,\ell}$, $\ell = 1, \dots, N_b$} \label{eq:bcbody},
\end{subnumcases}
where $\Omega \subset \mathbb{R}^3$ is the fluid domain, bounded by the sea bottom surface $\Gamma_b$, the reference free surface $\Gamma_s$ and the device immersed surfaces $\Gamma_{d,\ell}$, and unbounded on the lateral sides (see Fig.~\ref{fig:domain}).
The complex amplitude of the velocity of the $\ell$-th device is denoted by $\hat{\bm{v}}_\ell$. We assume that the bodies are restricted to vertical motion (heave). Hence, we have 
\begin{equation*}
\hat{\bm{v}}_\ell = [0, 0, \hat{v}_\ell] = [0, 0, -i\omega \hat{\zeta}_\ell],
\end{equation*}
where $\hat{\zeta}_\ell$ is the complex amplitude of the vertical position of the $\ell$-th device.
The free surface elevation $\hat{\eta}$ can be computed using the kinematic condition
$\hat{\eta} = i \hat{\phi} \omega/g $.
For the $\ell$-th device, the dynamic problem in the $z$ direction is
\begin{equation}
(-\omega^2 m_\ell + i \omega c_\ell + k_\ell + s_\ell)\hat{\zeta}_\ell = \hat{F}_\ell
\label{eq:dynsingledevice}
\end{equation} 
where $m_\ell$ is the mass, $k_\ell$ is the combined hydrostatic and mechanical stiffness, $c_\ell$ and $s_\ell$ are the equivalent damping and stiffness coefficients of the electrical generator, respectively, and $\hat{F}_\ell$ is the vertical component of the hydrodynamic force. The latter is obtained from the potential $\hat{\phi}$:
\begin{equation}
\hat{F}_\ell = -i\omega \rho \int_{\Gamma_{d,\ell}} \hat{\phi} n_z \, \mathrm{d\Gamma},
\label{eq:forceint}
\end{equation}
$n_z$ being the vertical component of the outer surface normal of the body. This is a consequence of the unsteady, linearized Bernoulli equation, which holds thanks to the assumption of inviscid flow.

The equations above show that there is a two-way coupling between the hydrodynamic problem and the mechanical problem. The two may be solved monolithically. However, for the purpose of control design, it is more common and insightful to define intermediate hydrodynamic quantities, namely added mass, radiation damping and excitation force \cite{Falnes2020}. 
In this way, once the hydrodynamic problem has been solved to determine such quantities, the dynamic problem can be treated independently. To apply this approach, the potential is split into the sum of three harmonic functions: an incident potential $\hat{\phi}_0$, a diffraction potential $\hat{\phi}_d$ and a radiation potential $\hat{\phi}_r$. Given a wave of direction $\theta$ and height $H$, the complex amplitude of the incident potential $\hat{\phi}_0$ is given by \cite[Sec. 3.5]{Chakrabarti1987}:
\begin{equation}
\hat{\phi}_0(x, y, z) = -i \dfrac{H}{2} \dfrac{g}{\omega} \dfrac{\cosh[k(z+h)]}{\cosh(kh)} \exp[ik(x\cos\theta + y\sin\theta)].
\label{eq:incpotential}
\end{equation}
This function satisfies the Laplace equation in $\Omega$ and the boundary conditions on the bottom and the free surface, but not the boundary condition on device surfaces. We can then rewrite \eqref{eq:bcbody} as
\begin{equation*}
\dfrac{\partial \hat{\phi}_d}{\partial n} + \dfrac{\partial \hat{\phi}_r}{\partial n} = - \dfrac{\partial \hat{\phi}_0}{\partial n} + \hat{\bm{v}}_\ell \cdot \bm{n} \quad \text{on $\Gamma_{d,\ell}$} , \quad \ell=1, \dots, N_b.
\end{equation*}
Such condition can be satisfied by requiring that the diffraction and radiation potentials satisfy 
\begin{align*}
\dfrac{\partial \hat{\phi}_d}{\partial r} = - \dfrac{\partial \hat{\phi}_0}{\partial n} \quad \text{and} \quad
\dfrac{\partial \hat{\phi}_r}{\partial r} = \hat{\bm{v}}_\ell \cdot \bm{n} \quad \text{on $\Gamma_{d,\ell}$}.
\end{align*}
The radiation potential $\phi_r$ can then be decomposed as a sum of the effects of radiation from each body,
\begin{equation*}
\hat{\phi}_r = \sum_{\ell=1}^{N_b} \hat{v}_\ell \, \varphi_\ell,
\end{equation*}
so that each $\varphi_\ell$ satisfies a condition of unit velocity amplitude on the $\ell$-th body and an homogeneous condition on all the others. Diffraction and radiation potentials are also required to satisfy a radiation condition at infinity, which is needed for energy conservation. The full potential $\hat{\phi} = \hat{\phi}_0 + \hat{\phi}_d + \hat{\phi}_r$ is then computed by solving a single diffraction problem to obtain $\hat{\phi}_d$ and a radiation problem for each body to get $\varphi_\ell$, $\ell=1,\dots,N_b$:
\begin{equation}
\begin{split}
\begin{cases}
\Delta \hat{\phi}_d = 0 & \text{in $\Omega$}\\[.5em]
\dfrac{\partial \hat{\phi}_d}{\partial n}  =0 & \text{on $\Gamma_b$}\\[.5em]
\dfrac{\partial \hat{\phi}_d}{\partial z} - \dfrac{\omega^2}{g} \hat{\phi}_d = 0 & \text{on $\Gamma_s$}\\[.5em]
\dfrac{\partial \hat{\phi}_d}{\partial n} = - \dfrac{\partial \hat{\phi}_0}{\partial n} & \text{on $\Gamma_{d,\ell}$, $\ell = 1, \dots, N_b$}\\[.5em] \dfrac{\partial \hat{\phi}_d}{\partial r} - i k \hat{\phi}_d  \to 0 & \text{for $r \to \infty$},
\end{cases}
\\
\begin{cases}
\Delta \varphi_\ell = 0 & \text{in $\Omega$}\\[.5em]
\dfrac{\partial \varphi_\ell}{\partial n}  =0 & \text{on $\Gamma_b$}\\[.5em]
\dfrac{\partial \varphi_\ell}{\partial z} - \dfrac{\omega^2}{g} \varphi_\ell = 0 & \text{on $\Gamma_s$}\\[.5em]
\dfrac{\partial \varphi_\ell}{\partial n} = n_{z} & \text{on $\Gamma_{d,\ell}$} \\[.5em]
\dfrac{\partial \varphi_\ell}{\partial n} = 0 & \text{on $\Gamma_{d,m}, \quad m=1, \dots, N_b \land m\neq\ell$}\\[.5em] \dfrac{\partial {\varphi}_\ell}{\partial r} - i k \varphi_\ell  \to 0 & \text{for $r \to \infty$}.
\end{cases}
\end{split}
\label{eq:diffprob}
\end{equation}

Using the decomposition of the potential, equation \eqref{eq:forceint} can be rewritten as
\begin{equation}
\hat{F}_\ell = \underbrace{i \omega \rho \int_{\Gamma_{d,\ell}} (\hat{\phi}_0 + \hat{\phi}_d) n_z \, \mathrm{d \Gamma}}_{\hat{F}_{e,\ell}} + \sum_{m=1}^{N_b} \hat{v}_m \underbrace{\int_{\Gamma_{d,\ell}} \varphi_m n_z \, \mathrm{d \Gamma}}_{R_{\ell m}},
\label{eq:forcedecomp}
\end{equation}
where the first term is the excitation force $\hat{F}_{e,\ell}$, and the $m$-th term of the summation is the radiation impedance $R_{\ell m}$.
This is further split into real (in phase with velocity) and imaginary (out of phase) parts:
\begin{equation}
R_{\ell m} = B_{\ell m} + i \omega A_{\ell m},
\label{eq:raddecomp}
\end{equation}
where $A_{\ell m}$ is called added mass and $B_{\ell m}$ radiation damping. 

Summarising, using the above decomposition of the potential together with \eqref{eq:forcedecomp} and \eqref{eq:raddecomp}, the dynamical system \eqref{eq:dynsingledevice} can be recast as a complex linear system for each frequency:
\begin{equation*}
\sum_{m=1}^{N_b} Z_{\ell m}(\omega_q)  \hat{\zeta}_{m, q} = \hat{F}_{e, \ell}^\theta(\omega_q), \quad q=1, \dots, N_f, 
\end{equation*}
where the impedance matrix $Z(\omega_q)$ is defined as
\begin{equation}
Z_{\ell m}(\omega_q) = -\omega_q^2 (m_\ell \delta_{\ell m} + A_{\ell m}(\omega_q)) - i \omega_q (c_\ell \delta_{\ell m} + B_{\ell m}(\omega_q)) + (k_\ell + s_\ell) \delta_{\ell m},
\label{eq:impmatdef}
\end{equation}
and $\delta_{\ell m}$ is the Kronecker symbol. The dependence on frequency $\omega_q$ and on direction $\theta$ has been made explicit. Notice that only excitation forces depend on the wave direction. The only terms coupling the motions of different devices are the added mass and radiation damping coefficients.

\subsection{Modeling of constraints}\label{sec:constraints}

The model described in Section~\ref{sec:hydro} needs to be complemented with constraints to guarantee that the solutions are physical and that the control forces are feasible. A description of some possible constraints is reported in \cite{Backer2010}, namely slamming, stroke and force constraints. The slamming constraint guarantees that the bodies do not leave the water, and it is needed because the linear potential model is built on the assumption of small displacements. The stroke constraint limits the amplitude of body motions; it is related to the dimensions of the power take-off system. Finally, the force constraint ensures that the control force can be realized by an electromechanical system.

In this work, since we aim for generality, we focus on the slamming constraint, that is a purely hydrodynamic one. Conversely, force and stroke constraints depend strongly on the design details of a specific power take-off system. Once a system is specified, the other two constraints can be imposed with the same methodologies presented here.

The slamming constraint is a state constraint expressed in time domain as
\begin{equation}
w_\ell(t) = \zeta_\ell(t) - \eta_\ell(t) \leq d_\ell \quad \forall t, \quad \ell=1, \dots, N_b,
\label{eq:slamtimedom}
\end{equation} 
where $d_\ell$ is the draft of the $\ell$-th device and $\eta_\ell$ is the wave height at the center of the $\ell$-th device.
Let us now formulate this constraint in frequency domain. First, we consider the approximation introduced in \cite{Backer2010}, which consists in neglecting diffraction and radiation effects and taking the wave height from the incident wave whose complex amplitude is
\begin{equation*}
\hat{\eta}_0 = A  \exp[ik(x\cos\theta + y\sin\theta)].
\end{equation*}
For monochromatic waves, constraint \eqref{eq:slamtimedom} implies
\begin{equation}
\left|\hat{\zeta}_\ell - \hat{\eta}_{\ell 0}\right| \leq d_\ell.
\label{eq:1freqslam}
\end{equation} 
For irregular waves, it is not possible to obtain an explicit formula such as \eqref{eq:1freqslam}. The constraint can instead be imposed in a statistical sense by limiting the fraction of time above threshold $t_{at}$, meaning the fraction of time for which the constraint is violated, or the fraction of peaks above threshold $Q$, where constraint violations are considered as discrete events. It has been observed that for deep sea waves the free surface elevation $\eta(t)$ at a fixed point is well described as a Gaussian process \cite{Rychlik1997}. Moreover, if the input of a linear time invariant system is a Gaussian process, then its output is also a Gaussian process. We thus work under the approximation that $w(t)$ is a Gaussian process, which implies that the constraint can be enforced on the root mean square value $w_{rms}$ \cite{Lalanne1999}. Such quantity can be computed from frequency domain data:
\begin{equation}
w_{\ell, rms} = \sqrt{\dfrac{1}{2} \sum_{q=1}^{N_f} |w_{q\ell}|^2}.
\label{eq:defwrms}
\end{equation}
The constraint is then written for each body as 
\begin{equation}
w_{\ell, rms} \leq \alpha d_\ell,
\label{eq:slammingrmsconstraint}
\end{equation}
where $\alpha$ is a constant chosen to control the exceedance probability. In particular, the fraction of time above threshold is
\begin{equation*}
t_{at, \ell} = 2 \left[1 - \Phi\left( \dfrac{d_\ell}{w_{\ell, rms}} \right) \right],
\end{equation*}
where $\Phi(\cdot)$ is the normal cumulative distribution function. The fraction of peaks above threshold is, for a narrow banded process \cite{Lalanne1999}, 
\begin{equation}
Q_{\ell} = \exp \left( - \dfrac{d_\ell^2}{2 w_{\ell, rms}^2} \right).
\label{eq:peakfraction}
\end{equation}
The process we consider does not have a narrow band. However, it can be shown that, for a generic Gaussian process, \eqref{eq:peakfraction} is a cautionary estimate (upper bound) of the exceedance probability. As an example, setting
$\alpha=1/2$ and assuming that the constraint \eqref{eq:slammingrmsconstraint} is active, i.e. $w_{\ell, rms} = \alpha d_\ell$, yields $t_{at,\ell}=4.55\%$ and $Q_{p,\ell}=13.5\%$.

For convenience, we rewrite \eqref{eq:slammingrmsconstraint} using \eqref{eq:defwrms} as 
\begin{equation*}
g_\ell\left(\hat{\bm{\zeta}}\right) \leq 0, \quad g_\ell(\hat{\bm{\zeta}}) := \sum_q \left| \hat{\zeta}_{q\ell} - \hat{\eta}_{q\ell} \right|^2 - 2 \alpha^2 d_\ell^2, \quad \ell=1, \dots, N_b,
\end{equation*}
and we recast it as a differentiable equality constraint $h_\ell\left(\hat{\bm{\zeta}}\right)=0$, where
\begin{equation*} 
h_\ell\left(\hat{\bm{\zeta}}\right) := \left[g_\ell\left(\hat{\bm{\zeta}}\right)\right]_+^2, \quad [g]_+ := \max(g, 0).
\end{equation*}
The two constraints are equivalent in the sense that they have the same active set:
$h_\ell\left(\hat{\bm{\zeta}}\right) = 0 \Leftrightarrow g_\ell\left(\hat{\bm{\zeta}}\right) \leq 0$.


In addition to the slamming constraint, a constraint of positive control stiffness 
\begin{equation*}
s_\ell \geq 0, \quad \ell=1, \dots, N_b 
\end{equation*}
is enforced in some of the simulations to quantify the advantage of using negative stiffness technologies. 

\subsection{Forces stochasticity modeling}\label{sec:stochforce}
As noted in section \ref{sec:hydro}, the wave climate influences the system dynamics through the excitation force $\hat{F}_e$. Since our aim is to obtain control parameters that are robust with respect to the wave direction, we define a realization as a condition with irregular unidirectional waves, where the direction is sampled from a probability distribution. By irregular we mean that waves are not monochromatic, but rather a superposition of harmonics. Frequency and amplitude of each harmonic are the same for all realizations.

The wave climate is described by the directional spectrum $S(f, \theta)$, which is the distribution of wave power with respect to frequency $f$ and direction $\theta$. Its general form is \cite{Goda2010}
\begin{equation*}
S(f, \theta) = S(f) D(\theta | f),
\end{equation*}
with short (high-frequency) waves generally showing a larger directional spreading than long waves. In this work, we consider the simpler approximate form $S(f, \theta) = S(f) D(\theta)$. 
For the frequency spectrum we use the two-parameter Piersov-Moskovitz form \cite{Falnes2020}, whose expression is
\begin{equation}
    S(f) = \gamma_1 f^{-5} \exp\left(- \gamma_2 f^{-4}\right),
    \quad \text{ with } \quad \gamma_1 = \dfrac{\gamma_2 H_s^2}{4}, \quad \gamma_2 = \dfrac{5}{4} f_p^4,
    \label{eq:PMspectrum}
\end{equation}
and where $H_s$ is the significant wave height and $f_p$ is the peak frequency. The spectrum is discretized using the deterministic spectral amplitude method \cite{Chakrabarti1987} by first defining $N_f$ intervals of center $f_q$ and width $\Delta f_q$. Intervals are chosen so that they all correspond to the same power fraction (meaning they have in general different widths). The incident wave potential is then a superposition of potentials of the form \eqref{eq:incpotential}, where the $q$-th component has height $
H_q = 2 \sqrt{2 S(f_q) \Delta f_q}$. If necessary, the wave profile in time domain can be reconstructed as
\begin{equation*}
\eta(x, t) = \sum_{q=1}^{N_f} \dfrac{H_q}{2} \cos\left[k(x\cos\theta + y\sin\theta) - \omega_q t + \phi_i\right],
\end{equation*}
with $\phi_i \sim \mathcal{U}(0, 2\pi)$. In the limit of $N_f \to \infty$,  $\eta(x,t)$ is a Gaussian process \cite{Tucker1984}.

For the directional part, choosing a spectrum that has a closed-form inverse cumulative distribution makes sampling straightforward. A suitable choice is Donelan's spreading function \cite{Goda1999}:
\begin{equation}
D(\theta) = \dfrac{1}{2} \beta (\text{sech}[\beta (\theta-\theta_0)] )^2,
\label{eq:donelan}
\end{equation}
with the corresponding cumulative distribution function
\begin{equation*}
\tilde D(\theta) = \dfrac{1}{2} \left[\text{tanh}[\beta (\theta-\theta_0)] + 1 \right],
\end{equation*}
where $\theta_0$ is the dominating direction and $\beta$ is a scale parameter: for increasing $\beta$, the distribution is more sharply peaked (see Fig.~\ref{fig:donelan}). 
To sample a wave direction, we use the inverse transformation technique \cite[Sec. 9.2]{Haldar2000}: we first obtain a sample $\xi$ from a uniformly distributed variable $X \sim \mathcal{U}(0,1)$, and then apply the transformation 
\begin{equation*}
\theta = \tilde D^{-1}(\xi).
\end{equation*}

\begin{figure}[h!]
\centering
\subfloat{\includegraphics[width=0.45\textwidth]{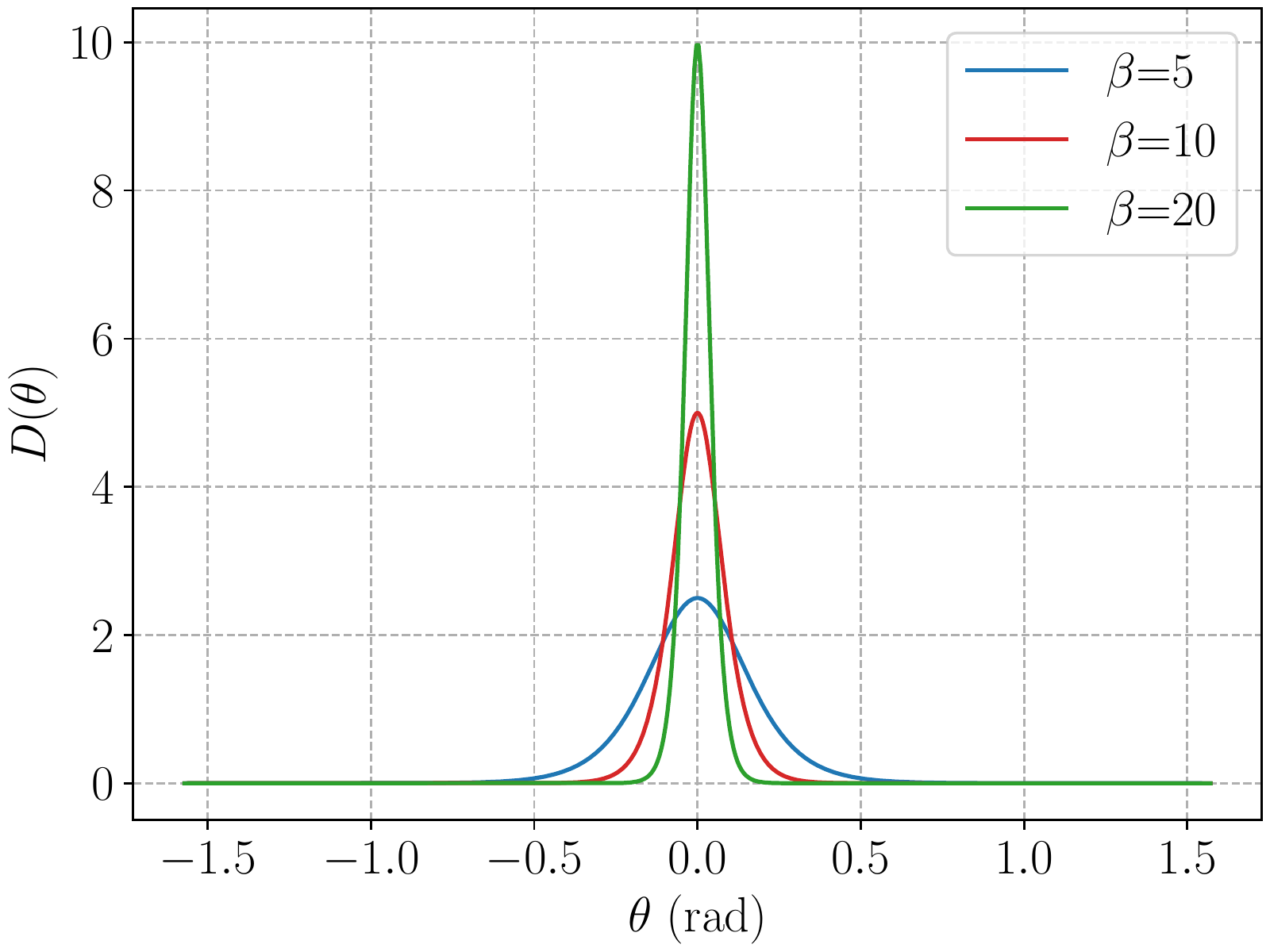}}
\hfill
\subfloat{\includegraphics[width=0.45\textwidth]{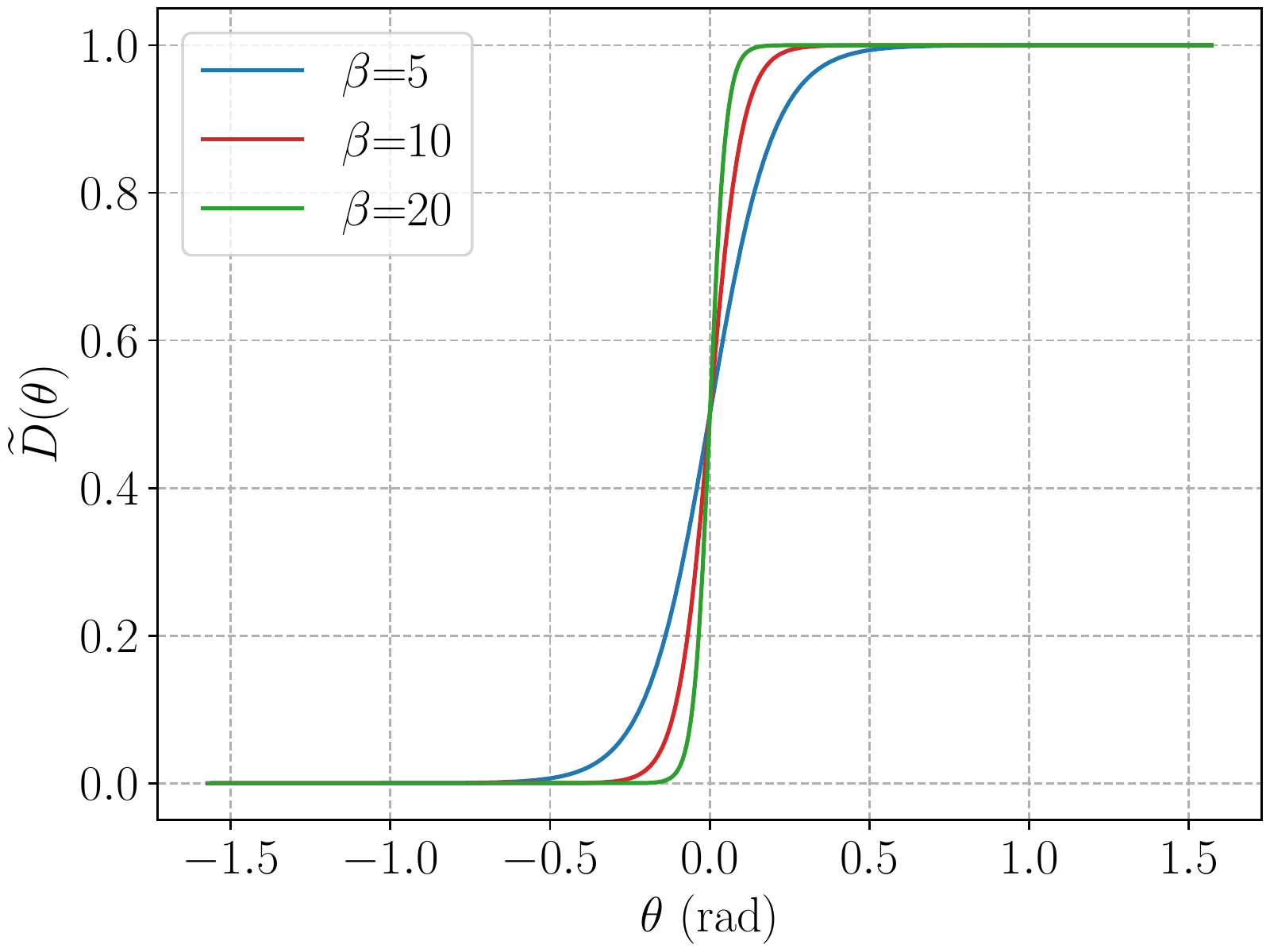}}
\caption{Probability density (left) and cumulative distribution (right) for Donelan's function}
\label{fig:donelan}
\end{figure}

\subsection{Computational aspects and Haskind's relation}

The hydrodynamic problem is solved with the boundary element method (BEM), based on boundary integral reformulations of \eqref{eq:diffprob}. 
In particular, the source distribution method is used \cite[Sec. 4.11]{Newman1977}: an unknown source distribution $\sigma$ is defined on the device surfaces $\Gamma_{d,\ell}, \, \ell=1, \dots, N_b$. The potential in any point $\bm{x}$ of the domain is obtained as the effect of such surface distribution through a convolution integral with a Green function $\mathcal{G}$, which satisfies the boundary conditions on $\partial \Omega \setminus \cup_\ell \Gamma_{d,\ell}$:
\begin{equation}
\phi(\bm{x}) = \int_{\cup_\ell \Gamma_{d,\ell}} \sigma(\bm{x}') \mathcal{G}(\bm{x}; \bm{x}') \, \mathrm{d} \bm{x}', \quad \forall \bm{x} \in \Omega.
\label{eq:reconstrphi}
\end{equation}
The boundary conditions on the body surface are then enforced:
\begin{equation}
\dfrac{1}{2} \sigma(\bm{x}) + \int_{\cup_\ell \Gamma_{d,\ell}} \sigma(\bm{x}') \dfrac{\partial \mathcal{G}}{\partial n} \, \mathrm{d} \bm{x}' = \text{b.c.}
\label{eq:findsigma}
\end{equation}
where the b.c. depends on the considered problem in \eqref{eq:diffprob}
The two equations \eqref{eq:reconstrphi}, \eqref{eq:findsigma} are transformed into linear algebraic equations by discretization:
\begin{align*}
K_1 \bm{\sigma} &= \bm{b}, \qquad
\bm{\phi} = K_2 \bm{\sigma}.
\end{align*}
Once the first system is solved, the potential is obtained from the second. 

We now discuss the computational effort required to construct a realization of the excitation force. In principle, one would need to draw a sample $\theta$ and then, for all frequencies, compute the corresponding incident wave potential using \eqref{eq:incpotential} and use it as a boundary condition for a diffraction problem. This requires storing or recomputing $N_f$ dense, and possibly large, matrices and solving each system once. This would make the process very computationally demanding. 
A much cheaper approach would be to run a batch of simulations for different values of the wave direction before the optimization process and then recover the force by interpolation for each realization. This strategy, however, could result in large inaccuracies because of the difficulty and ambiguity of interpolating complex functions \cite{Thornhill2019}. 
Another possibility is the use of Haskind's relation, an exact integral relation which allows computing the diffraction force for any wave direction as a combination of the potential of incident waves and radiation potentials \cite[Sec. 5.4]{Falnes2020}. The recomputation of the incident wave field is relatively cheap, while radiation potentials would need to be computed anyway to obtain the added mass and radiation damping matrices. For these reasons, this approach is intermediate in computational cost between full diffraction simulations and interpolation. The use of Haskind's relation is the choice that has been made here. Its expression is
\begin{equation*}
\hat{F}_{e,\ell} = - i \omega \rho \int_{\cup_\ell \Gamma_{d,\ell}} \left[\hat{\phi}_0\dfrac{\partial \varphi_\ell}{\partial n} - \varphi_\ell \dfrac{\partial \hat{\phi}_0}{\partial n} \right]\, \mathrm{d\Gamma}.
\end{equation*}
The integral is discretized as a sum on the mesh elements.

\section{Design of damping and stiffness by optimization}\label{sec:optmethods}

In this section, we formalize the optimization problem and present numerical methods for its solution.

\subsection{Problem statement}
The cost function is the average power with negative sign, which can be written as
\begin{equation*}
    J := - \sum_{\ell=1}^{N_b} \dfrac{1}{T} \int_0^T c_\ell \left( \dot{\zeta}_\ell(t) \right)^2 \, \mathrm{d}t.
\end{equation*}
The motion of each body is realized in the time domain as a sum of harmonics. Because of the orthogonality of such functions, the mean square of the signal depends only on the amplitudes of the harmonic components, phases being irrelevant. If incident waves are realized as the sum of $N_f$ monochromatic waves, as explained in Section \ref{sec:stochforce}, the cost is obtained by summing the contributions of all frequency components:
\begin{equation}
    J(\bm{u}, \bm{\hat{\zeta}}^\theta) = -\dfrac{1}{2} \sum_{q=1}^{N_f} \omega_q^2 \left( \hat{\bm{\zeta}}_q^\theta \right)^H C \hat{\bm{\zeta}}_q^\theta,
    \label{eq:Jdeffreqdom}
\end{equation}
where $\bm{u} \in \mathbb{R}^{2 N_b}$ is the control vector having the form
$\bm{u} = \left[ c_1, \dots, c_{N_b}, s_1, \dots, s_{N_b} \right]$,
$C$ is a diagonal matrix with $C_{\ell\ell} = c_\ell$, and we denote by the superscript $\theta$ all quantities depending on the wave angle. Here, $\hat{\bm{\zeta}}_q^\theta \in \mathbb{C}^{N_b}$ is the vector of complex amplitudes of all bodies at the $q$-th frequency, while $\hat{\bm{\zeta}}^\theta \in \mathbb{C}^{N_bN_f}$ is the collection of all $\hat{\bm{\zeta}}_q^\theta$ for $q=1, \dots, N_f$.
For a single wave angle, the problem of power maximization can then be stated as 

\begin{equation}
	\begin{split}
		\min_{\hat{\bm{\zeta}}^\theta, \bm{u} \in \mathcal{U}_{ad}} J(\bm{u}, \hat{\bm{\zeta}}^\theta) \quad 
		\text{s.t.} \quad
		\begin{cases}
			Z_q(\mathbf{u}) \hat{\bm{\zeta}}_q^\theta = \hat{\bm{F}}_q^\theta & q=1,\dots, N_f,\\
			h_\ell(\hat{\bm{\zeta}}^\theta) = 0 & \ell=1, \dots, N_b. \\
		\end{cases} 
	\end{split}
 \label{eq:optprob}
\end{equation}
The admissible set of controls $\mathcal{U}_{ad}$ is defined by
\begin{equation*}
    \mathcal{U}_{ad} = \left\{ \bm{u} \in \mathbb{R}^{2N_b}: c_\ell\geq \varepsilon >0, s_\ell \geq \gamma , \quad \ell=1,\dots,N_b\right\},
\end{equation*}
where $\varepsilon$ is a positive constant, and $\gamma=0$ if the positive stiffness constraint is applied, $\gamma=-\infty$ otherwise.
Next, we discuss sufficient conditions for $Z_q(\bm{u})$ to be invertible. 
\begin{lemma}\label{lemma:Zinvertible}
Let $Z \in \mathbb{C}^{N\times N}$ and call $Z_r = \Re(Z)$, $Z_i = \Im(Z)$. If $Z_r$ is symmetric and $Z_i$ is symmetric positive (or negative) definite, then $Z$ is invertible.
\end{lemma}
\begin{proof}
Consider a generic system
$(Z_r + i Z_i)(\mathbf{x}_r + i \mathbf{x}_i) = \mathbf{f}_r + i \mathbf{f}_i$.
Splitting into real and imaginary parts and solving formally for $\mathbf{x}_r$ and $\mathbf{x}_i$ yields
\begin{equation}
\mathbf{x}_r = \widetilde{Z}^{-1} (Z_rZ_i^{-1} \mathbf{f}_r + \mathbf{f}_i), 
\quad 
\mathbf{x}_i = (I + Z_i^{-1}Z_r\widetilde{Z}^{-1}Z_r)Z^{-1} \mathbf{f}_r + Z_i^{-1}Z_r\widetilde{Z}^{-1} \mathbf{f}_i,
\label{eq:invcomplex}
\end{equation}
where $\widetilde{Z}$ is the Schur complement $\widetilde{Z} = Z_i + Z_r Z_i^{-1} Z_r$.
Suppose $Z_i$ is positive definite, so that its inverse is also positive definite (an analogous argument applies if $Z_i$ is negative definite). $Z_r$ is real and symmetric, thus $Z_r^H = Z_r$. Then $\widetilde{Z}$ is positive definite:
\begin{equation*}
\begin{split}
\mathbf{v}^H (Z_i + Z_r Z_i^{-1} Z_r) \mathbf{v}  &= \mathbf{v}^H Z_i \mathbf{v} + \left(\mathbf{v}^H Z_r^H \right) Z_i^{-1} \left(Z_r \mathbf{v}\right) 
= \mathbf{v}^H Z_i \mathbf{v} +  \left( Z_r \mathbf{v} \right)^H Z_i^{-1} \left( Z_r \mathbf{v} \right) 
> 0, \quad \forall\mathbf{v}\neq 0.
\end{split}
\end{equation*}
Since $\widetilde{Z}$ is invertible, the expressions in \eqref{eq:invcomplex} are well defined for any $\mathbf{f}$ and thus $Z$ is invertible.
\end{proof}

Consider now the impedance matrix defined in \eqref{eq:impmatdef}. Matrices $A$ and $B$ are symmetric, and because of energy conservation $B$ is also positive semidefinite \cite[Sec. 5.2, 6.5]{Falnes2020}. The other involved matrices are diagonal and thus symmetric. 
Then, the real part of the impedance matrix
    $Z_r = -\omega^2(M + A) + K + S$
is symmetric. Regarding the imaginary part
    $Z_i = -\omega(B + C)$, 
we have that $B$ is symmetric positive semidefinite. If we further assume that all bodies have positive generator damping, implying that the mean extracted power is positive, then $C$ is diagonal with positive values. Thus $Z_i$ is symmetric positive definite and $Z_r$ is symmetric, satisfying the hypotheses of Lemma~\ref{lemma:Zinvertible}.

Since the matrix $Z_q(\mathbf{u})$ is invertible for any admissible control vector $\mathbf{u}$, we can formally eliminate the constraint $Z_q(\mathbf{u})\hat{\bm{\zeta}}^\theta = \hat{\mathbf{F}}^\theta$ and rewrite \eqref{eq:optprob} in the reduced form
\begin{equation}
		\min_{\bm{u}\in \mathcal{U}_{ad}} \widetilde{J}(\bm{u}; \theta) := J(\bm{u}, \hat{\bm{\zeta}}^\theta(\bm{u})) \quad
		\text{s.t.}  \quad
        h_\ell(\hat{\bm{\zeta}}^\theta(\bm{u})) = 0 \quad \ell=1, \dots, N_b.
 \label{eq:redoptprob}
\end{equation}

The following properties hold for the case of a single wave direction $\theta$. 
\begin{lemma}\label{lemma:negativeJ}
    If there exists an index $\overline{q} \in 1, \dots, N_f$ such that $\hat{\bm{F}}^{\theta}_{\overline{q}} \neq 0$, then $\widetilde{J}(\bm{u}) < 0 \quad \forall \bm{u} \in \mathcal{U}_{ad}$.
    Otherwise, $\widetilde{J}(\bm{u}) = 0$ $\forall \bm{u} \in \mathcal{U}_{ad}$.
\end{lemma}
\begin{proof}
    Lemma \ref{lemma:Zinvertible} ensures that $Z_{\overline{q}}(\bm{u})$ is invertible for all $\bm{u}\in \mathcal{U}_{ad}$. Then, $\hat{\bm{F}}^{\theta}_{\overline{q}} \neq 0$ implies $\hat{\bm{\zeta}}^{\theta}_{\overline{q}} \neq 0$. From this and using again the assumption that $\bm{u} \in \mathcal{U}_{ad}$, so that in particular $c_l>0$, it follows that $\left(\hat{\bm{\zeta}}^{\theta}_{\overline{q}}\right)^H C \hat{\bm{\zeta}}^{\theta}_{\overline{q}}>0$. From the definition of $J$ \eqref{eq:Jdeffreqdom} we have
    $\widetilde{J}(\bm{u}) \leq - \dfrac{1}{2} \omega_{\overline{q}}^2 \left( \hat{\bm{\zeta}}_q^\theta \right)^H C \hat{\bm{\zeta}}^{\theta}_{\overline{q}} < 0$.
\end{proof}

\begin{lemma}\label{lemma:asympt}
    The following asymptotics hold for $\|\bm{u}\|_2 \to \infty$ and $\forall q$, $\forall \theta$:
    $$ \|Z_q(\bm{u})\|_2 = \mathcal{O}\left( \|\bm{u}\|_2 \right), \quad \| \hat{\bm{\zeta}}^{\theta}_q \|_2 = \mathcal{O} \left( \dfrac{1}{\| \bm{u} \|_2} \right). $$
    In particular, 
    $\lim_{\|\bm{u}\|_2 \to \infty} \widetilde{J}(\bm{u}) = 0$.
\end{lemma}

\begin{proof}
    From the definition \eqref{eq:impmatdef} of the impedance matrix, one directly finds $\|Z_q(\bm{u})\|_F = \mathcal{O}(\| \bm{u} \|_2)$ for $\|\bm{u}\| \to \infty$. The equivalence between the Frobenius norm and the 2-norm implies $\|Z_q(\bm{u})\|_2 = \mathcal{O}(\| \bm{u} \|_2)$.

    For the second statement, consider
    $$ \|\hat{\bm{\zeta}}^{\theta}_q (\bm{u}) \|_2 = \|Z_q(\bm{u})^{-1} \bm{F}_q \|_2 \leq \|Z_q(\bm{u})^{-1}\|_2 \|\bm{F}^{\theta}_q\|_2. $$
    Since $\bm{F}_q$ does not depend on $\bm{u}$, we have
    $ \|\hat{\bm{\zeta}}^{\theta}_q (\bm{u}) \|_2 = \mathcal{O} \left( \dfrac{1}{\|Z_q(\bm{u})^{-1}\|_2} \right) = \mathcal{O} \left( \dfrac{1}{\| \bm{u} \|_2} \right)$.
\end{proof}

\begin{lemma} \label{lemma:ball}
    There exists a constant $\widetilde{C} > 0$ such that
    $$ \min_{\bm{u} \in \mathcal{U}_{ad}} \widetilde{J}(\bm{u}) \Leftrightarrow \min_{\bm{u} \in \mathcal{U}_{ad}\cap \overline{\mathcal{B}_0(\widetilde{C})}} \widetilde{J}(\bm{u}),$$
    where $\mathcal{B}_0(\widetilde{C})$ is a ball of radius $\widetilde{C}$ centered in the origin.
\end{lemma}
\begin{proof}
    Lemmas \ref{lemma:negativeJ} and \ref{lemma:asympt} imply that
    $$ \forall M<0 \quad \exists \widetilde{C}>0 \quad \text{s.t.} \quad \widetilde{J}(\bm{u}) > M \; \forall \bm{u} \in \mathcal{U}_{ad} \cap \overline{\mathcal{B}_0\left(\widetilde{C}\right)}^{\mathcal{C}}.$$
    Consider a control vector $\widetilde{\bm{u}}$ such that $\widetilde{J}(\widetilde{\bm{u}}) = M <0.$ Then $\widetilde{J}(\bm{u})>\widetilde{J}(\widetilde{\bm{u}})$ $\forall \bm{u}\in \mathcal{U}_{ad} \cap \overline{\mathcal{B}_0\left(\widetilde{C}\right)}^\mathcal{C}$. In particular, if $\widetilde{J}$ has a minimum over such set, then the latter is also a minimum over $\mathcal{U}_{ad}$.
    
\end{proof}

\begin{theo}\label{theo:existunconstr}
    Under the assumptions of Lemma \ref{lemma:negativeJ}, Problem 
    $$
    \min_{\bm{u}\in \mathcal{U}_{ad}} \widetilde{J}(\bm{u}; \theta) 
    $$
    admits a solution for any $\theta$.
\end{theo}
\begin{proof}
    Invertibility of the impedance matrix $Z_q(\bm{u})$ guarantees that the reduced cost $\widetilde{J}$ is a continuous function of $\bm{u}$. Lemma \ref{lemma:ball} implies that the problem can be equivalently posed on the compact set $\mathcal{U}_{ad} \cap \overline{\mathcal{B}_0\left(\widetilde{C}\right)}$. The existence of a minimum in such set is guaranteed by the Weierstrass theorem.
\end{proof}

\begin{theo} \label{theo:existconstr}
    Under the assumptions of Lemma \ref{lemma:negativeJ}, Problem \eqref{eq:redoptprob} admits a solution for any $\theta$.
\end{theo}
\begin{proof}
    Define $h_\ell(\bm{u}) := h_\ell(\hat{\bm{\zeta}^{\theta}}(\bm{u}))$, $h_\ell : \mathbb{R}^{2N_b} \to \mathbb{R}_+$.
    We equivalently recast \eqref{eq:redoptprob}  as 
    $\min_{\bm{u}\in \mathcal{U'}_{ad}} \widetilde{J}(\bm{u}; \theta)$,
    where 
    $
    \mathcal{U'}_{ad} = \mathcal{U}_{ad} \cap \overline{\mathcal{B}_0\left(\widetilde{C}\right)} \cap\left( \cap_{\ell=1}^{N_b} h_\ell^{-1}(0)\right)
    $.
    Since $\{0\}$ is a compact set and $h_\ell$ is continuous for all $\ell$, $h_\ell^{-1}(0)$ is compact. Then, $\mathcal{U'}_{ad}$ is also compact and the same reasonings of Theorem \ref{theo:existunconstr} can be applied.
\end{proof}

In the stochastic case, \eqref{eq:redoptprob} turns into

\begin{equation}
\begin{split}
		\min_{\bm{u}\in \mathcal{U}_{ad}} \widetilde{j}(\mathbf{u}) := \mathbb{E}_\theta\left(\widetilde{J}(\bm{u};\theta)\right) \\
		\text{s.t.} \quad
        \mathbb{E}_\theta\left( h_\ell \right) = 0, \quad \ell=1, \dots, N_b,	
	\label{eq:redstochJ}
 \end{split}
\end{equation}
where $\mathbb{E}_\theta$ denotes the expected value with respect to $\theta$.

\begin{theo}
    Problem 
    \begin{equation*}
    \begin{split}
		\min_{\bm{u}\in \mathcal{U}_{ad} \cap \overline{\mathcal{B}_0\left(\widetilde{C}\right)}} \widetilde{j}(\mathbf{u} ) := \mathbb{E}_\theta\left(\widetilde{J}(\bm{u};\theta)\right) \\
		\text{s.t.} \quad
        \mathbb{E}_\theta\left( h_\ell \right) = 0, \quad \ell=1, \dots, N_b	
    \end{split}
\end{equation*}
admits a solution for all $\widetilde{C}>0$.
\end{theo}
\begin{proof}
    Lemma \ref{lemma:negativeJ} guarantees that $\widetilde{J}(\bm{u}, \theta)\leq 0 \, \forall \theta, \, \forall \bm{u} \in \mathcal{U}_{ad}$. Then, we have 
    $$\widetilde{j}(\bm{u}) = \mathbb{E_\theta}\left(\widetilde{J}(\bm{u}, \theta)\right) \leq 0 \quad \forall \bm{u} \in \mathcal{U}_{ad}.$$ 
    Furthermore, we have 
    $$
    \lim_{\|\bm{u}\|_2 \to \infty} \left| \widetilde{j}(\bm{u}) \right| = \lim_{\|\bm{u}\|_2 \to \infty} \left| \int_0^{2\pi} \widetilde{J}(\bm{u}, \theta) D(\theta) \, \mathrm{d}\theta \right| \leq  \lim_{\|\bm{u}\|_2 \to \infty} \left| \max_\theta \widetilde{J}(\bm{u}, \theta) \right| \left| \int_0^{2\pi} D(\theta) \, \mathrm{d}\theta \right| = 0, 
    $$
    where the last equality follows from Lemma \ref{lemma:asympt}. Then, the problem can be equivalently recast in a compact set $\mathcal{U}_{ad} \cap \overline{\mathcal{B}_0\left(\widetilde{C}\right)}$, for some suitable positive constant $\widetilde{C}$. Since such set is compact and $\widetilde{J}$ is a continuous function of both its arguments, the Heine-Cantor theorem implies $\widetilde{J}$ is also uniformly continuous.
    We now use uniform continuity and the fact that $D(\theta)$ is a probability distribution, so that it has unit integral. We have
    $$
    \forall \varepsilon>0 \quad \exists \delta >0 \quad\text{s.t.} \quad \| \bm{u}  - \bm{v} \|<\delta \Rightarrow \left|\widetilde{J}(\bm{u}, \theta) - \widetilde{J}(\bm{v}, \theta)\right| < \varepsilon,
    $$
    where $\delta$ does not depend on $\theta$.  Then, if $\|\bm{u} - \bm{v}\|<\delta$, the following inequality 
    \begin{align*}
        \left| \widetilde{j}(\bm{u}) - \widetilde{j}(\bm{v}) \right| &=
        \left| \int_0^{2\pi} \widetilde{J}(\bm{u}, \theta) D(\theta) \, \mathrm{d}\theta -  \int_0^{2\pi} \widetilde{J}(\bm{v}, \theta) D(\theta) \, \mathrm{d}\theta \right| \\&= \left| \int_0^{2\pi} \left[\widetilde{J}(\bm{u}, \theta) -    \widetilde{J}(\bm{v}, \theta) \right] D(\theta) \, \mathrm{d}\theta \right| \\
        &  \leq \max_{\theta} \left| \widetilde{J}(\bm{u}, \theta) - \widetilde{J}(\bm{v}, \theta)  \right| \left| \int_0^{2\pi} D(\theta) \, \mathrm{d}\theta \right| < \varepsilon
    \end{align*}
    proves that $\widetilde{j}(\bm{u})$ is also continuous.
    Then, existence of a solution for the unconstrained problem follows from the same reasonings applied in the proofs of Lemma \ref{lemma:ball} and Theorem \ref{theo:existunconstr}.

    For the constrained problem, we can prove, as done above for $\widetilde{j}$, that $\mathbb{E}_\theta[h_\ell(\bm{u})]$ is continuous, so its inverse maps $\{0\}$ into a compact set. Then, the remainder of the proof is the same as for Theorem \ref{theo:existconstr}. 

\end{proof}

\subsection{Numerical optimization strategy}

In this section, we present a computational framework for the solution of the stochastic optimization problem \eqref{eq:optprob} modelled in Sections \ref{sec:modeling} and \ref{sec:optmethods}.
Problem \eqref{eq:optprob} has three different constraints that must be appropriately treated: a system of complex algebraic equations (the state problem), the slamming constraint (a constraint on the state) and the positive stiffness constraint (a constraint on the controls).
First, the slamming (state) constraint is treated by using a penalty method, as described in Section \ref{subsec:penalty}.
Our framework is based on a reduced approach working in the space of solutions to the state equation and considers the control $\bm{u}$ as the only optimization variable. 
To do so, the gradient of the (reduced) cost $\widetilde{j}$ (and its penalized counterpart $Q_\mu$) is obtained by a Lagrangian approach leading to an adjoint problem, as show in Section~\ref{sec:lagrangian}.
Each iteration of the penalty method requires the solution of a stochastic problem with control constraint (positive stiffness constraint). 
This is achieved by either an SAA approach (Section~\ref{sec:SAA}), or by a Robust SA (Section~\ref{sec:SA}).
In both cases, the control constraint is enforced by projection onto the admissible set at each iteration. Implementation details and a comparison between SAA and Robust SA are given in Section \ref{subsec:comp}.

\subsubsection{Penalty method}\label{subsec:penalty}
The penalized cost function used to enforce the slamming constraint is defined as
\begin{equation*}
		Q_\mu(\bm{u}) := \widetilde{j}(\mathbf{u}) + \dfrac{\mu}{2} \sum_{\ell=1}^{N_b} \mathbb{E}_\theta[h_\ell(\mathbf{u})]
\end{equation*}
and the penalty method consists in solving the sequence of problems
\begin{equation*}
    \min_{\bm{u} \in \mathcal{U}_{ad}} Q_{\mu}(\bm{u}), \quad \mu \to \infty. 
\end{equation*}
Each of these problems is solved using either sample average approximation or robust stochastic approximation; these methods are described in Sections~\ref{sec:SAA} and \ref{sec:SA}, respectively.
The penalized cost function can be rewritten as an expected value thanks to linearity,
\begin{equation*}
Q_\mu(\bm{u}) = \mathbb{E}_\theta \left[ \widetilde{J}(\mathbf{u}) + \dfrac{\mu}{2} \sum_{\ell=1}^{N_b} h_\ell(\mathbf{u}) \right],
\end{equation*}
and, again because of linearity, also the gradient of $Q_\mu$ may be computed as the expected value of the quantity in brackets.
The optimization procedure is detailed in Algorithm \ref{alg:penalty}.

\begin{algorithm}[h!]
\caption{Stochastic quadratic penalty algorithm}
\label{alg:penalty}
\begin{algorithmic}[1]
\STATE Set $\bm{u}_0$, $\mu_0$, $k_\mu > 1$, $\tau^{out}$, $0<\tau^{out}<1$, $it_{max}^{out}$
\STATE Set $j=0$
\WHILE{$\max_\ell \mathbb{E}_\theta[h_\ell] > \tau^{out}$ \AND $j<it_{max}^{out}$}
\STATE Compute $\bm{u}^{j+1} \in \mathcal{U}_{ad}$ by minimizing $Q^{j+1}$, using either SA or SAA with initial guess $\bm{u}^j$ and tolerance $\tau^j$
\STATE $\mu^{j+1} = k_\mu \mu^j$
\STATE $\tau_{in}^{j+1} = k_\tau \tau^j_{in}$
\STATE $j \leftarrow j+1$
\ENDWHILE
\end{algorithmic}
\end{algorithm}

Convergence results for Algorithm \ref{alg:penalty} can be obtained from  \cite[Th. 17.1, 17.2]{Nocedal2006}. It is proved that, if at each iteration the exact minimum of $Q$ is found and if $\mu_j \to \infty$, then every limit point of the sequence $\bm{u}_j$ is a global solution of the original problem. It is further proved that if  at each iteration the solution of the optimization problem is computed approximately with vanishing tolerance $k_\tau$, as in Algorithm \ref{alg:penalty}, either a limit point of the sequence $\bm{u}_j$ is infeasible and corresponds to a stationary point of the penalty term, or it is a KKT point for the original problem. The above results require differentiability of the penalty term, which is satisfied in our case, as
$\nabla h_\ell = \nabla [g_\ell]^2_+ = 2 [g_\ell]_+ \nabla g_\ell$.
Since convergence is only guaranteed in the limit, the method is terminated when constraint violations are below a finite tolerance, that is when the obtained iterate is in, or close enough to, the feasible set.


\subsubsection{Gradient computation and first-order optimality condition} \label{sec:lagrangian}
Algorithm \ref{alg:penalty} requires finding the minimum of the penalized cost function $Q_\mu$. In order to do this, the computation of its gradient is needed. The SA method needs a stochastic gradient at each iterate, that is, a gradient computed for a single realization of the wave angle $\theta$. The SAA method, instead, needs an estimate of the true gradient, which is the expected value of the stochastic gradient and thus can be computed as a suitable average of stochastic gradients.
For a single given angle $\theta$, the gradient of the reduced cost can be obtained by using the Lagrangian function defined as
\begin{equation*}
\mathcal{L}(\bm{u}, \hat{\bm{\zeta}}^\theta, \bm{y}) = J(\bm{u}, \hat{\bm{\zeta}}^\theta) + \dfrac{\mu}{2}\sum_{\ell=1}^{N_b}[g_\ell(\mathbf{u})]_+^2 + \sum_{q=1}^{N_f} \Re \left[ \bm{y}_q^H \left( Z_q(\bm{u}) \hat{\bm{\zeta}}_q^\theta - \hat{\bm{F}}_q^\theta \right) \right].
\end{equation*}
The problem is set in the space of complex vectors over the field of real numbers, with scalar product $(u, v) = \Re[u^H v]$, thus the presence of the real part in the Lagrangian.
The adjoint equation is obtained by differentiation with respect to the state vector,
\begin{equation*}
Z_q^H(\bm{u}) \bm{y}_q = \omega^2_q C \hat{\bm{\zeta}}_q^\theta - 2 \mu V(\hat{\bm{\zeta}}_q^\theta - \bm{\eta}_q), \quad q=1,\dots, N_f,
\end{equation*}
where $V$ is a diagonal matrix
\begin{equation*}
V_{\ell\ell} = [g_\ell(\hat{\bm{\zeta}})]_+, \quad i=1,\dots,N_b,
\end{equation*}
that is independent of the frequency.
The components of the stochastic gradient are obtained by differentiating with respect to the controls:
\begin{equation*}
\dfrac{\partial \mathcal{L}}{\partial c_\ell}(\bm{u}, \hat{\bm{\zeta}}^\theta, \bm{y}) = - \sum_{q=1}^{N_f} \left( \dfrac{1}{2} \omega_q^2 | \hat{\zeta}_{\ell, q} |^2 + \Re[j \omega_q y_{\ell,q}^* \hat{\zeta}_{\ell,q}] \right),
\end{equation*}
\begin{equation*}
\dfrac{\partial \mathcal{L}}{\partial s_\ell}(\bm{u}, \hat{\bm{\zeta}}^\theta, \bm{y}) = 
\sum_{q=1}^{N_f} \Re[y_{\ell,q}^* \hat{\zeta}_{\ell,q}],
\end{equation*}
\begin{equation}
G(\bm{u}, \theta) = \dfrac{\partial \mathcal{L}}{\partial \bm{u}}(\bm{u}, \hat{\bm{\zeta}}^\theta, \bm{y}) = \left[ \dfrac{\partial \mathcal{L}}{\partial c_1}, \dots, \dfrac{\partial \mathcal{L}}{\partial c_{N_b}}, \dfrac{\partial \mathcal{L}}{\partial s_1}, \dots, \dfrac{\partial \mathcal{L}}{\partial s_{N_b}} \right].
\label{eq:stochgrad}
\end{equation}
The first-order optimality condition requires \cite{Nocedal2006} 
\begin{equation*}
    \bm{u} - \mathcal{P}_{\mathcal{U}_{ad}}\left(\bm{u} - \mathbb{E}_\theta[G(\bm{u}, \theta)] \right) = 0,
\end{equation*}
where $\mathcal{P}_{\mathcal{U}_{ad}}$ is the projector onto $\mathcal{U}_{ad}$:
\begin{equation}
\begin{split}
\mathcal{P}(\bm{u}) &= \mathcal{P}[c_1, \dots, c_{N_b}, s_1, \dots, s_{N_b}] \\
& = [c_1, \dots, c_{N_b}, [s_1]_+, \dots, [s_{N_b}]_+].
\end{split}
\label{eq:projdef}
\end{equation}


\subsubsection{SAA}\label{sec:SAA}
The sample average approximation (SAA) method is based on writing the expected value in Problem~\eqref{eq:redstochJ} as an integral, and approximating such integral with a suitable quadrature rule:
\begin{equation*}
\mathbb{E}_\theta \left[ \widetilde{J}(\bm{u}; \theta) \right] = \int_0^{2\pi} \widetilde{J}(\bm{u}; \theta) \, D(\theta) \, \mathrm{d\theta} \approx \sum_{i=1}^N \widetilde{J}(\bm{u}; \theta_i) \, w_i.
\end{equation*}
The same approximation is applied to the gradient. One then obtains a deterministic problem that can be solved using standard optimization algorithms. In this work, we use the gradient descent method with step length determined according to the Armijo rule \cite[Chap. 3]{Nocedal2006}.
The quadrature rule is defined by the choice of nodes $\theta_i$ and weights $w_i$. In the following, we consider the Monte Carlo \cite{Shapiro2014} and Gauss-Legendre \cite{Quarteroni2006,Trefethen2019} methods. 

For the Monte Carlo method, the nodes $\theta_i$ are obtained by sampling  the wave direction, following the procedure described in Section \ref{sec:stochforce}, and the choice of weights corresponds to averaging, $w_i = 1/N$. In the case of Gauss-Legendre, instead, nodes are defined as the zeros of Legendre polynomials. Numerical libraries provide nodes and weights $\widehat{\theta}_i$, $\widehat{w}_i$ on the interval $[-1,1]$, which are then mapped to the desired interval $[a,b]$ using
\begin{equation*}
    \theta_i = \dfrac{a+b}{2} +  \dfrac{b-a}{2}\widehat{\theta}_i, \quad w_i = \dfrac{b-a}{2} D(\theta_i) \widehat{w_i}.
\end{equation*}
In our case, the interval $[a, b]$ is chosen as $[\theta_{min}, \theta_{max}]$.
Here, $\theta_{min}$ and $\theta_{max}$ define an effective interval of wave angles,  which depends on the parameters $\beta$, $\theta_0$, such that the probability is considered negligible outside such interval.
Notice that each weight for Gauss-Legendre integration includes the value of the probability distribution function in the corresponding node. This is not the case for Monte Carlo, where the probability distribution is accounted for in the sampling phase.

\begin{figure}[h!]
\centering
\subfloat{\includegraphics[width=0.45\textwidth]{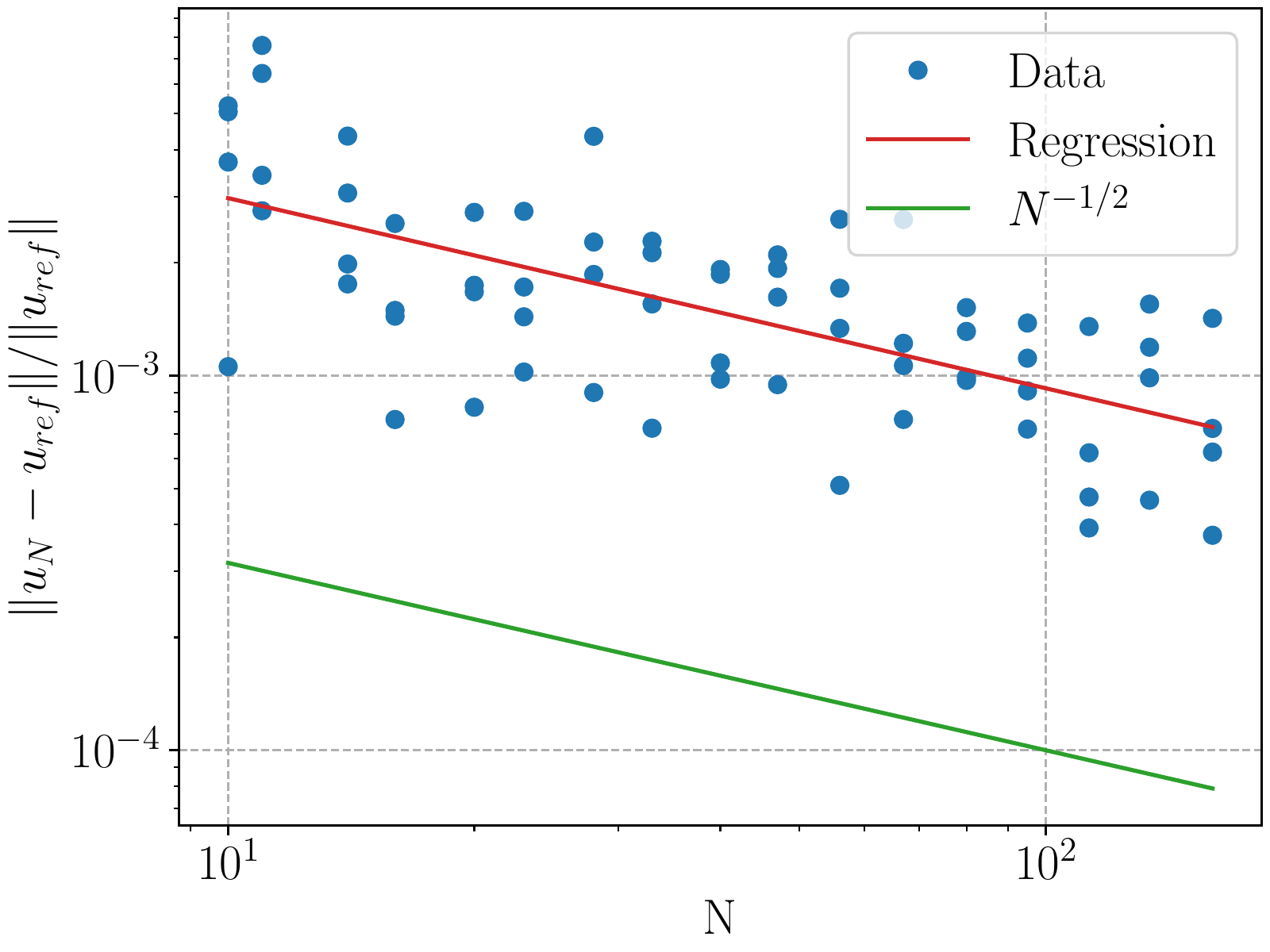}}
\hfill
\subfloat{\includegraphics[width=0.45\textwidth]{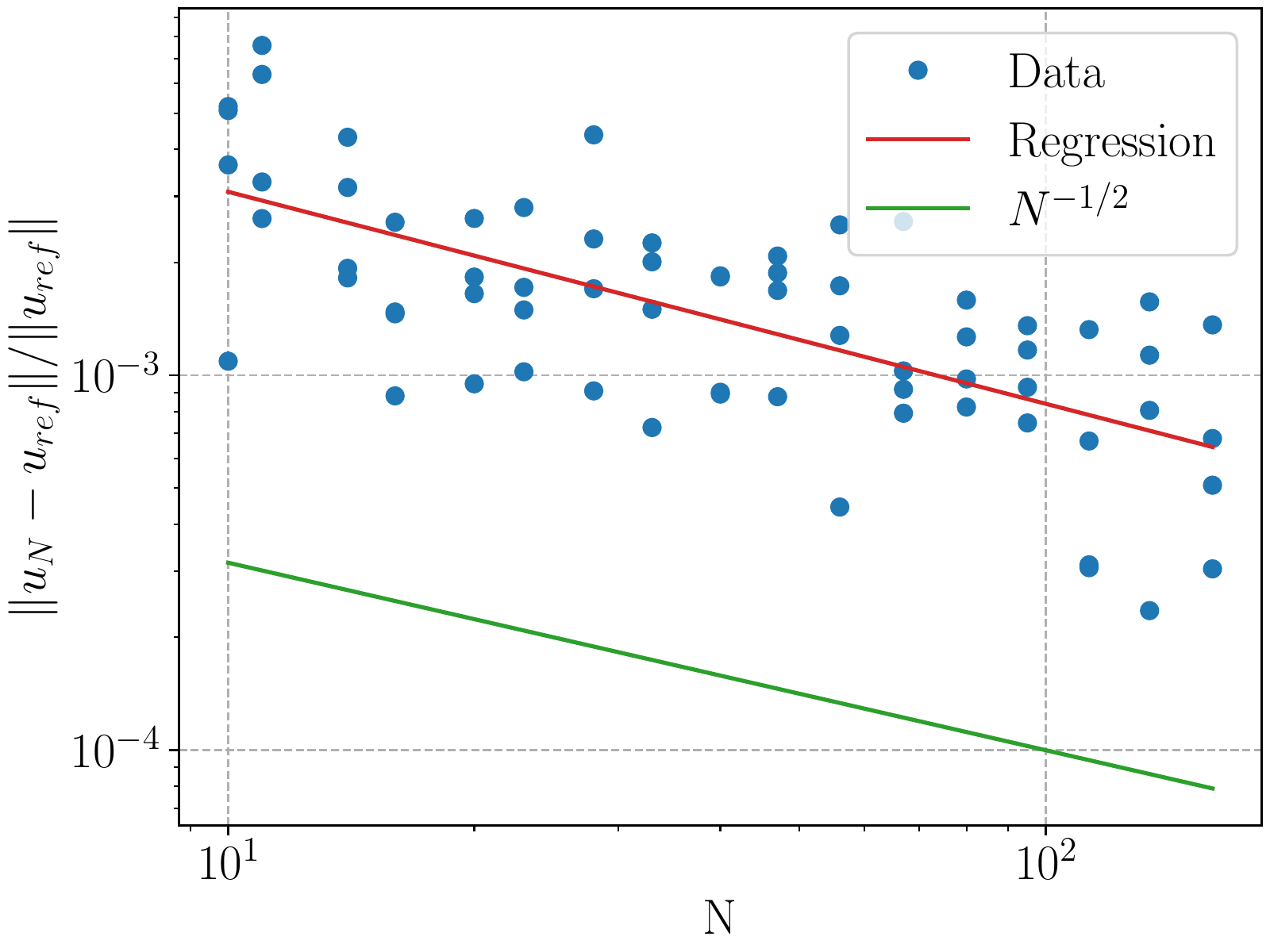}}
\caption{Relative error in the constrained case of the Monte Carlo method with respect to a solution obtained with Gauss-Legendre using $10$ (left) and $320$ (right) nodes}
\label{fig:MCconv}
\end{figure}


\begin{figure}[h!]
\centering
\subfloat{
\includegraphics[width=0.45\textwidth]{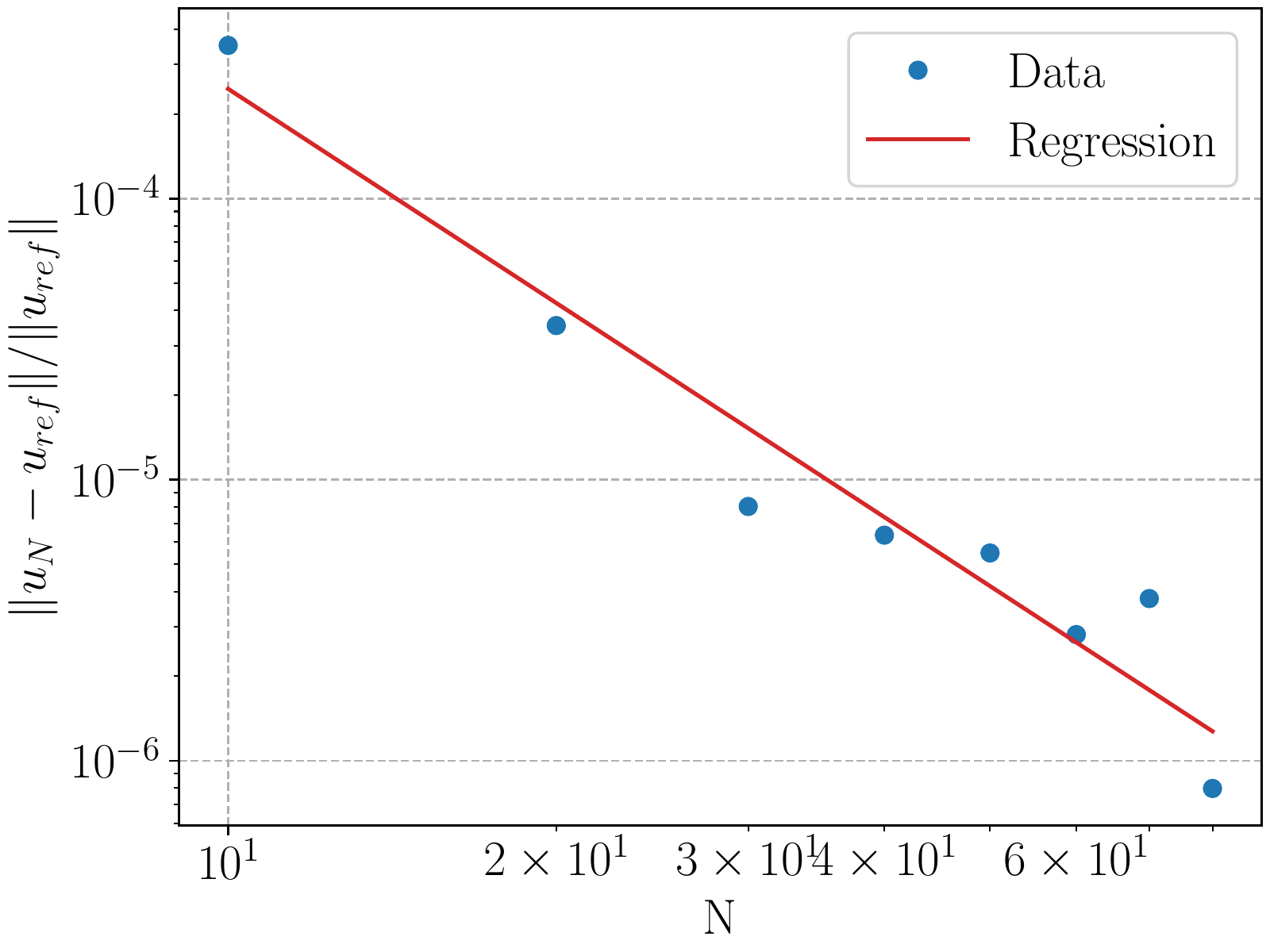}}
\hfill
\subfloat{
\includegraphics[width=0.45\textwidth]{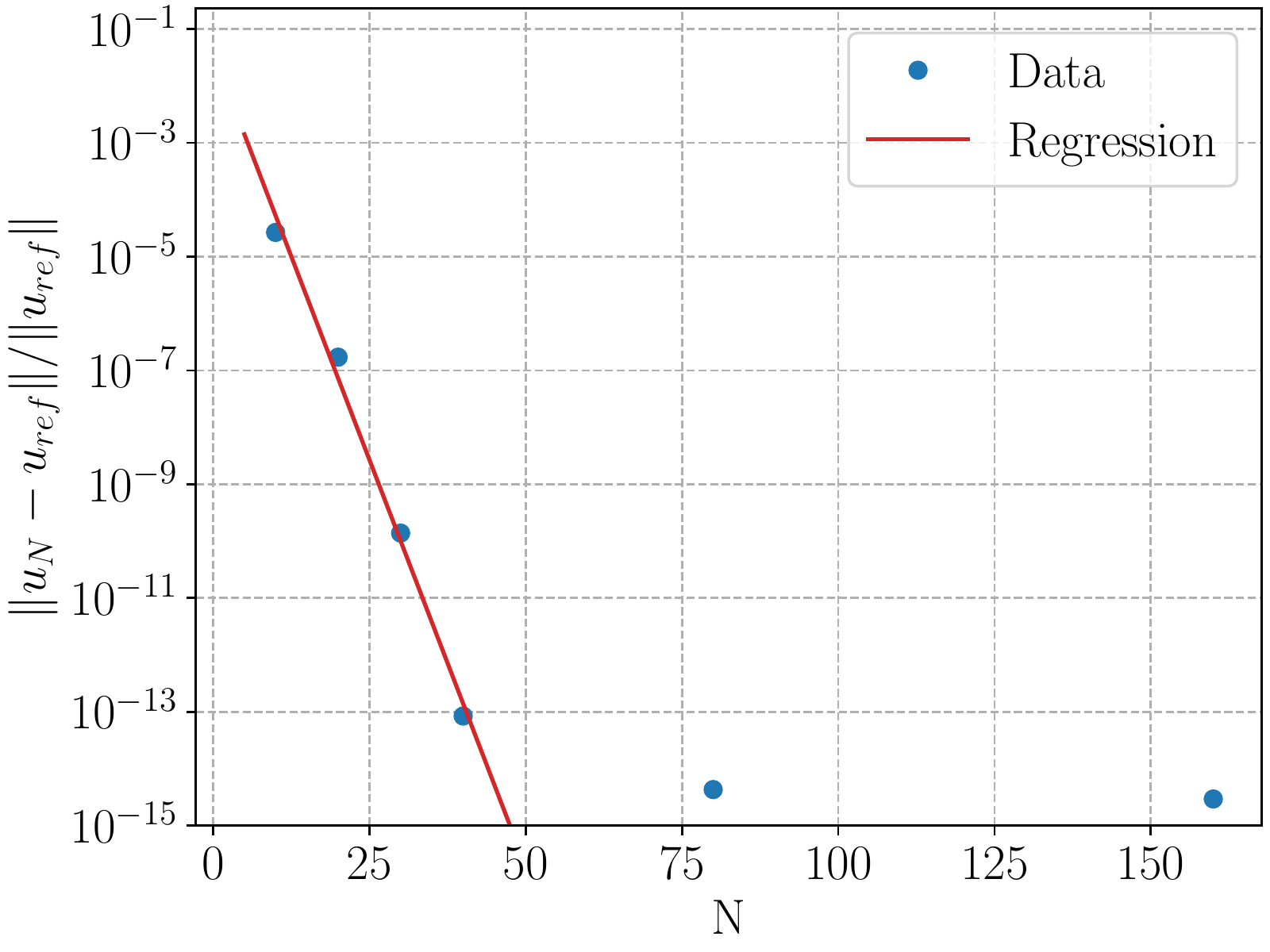}}
\caption{Left: Relative error in the constrained case of Gauss-Legendre method with respect to a reference solution computed with 320 nodes. Right: Relative error in the unconstrained case of the Gauss-Legendre method with respect to a reference solution computed with 320 nodes.}
\label{fig:GLconv}
\end{figure}

Convergence of the two methods is studied on a test case with 15 bodies, enforcing the slamming constraint. The results are reported in Figure \ref{fig:MCconv} for the Monte Carlo method and in Figure \ref{fig:GLconv} (left) for the Gauss-Legendre method. Since the exact solution of the optimization problem is not known, two reference solutions are computed using Gauss-Legendre, one with 10 nodes, and another with 320 nodes. The convergence rate of Monte Carlo is estimated by computing the slope of the regression line in logarithmic scale. For Monte Carlo, multiple points are available for each value of $N$, obtained by repeating the computation with different sets of samples. Using the reference solution computed with 10 nodes yields a slope of -0.507 with standard error 0.070, while using the reference solution computed with 320 nodes yields a slope of -0.565 with standard error 0.074. Both values are consistent with the expected convergence rate $-1/2$. This result shows that a Gauss-Legendre solution obtained with a very limited sample size is very accurate compared to a Monte Carlo solution with larger sample size. 

The order of convergence of Gauss-Legendre, instead, is estimated at a value of 2.53, while repeating the test for the unconstrained problem leads to a spectral convergence of the Gauss-Legendre method, as shown in Fig. \ref{fig:GLconv} (right). 
In particular, we obtain a behaviour of the type $k \exp(-mN)$, where $m \approx 0.658$. The loss of spectral convergence in the constrained case is explained by the presence of the max operator in the slamming constraint; such operator is present in the right hand side of the adjoint equation. 


\subsubsection{Robust SA}\label{sec:SA}
The stochastic approximation method (SA) moves, at each iteration, in the direction of a gradient computed from a single sample of the random variable. We call this vector the stochastic gradient. Since there is no guarantee that the stochastic gradient defines a descent direction, the line-search stepsize strategies suitable for deterministic problems cannot be applied. Instead, the sequence of stepsizes is fixed once the initial step is chosen.
A suitable initial step $t_0$ can be estimated by performing a line search using a single sample of the stochastic gradient.

Some stepsize strategies for SA method are discussed in \cite{Nemirovski2009}.  We choose the robust variable stepsize strategy, that is, the sequence of steps is
\begin{equation}
t^k = \dfrac{t^0}{\sqrt{k+1}}, \quad  k=0,1, \dots
\label{eq:stepsize}
\end{equation}
A convergence estimate is available, based on the assumptions that 
\begin{equation*}
\exists M>0 \quad \text{s.t.} \quad \mathbb{E}_\theta [\| G(\bm{u}, \theta) \|_2^2] \leq M^2 \quad \forall \bm{u} \in \mathcal{U}_{ad}, 
\end{equation*}
that $Q$ is convex and that the admissible set $\mathcal{U}_{ad}$ is convex
\begin{equation*}
\mathbb{E}\left[ Q\left(\tilde{\bm{u}}_K^N\right) - Q(\bm{u}_*) \right] \leq \dfrac{c_Q}{\sqrt{N}},
\end{equation*}
where $N$ is the number of iterations, $c_Q$ is a constant independent on $N$ and $\tilde{\bm{u}}_K^N$ is a weighted average defined as
\begin{equation}
\tilde{\bm{u}}_i^j = \sum_{k=i}^j t^k \bm{u}^k.
\label{eq:averaging}
\end{equation} 
Such average is taken as the solution of the optimization problem.

As a stopping indicator, we use the norm of an average of the stochastic gradients on the most recent iterations: 
\begin{equation}
\text{ind} = \left\| \dfrac{1}{W} \sum_{k=N-W}^N G^k \right\| \leq \tau_{in},
\label{eq:stopind}
\end{equation}
where $W$ is the window size. 
This is a modification of a criterion proposed in \cite{Patel2021} and it can be interpreted as an estimate of the gradient of the corresponding deterministic problem. Such quantity appears in the convergence proof of the penalty method mentioned above, suggesting that our choice is reasonable.

\begin{algorithm}[H]
\caption{Robust Stochastic Approximation (variable stepsize)}
\begin{algorithmic}[1]
\STATE Set initial step $t$, window size $W$
\WHILE{$k\leq it_{max}^{out}$ or $ind>\tau_{in}$}
\STATE Extract $\theta$
\FOR{$q=1:N_f$}
\STATE Compute $\hat{\mathbf{F}}^\theta_q$
\STATE Find $\hat{\bm{\zeta}}_q$ by solving the state problem with $(\textbf{u}_k,\hat{\mathbf{F}}^\theta_q)$ 
\STATE Find $\mathbf{y}_q$ by solving the adjoint problem with $(\textbf{u}_k,\hat{\mathbf{F}}^\theta_q)$
\ENDFOR
\STATE Compute the stochastic gradient $G$ from \eqref{eq:stochgrad}
\STATE Update $
\textbf{u}^{k+1} = \mathcal{P}(\textbf{u}^{k} - t^k G) $, with $\mathcal{P}$ from \eqref{eq:projdef} and $t^k$ from \eqref{eq:stepsize} 
\STATE Compute $\tilde{\textbf{u}}^{k+1}_{k+1-W}$ from \eqref{eq:averaging}
\STATE Compute the stopping indicator $ind$ from \eqref{eq:stopind}
\ENDWHILE
\end{algorithmic}
\end{algorithm}

\subsubsection{Implementation and comparison between SAA and robust SA}\label{subsec:comp}
Hydrodynamic data is computed using the code Capytaine \cite{Ancellin2019}, version 1.3. The Haskind relation is applied to compute diffraction properties. Such routine and the numerical optimization codes are written in Python, using Numba for performance improvement. 

\begin{figure}
\centering
\subfloat{\includegraphics[width=0.45\textwidth]{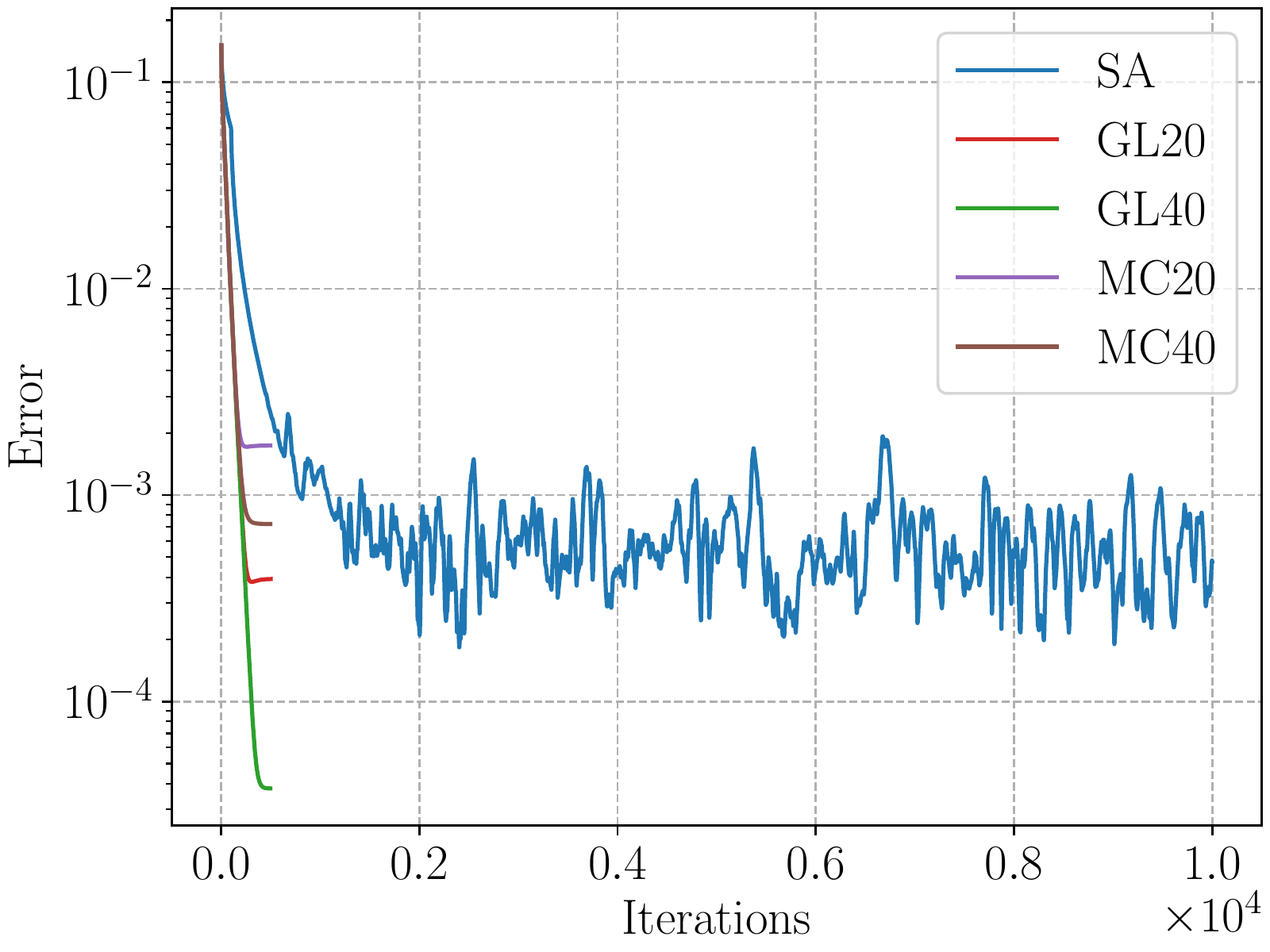}}
\hfill
\subfloat{\includegraphics[width=0.45\textwidth]{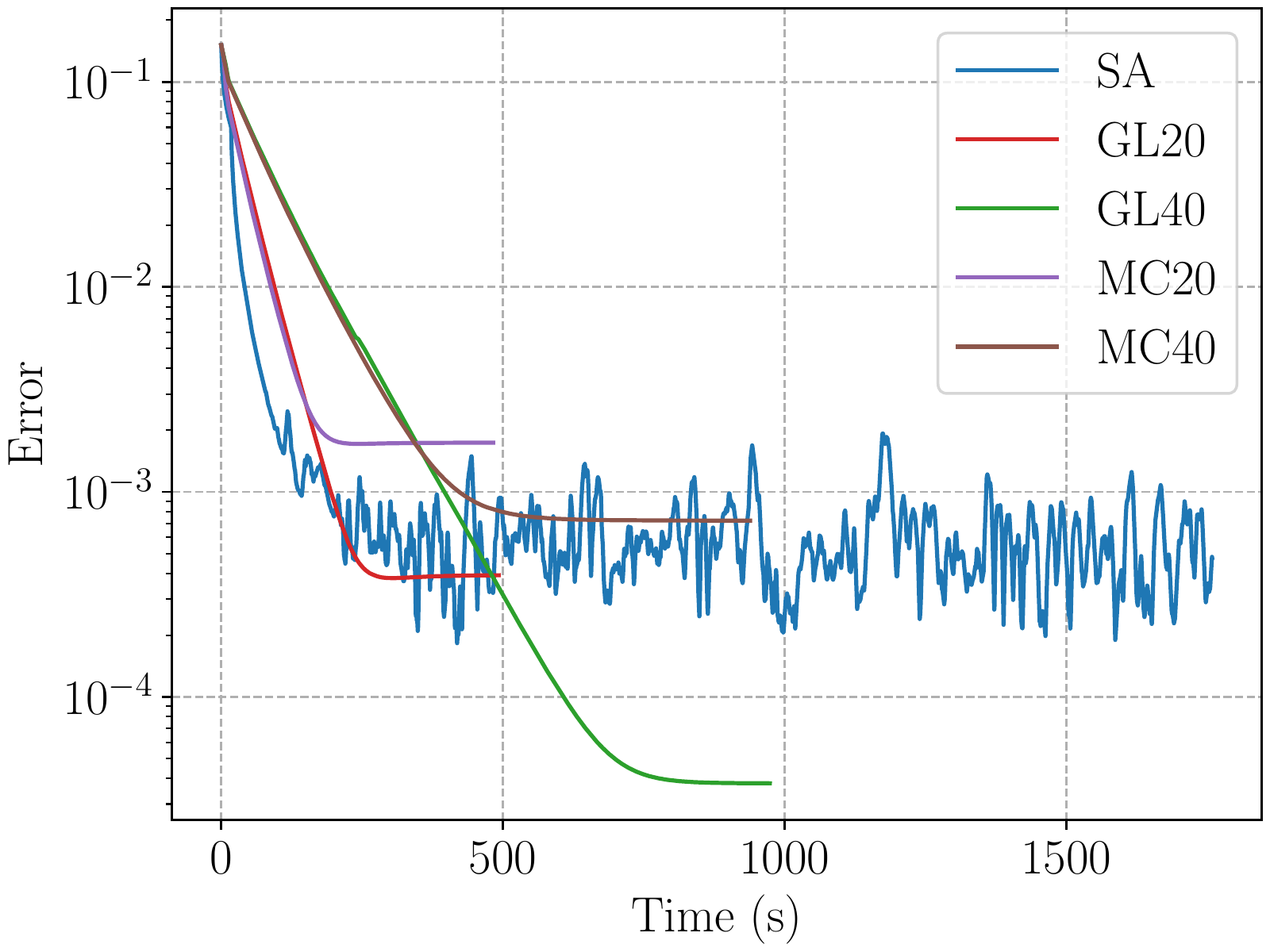}}
\caption{History of relative error on the optimal control with respect to iterations (left) and time (right). Numbers in the labels refer to the number of nodes (Gauss-Legendre) or samples (Monte Carlo)}
\label{fig:timeitersSAvsSAA}
\end{figure}
The performances of stochastic approximation and sample average approximation using Monte Carlo and Gauss-Legendre integration are compared in Fig.~\ref{fig:timeitersSAvsSAA}. A constrained case with 15 bodies and fixed penalization parameter has been considered, using the same initial guess and initial step.
The reference solution was obtained by a computation with Gauss-Legendre using 320 nodes. The relative error $\|\mathbf{u} - \mathbf{u}_{ref}\|/\|\mathbf{u}_{ref}\|$ is plot against execution time and number of iterations. The test is run on an average laptop. 
We can observe that the SAA methods require a much smaller number of iterations with respect to SA, but, since for the former many problems need to be solved at each iteration, the methods are comparable in terms of computational time. In terms of number of iterations, the histories of Monte Carlo and Gauss-Legendre are almost identical up to a certain value of the error, after which Gauss-Legendre exhibits better accuracy. However, in terms of computational times, the histories of the two methods are similar only when the same number of nodes is used.
The error of stochastic approximation, instead, initially decreases at a rate comparable to (in terms of number of iterations) or faster than (in terms of computational time) the other methods, and then starts oscillating, as expected from an SA approach (see, e.g., \cite{Nemirovski2009} and references therein).
Moreover, the results of Fig.~\ref{fig:timeitersSAvsSAA} provide useful insights into the choice of the method to use in a design phase. If one seeks indicative results in relatively short times, then an SA approach is indeed a reasonable choice: on the one hand, it converges faster in the initial phase when the starting point is far from the optimal solution; on the other hand, it does not require the choice of a number of nodes or of an effective interval of wave angles where the probability distribution is correctly represented. Instead, if one is interested in highly accurate results, then Gauss-Legendre is the method to choose. However, this strategy requires an appropriate choice of an effective interval of wave angles, which is not needed by a Monte Carlo approach. Moreover, if the number of stochastic parameters becomes very large, then the Gauss-Legendre method could become unfeasible, and  an SA strategy is more appropriate.

\section{Numerical experiments}\label{sec:numexp}

The numerical framework presented in the previous sections was used to compute the optimal control parameters of arrays of two different models of WECs (see Table \ref{tab:bodies}). In particular, the results presented in this section are obtained by simulations run using SA. The data for the first case is taken from \cite{Goeteman2014}. The second case is derived from the first by rescaling the generator mass proportionally with respect to the volume. In the first case, the resonance frequency of the uncontrolled body is higher than the peak spectrum frequency. The second case is tuned in such a way that the uncontrolled body frequency is lower than the peak spectrum frequency.

In both cases, the region of the spectrum containing $99.9\%$ of the power is discretized into 30 bins of equal power, so that each realization is obtained as a superposition of 30 harmonics. 

To quantify the performance of the obtained configurations, we define the interaction factor $q$ as the ratio between the total average power produced by the array and the sum of the powers of the isolated devices:
\begin{equation*}
    q = \dfrac{\sum_{\ell=1}^{N_b} P_\ell}{N_b \, P_\text{isolated}}.
\end{equation*}
\begin{table}[h!]
\centering
\begin{tabular}{@{}llllllll@{}}
\toprule
         & \multicolumn{1}{c}{\begin{tabular}[c]{@{}c@{}}radius \\ (m)\end{tabular}} & \multicolumn{1}{c}{\begin{tabular}[c]{@{}c@{}}draft \\ (m)\end{tabular}} & \multicolumn{1}{c}{\begin{tabular}[c]{@{}c@{}}generator \\ mass (kg)\end{tabular}} & \multicolumn{1}{c}{\begin{tabular}[c]{@{}c@{}}generator \\ stiffness (N/m)\end{tabular}} & \multicolumn{1}{c}{\begin{tabular}[c]{@{}c@{}}depth \\ (m)\end{tabular}} & \multicolumn{1}{c}{\begin{tabular}[c]{@{}c@{}}$H_s$ \\ (m)\end{tabular}} & \multicolumn{1}{c}{\begin{tabular}[c]{@{}c@{}}$T_p$ \\ (s)\end{tabular}} \\ \midrule
Case 1 & 2.5                                                                       & 0.5                                                                      & 2560                                                                               & 4000                                                                                     & 50                                                                       & 1.53                                                                     & 5.83                                                                     \\
Case 2 & 9                                                                         & 1.5                                                                      & 155520                                                                             & 4000                                                                                     & 50                                                                       & 1.53                                                                     & 3.49                                                                     \\ \bottomrule
\end{tabular}
\caption{Body parameters and wave climates}
\label{tab:bodies}
\end{table}

\subsection{Single-body optimization}
We first consider optimization for a single device. For case 1, if negative stiffnesses are allowed, the optimal power is 7636 W, achieved by setting the control parameters to $c = 31820.7 \text{ N}\cdot\text{s/m}$, $s = -27022.2 \text{ N/m}$. 
The outer convergence history is shown in Figure \ref{fig:1cconvmap} (left), superimposed on a contour map of the absorbed power. The region above the red dashed line is feasible with respect to the slamming constraint, and the black dot is the constrained optimum obtained by exhaustive search over a rectangular grid. If instead only positive values of stiffness are accepted, the optimal power is 5886 W, and the corresponding controls are $c = 61774.8 \text{ N}\cdot\text{s/m}$, $s = 0$. Figure \ref{fig:1cconvmap} (right) 
shows the statistical meaning of the slamming constraint, which has been formulated in Section \ref{sec:constraints} in the frequency domain. The fraction of peaks exceeding the bound can be controlled by varying $\alpha$ in \eqref{eq:slammingrmsconstraint}.

For case 2, the optimal power is 5712 W, obtained with controls $c = 733196 \text{ N}\cdot\text{s/m}$, $s = 1411344 \text{ N/m}$, without imposing constraints on the sign of $s$.

The results of this section are used as initial guesses for the multi-body optimization runs: all devices are initially assigned the same control parameters. 

\begin{figure}[t]
\centering
\subfloat{\includegraphics[scale=0.44]{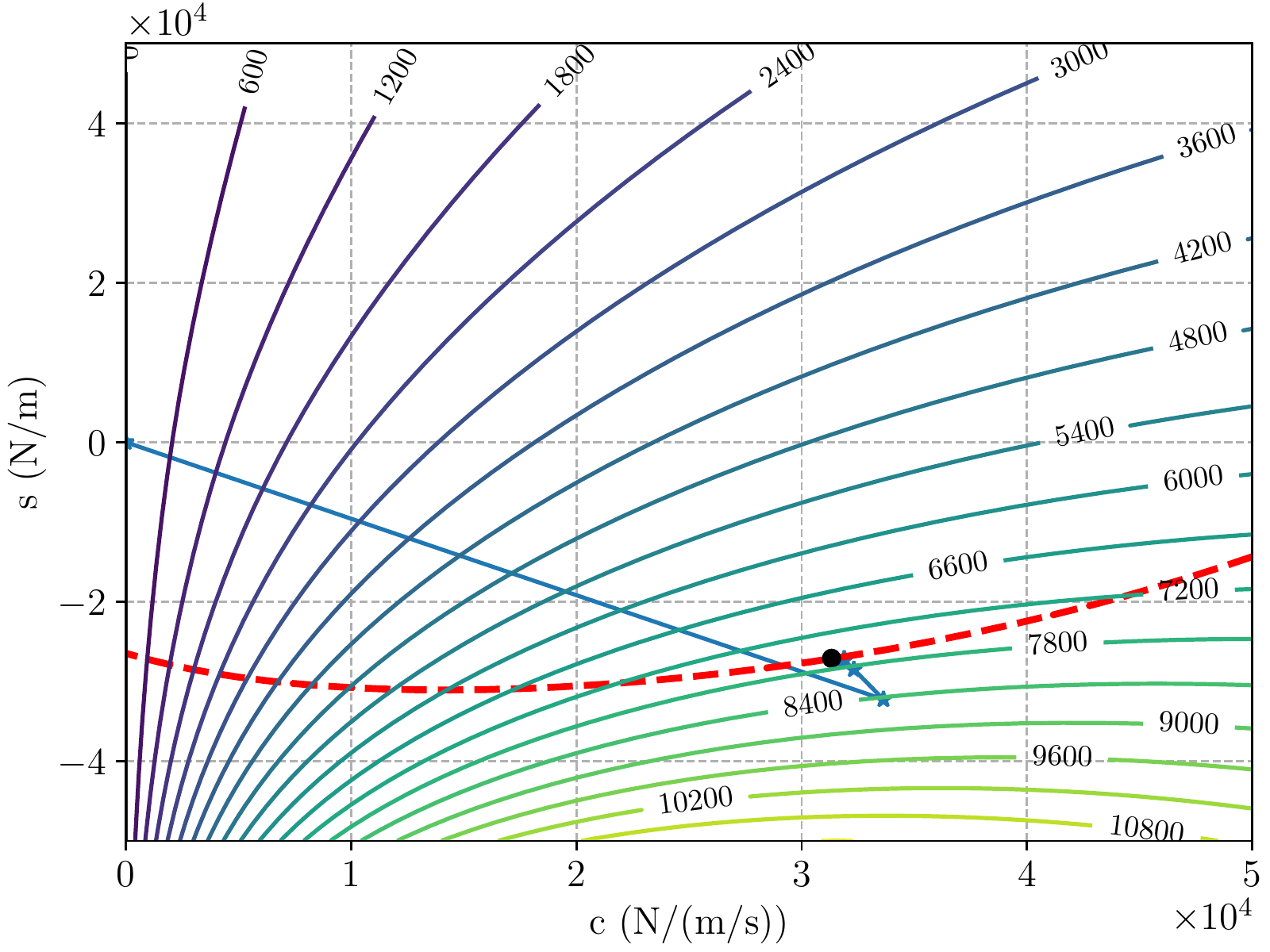}}
\hfill
\subfloat{\includegraphics[scale=0.44]{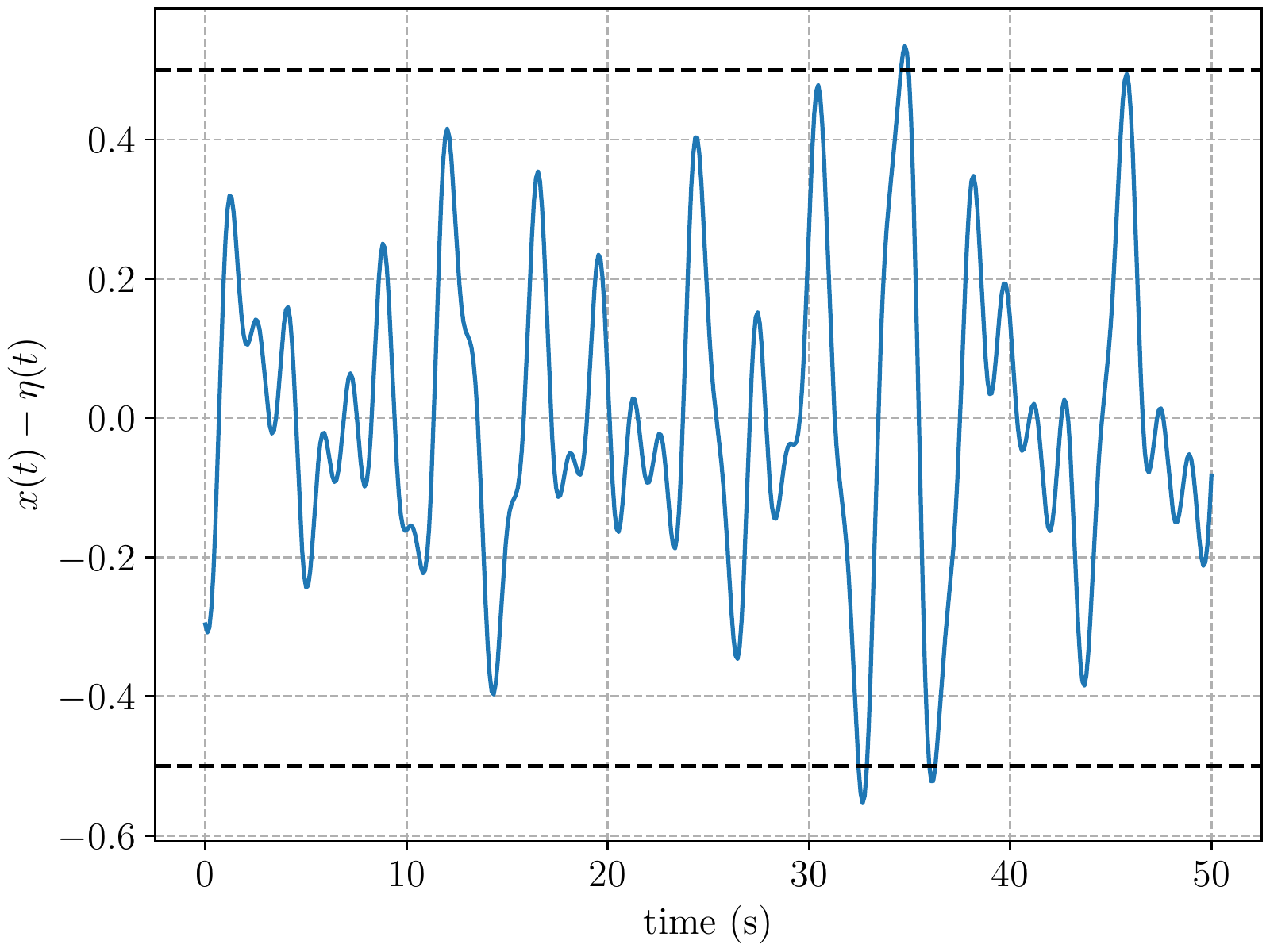}}
\caption{Left: Convergence plot (outer iterations) for the isolated body of case 1 superimposed on power contour map (values in W). Red dashed line: slamming constraint. Right: Time domain realization.}
\label{fig:1cconvmap}
\end{figure}


\subsection{Two-body optimization}

Table \ref{tab:2bodtest} shows the results of numerical experiments with 2 bodies of type 1. Configuration $\theta_0=0\degree$ corresponds to bodies aligned along the mean wave vector;  $\theta_0=90\degree$ corresponds to bodies aligned perpendicular to the mean wave vector. In both cases, the resulting interaction factor $q$ is larger when the spreading in wave direction is small, that is when $\beta$ is large (see  \eqref{eq:donelan}). In most cases, the slamming constraint is violated for the initial guess. This results in a reduction of power between the initial guess and the optimized result, the latter being compatible with the constraint up to the tolerance. 
The results for control parameters in the symmetric case $\theta_0=90\degree$ are not exactly symmetric: this is due to the use of a stochastic optimization method. The asymmetry may be reduced by increasing the window size used for averaging the iterates.

\begin{table}[h!]
\centering
\begin{tabular}{@{}cccccccc@{}}
\toprule
       & $\beta$ & $\theta_0$ & Initial violation & $\Delta P$ ($\%$) & $q$   & $c$ (N$\cdot$s/m)                                                         & $s$ (N/m)                                                           \\ \midrule
Test 1 & 20      & 0          & yes               & -2.65                & 0.949 & \begin{tabular}[c]{@{}c@{}}32095.1\\ 31140.3\end{tabular}   & \begin{tabular}[c]{@{}c@{}}-27384.7\\ -23636.9\end{tabular}   \\ \midrule
Test 2 & 5       & 0          & yes               & -2.59                & 0.947 & \begin{tabular}[c]{@{}c@{}}32156.5 \\ 31188.2\end{tabular}  & \begin{tabular}[c]{@{}c@{}}-27379.7  \\ -23701.0\end{tabular} \\ \midrule
Test 3 & 1       & 0          & yes               & -2.14                & 0.943 & \begin{tabular}[c]{@{}c@{}}32698.3\\ 32243.6\end{tabular}   & \begin{tabular}[c]{@{}c@{}}-27207.6 \\ -23555.7\end{tabular}  \\ \midrule
Test 4 & 20      & 90         & no                & +5.03                & 1.015 & \begin{tabular}[c]{@{}c@{}}33122.7\\ 33121.3\end{tabular}   & \begin{tabular}[c]{@{}c@{}}-29660.8\\ -29660.6\end{tabular}   \\ \midrule
Test 5 & 5       & 90         & no                & +4.59                & 1.009 & \begin{tabular}[c]{@{}c@{}}33306.1  \\ 33365.8\end{tabular} & \begin{tabular}[c]{@{}c@{}}-29258.7\\ -29280.5\end{tabular}   \\ \midrule
Test 6 & 1       & 90         & yes               & -2.27                & 0.940 & \begin{tabular}[c]{@{}c@{}}33353.7  \\ 33192.1\end{tabular} & \begin{tabular}[c]{@{}c@{}}-25061.1 \\ -24690.0\end{tabular}  \\ \bottomrule
\end{tabular}
\caption{Tests for 2 bodies, case 1}
\label{tab:2bodtest}
\end{table}

\begin{table}[h!]
\centering
\begin{tabular}{@{}cccccccc@{}}
\toprule
       & $\beta$ & $\theta_0$ & Initial violation & $\Delta P$ ($\%$) & $q$   & $c$ (N$\cdot$s/m)                                                         & $s$ (N/m)                                                           \\ \midrule
Test 1 & 20      & 0          & no               & +2.41                & 0.675 & \begin{tabular}[c]{@{}c@{}}689151 \\  728208\end{tabular}   & \begin{tabular}[c]{@{}c@{}}1104329 \\ 604768\end{tabular}   \\ \midrule
Test 2 & 5       & 0          & no               & +2.63                & 0.675 & \begin{tabular}[c]{@{}c@{}} 671380 \\  716425 \end{tabular}  & \begin{tabular}[c]{@{}c@{}}1124916 \\  559310\end{tabular} \\ \midrule
Test 3 & 1       & 0          & no               & +2.20                & 0.893 & \begin{tabular}[c]{@{}c@{}}659638 \\ 783297\end{tabular}   & \begin{tabular}[c]{@{}c@{}}1490803 \\ 1196127\end{tabular}  \\ \midrule
Test 4 & 20      & 90         & no                & +3.34               & 0.870 & \begin{tabular}[c]{@{}c@{}}878379 \\ 876926 \end{tabular}   & \begin{tabular}[c]{@{}c@{}}708776 \\ 706616\end{tabular}   \\ \midrule
Test 5 & 5       & 90         & no                & +1.01                & 0.924 & \begin{tabular}[c]{@{}c@{}}869725  \\  869114 \end{tabular} & \begin{tabular}[c]{@{}c@{}}1035442 \\ 1044426\end{tabular}   \\ \midrule
Test 6 & 1       & 90         & no               & +3.26               & 1.037 & \begin{tabular}[c]{@{}c@{}}735201 \\ 737221\end{tabular} & \begin{tabular}[c]{@{}c@{}}1526562 \\ 1528286\end{tabular}  \\ \bottomrule
\end{tabular}
\caption{Tests for 2 bodies, case 2}
\label{tab:2bodtestc2}
\end{table}

Figure \ref{fig:2bodtest1conv} reports the convergence history for test 1. The two jumps correspond to a switch to a higher value of the penalization parameter. After a jump, the cost starts decreasing again, but stops at a larger value than the one before the jump, which would now be unattainable because of the stronger enforcement of the constraint.

\begin{figure}[h!]
\centering
\includegraphics[scale=0.5]{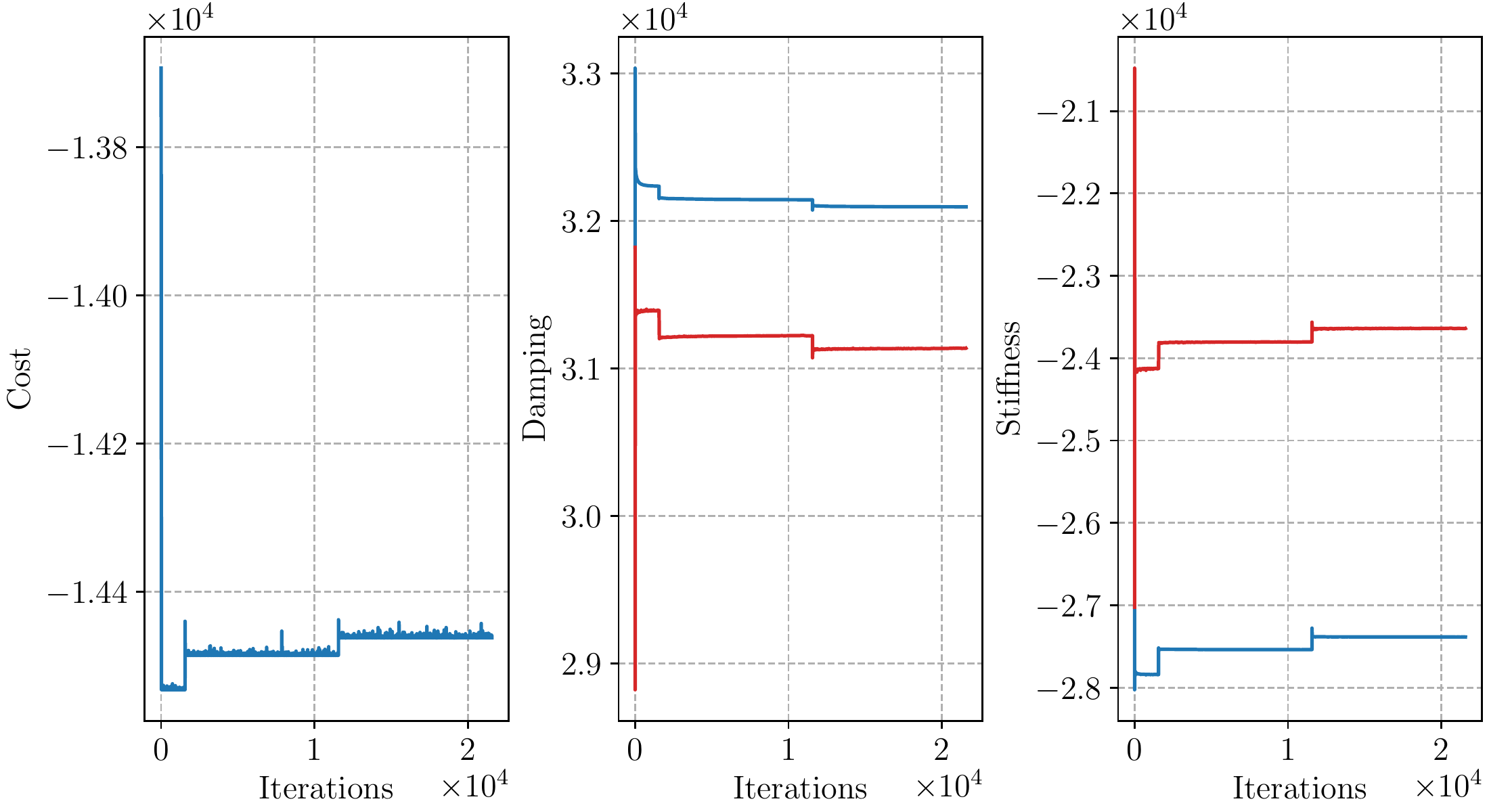}
\caption{Convergence plot for test 1 (see Table \ref{tab:2bodtest})}
\label{fig:2bodtest1conv}
\end{figure}

\subsection{Multi-body optimization: devices on two concentric arcs}
We now consider an array of 15 devices in an arrangement with 2 concentric arcs, derived from the ones considered in \cite{Goeteman2014}, for case 1. 
The geometry of the array and the results without constraints on the sign of the stiffness are shown in Fig. \ref{fig:15cs1pow} and \ref{fig:15cs1cont}. For all the following tests, the dominant wave direction is zero: the incident wave vector is aligned with the $x$ axis. We observe that the bodies with the smallest stiffnesses are the ones located in the downwave arc, at the center. They are thus the bodies with the lowest tuning frequency, and they receive the largest power increase. Conversely, the devices at the outer edges of the array have the largest stiffnesses and they are subject to a moderate power decrease. Regarding dampings, bodies in the upwave arc have larger values than the ones in the downwave arc, corresponding to wider response peaks: the bodies at the front, roughly speaking, are tuned so that they are able to absorbe power in a wider range of frequencies. 
The global result is a power increase of $3.44 \%$ compared to the initial guess, and an interaction factor of 0.966.

\begin{figure}[h!]
\centering
\includegraphics[scale=0.5]{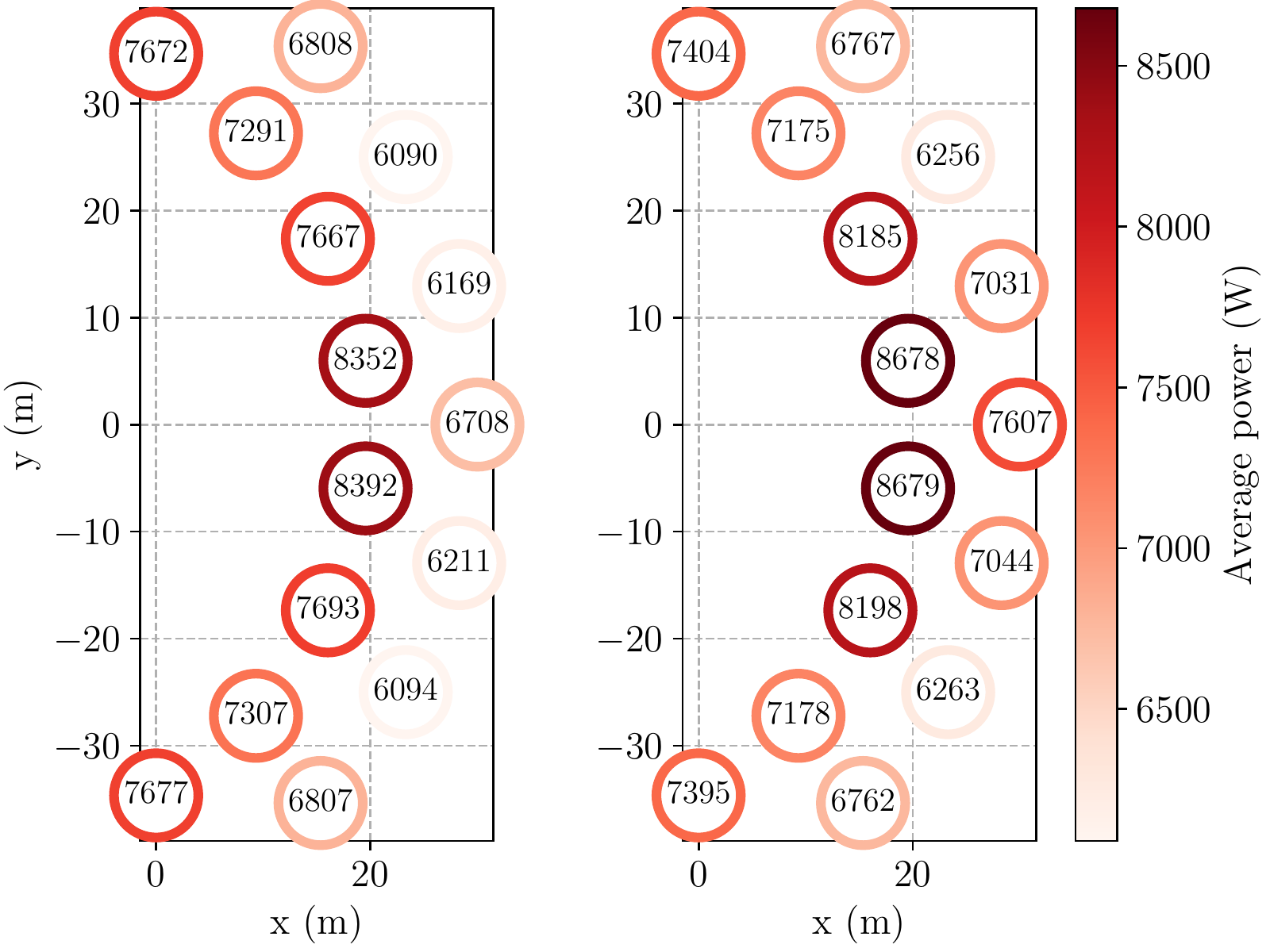}
\caption{Map of power before (left) and after (right) array optimization; case 1, 15 devices unconstrained stiffness. Bodies not to scale (radius is 2.5 m)}
\label{fig:15cs1pow}
\end{figure}

\begin{figure}[h!]
\centering
\subfloat{\includegraphics[scale=0.45]{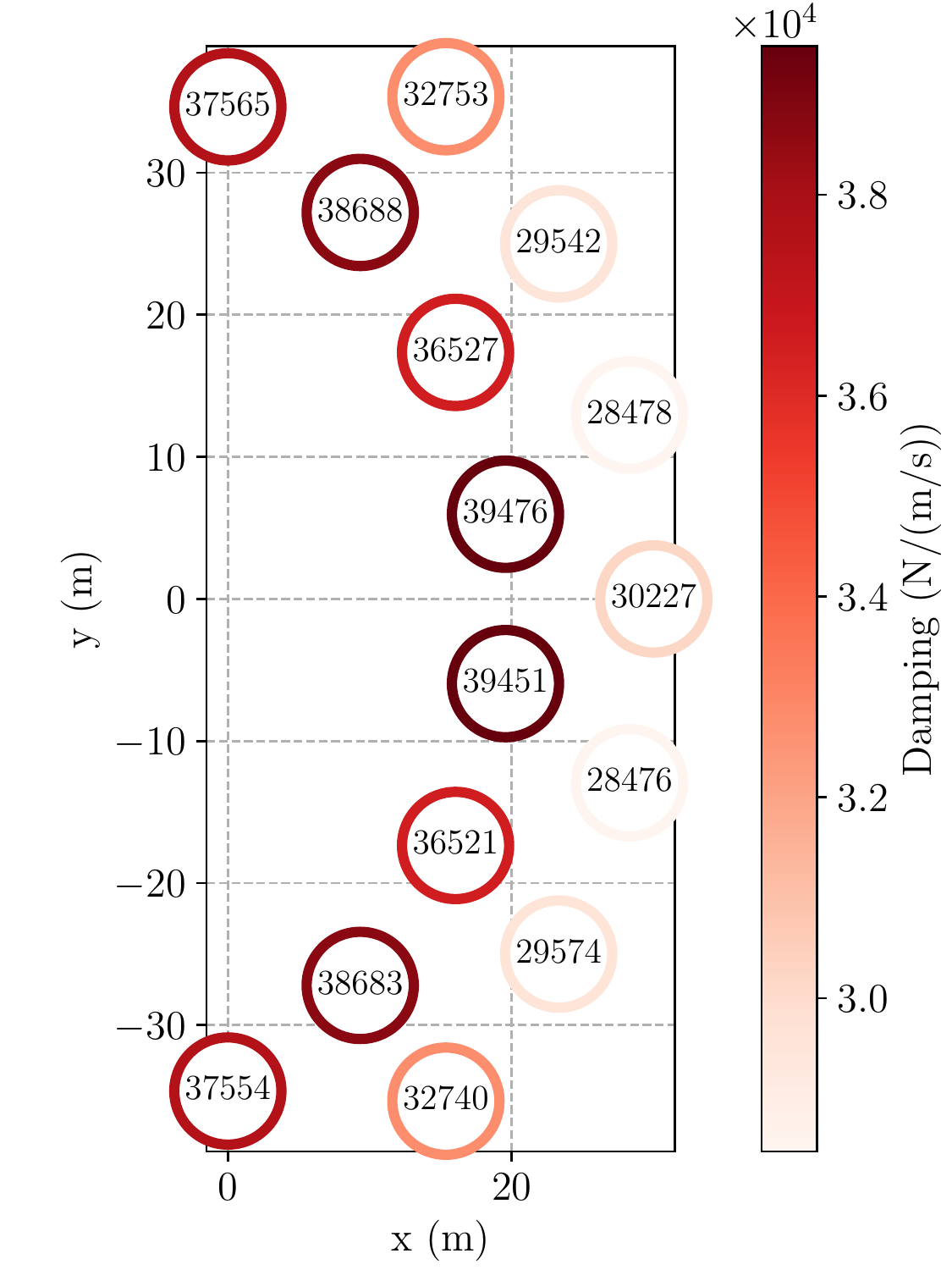}}
\subfloat{\includegraphics[scale=0.45]{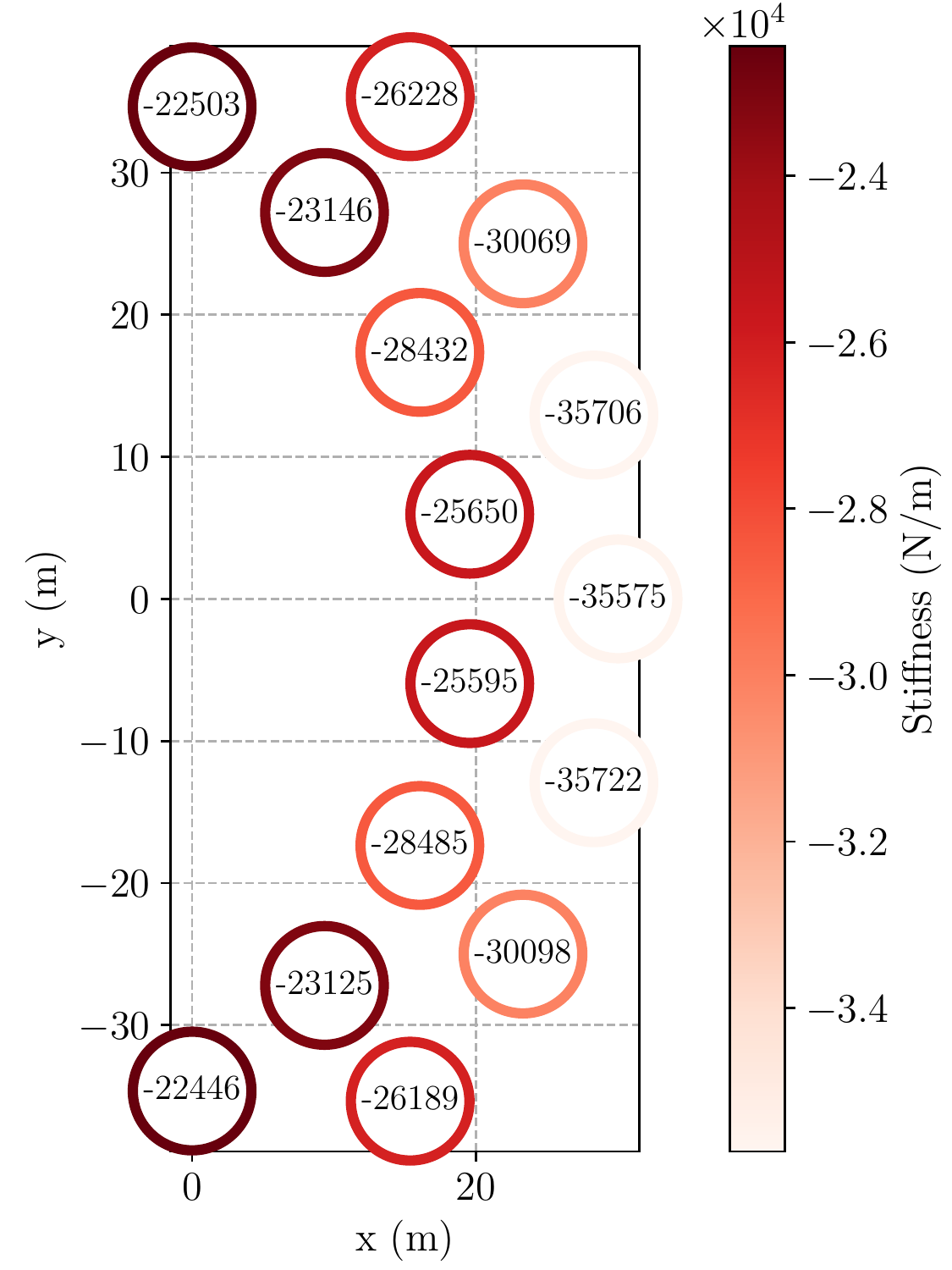}}
\caption{Map of control parameters; case 1, 15 devices, unconstrained stiffness}
\label{fig:15cs1cont}
\end{figure}

If the constraint of positive stiffnesses is enforced, the results of Fig. \ref{fig:15cs1proj} are obtained. Only a very small power increase of $0.31\%$ is obtained, and an interaction factor of $0.980$. Contrary to the previous case, bodies in the downwave arc have larger control dampings than the ones in the upwave arc.  All optimal control stiffnesses are zero; this is explained by the large, positive frequency distance between the response peaks of the uncontrolled bodies and the frequencies where most of wave power is distributed.

\begin{figure}[h!]
\centering
\subfloat{\includegraphics[scale=0.45]{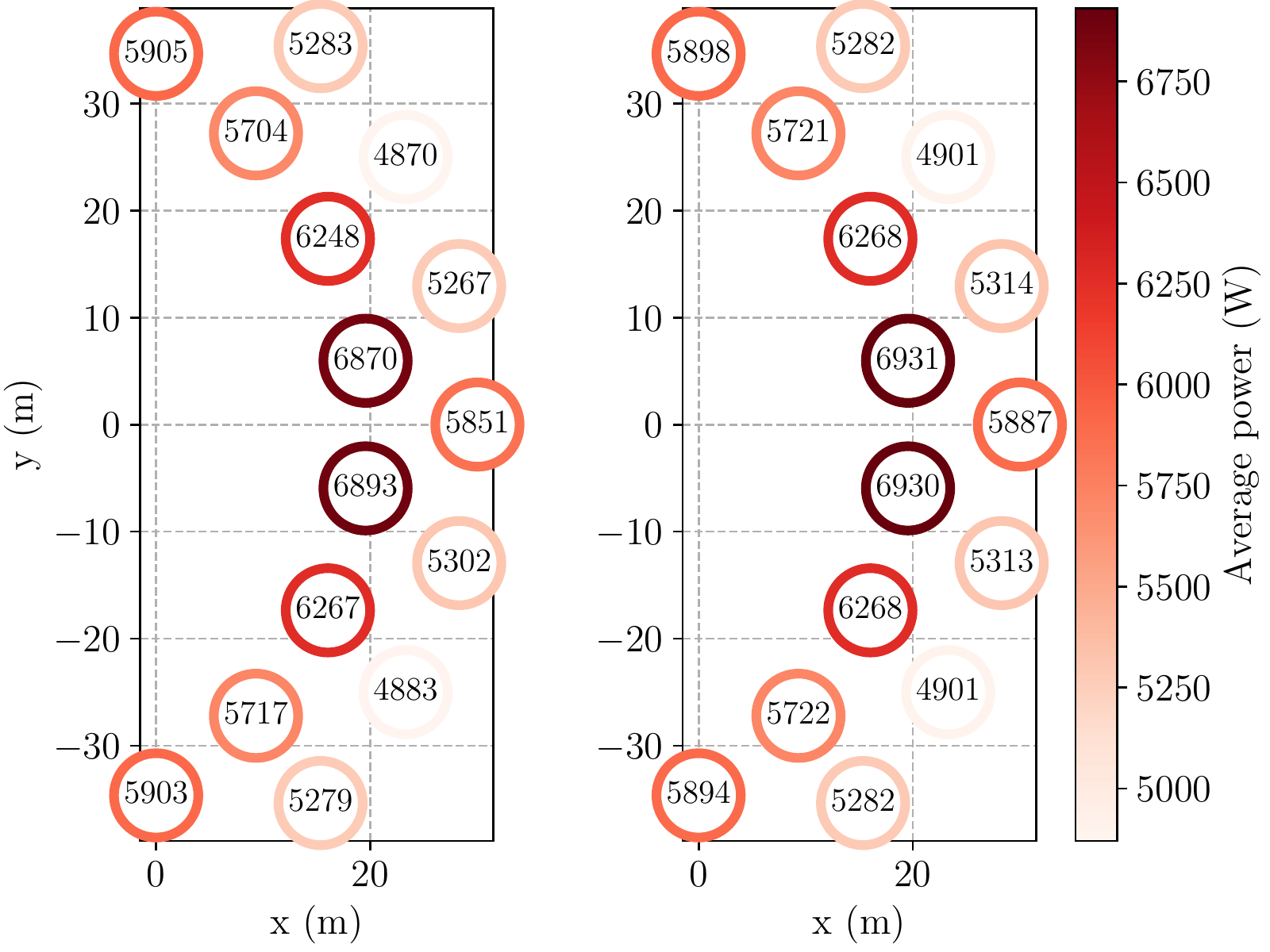}}
\hspace{0.5cm}
\subfloat{\includegraphics[scale=0.45]{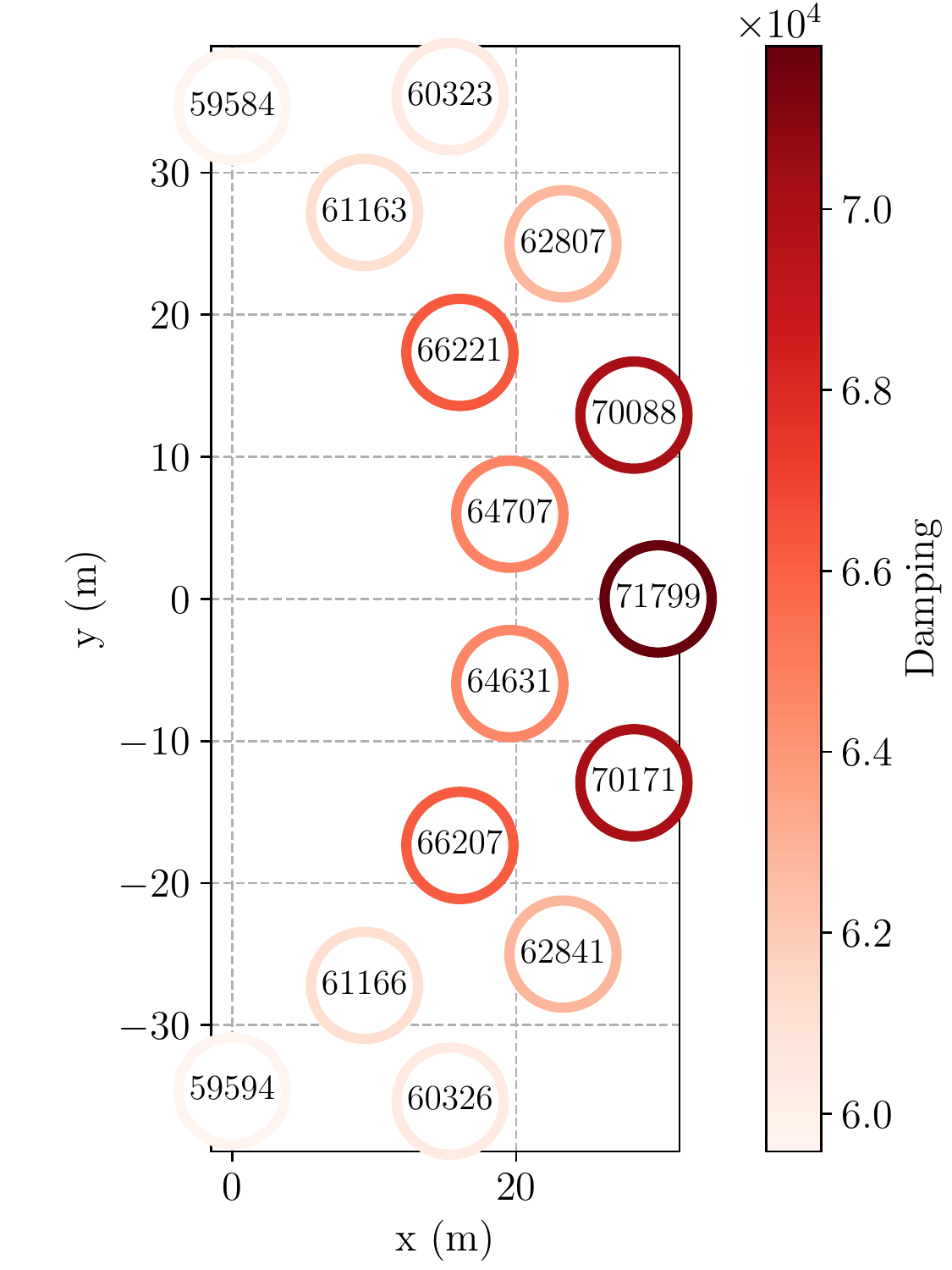}}
\caption{Maps of power (left) and control damping (right); case 1, 15 devices, positive stiffness constraint}
\label{fig:15cs1proj}
\end{figure}

For the same kind of geometry, with bodies of type 2, the results are reported in Figures \ref{fig:15cs2pow} and \ref{fig:15cs2cont}. All optimal stiffnesses are positive, without imposing any constraint. As in the corresponding arrangement for case 1, bodies in the upwave arc have generally larger stiffnesses than the ones in the downwave arc, but in this case there is the exception of the body located at $y=0$ in the downwave row. There is no evident behaviour for dampings between the two arcs. The power increase obtained by optimization is of $3.99\%$; the interaction factor is 1.144. Notice that the slamming constraint is never active, due to the relatively large draft of the bodies. Due to this, positive interactions in the array are promoted and they are not limited by the constraint.

\begin{figure}[h!]
\centering
\includegraphics[scale=0.5]{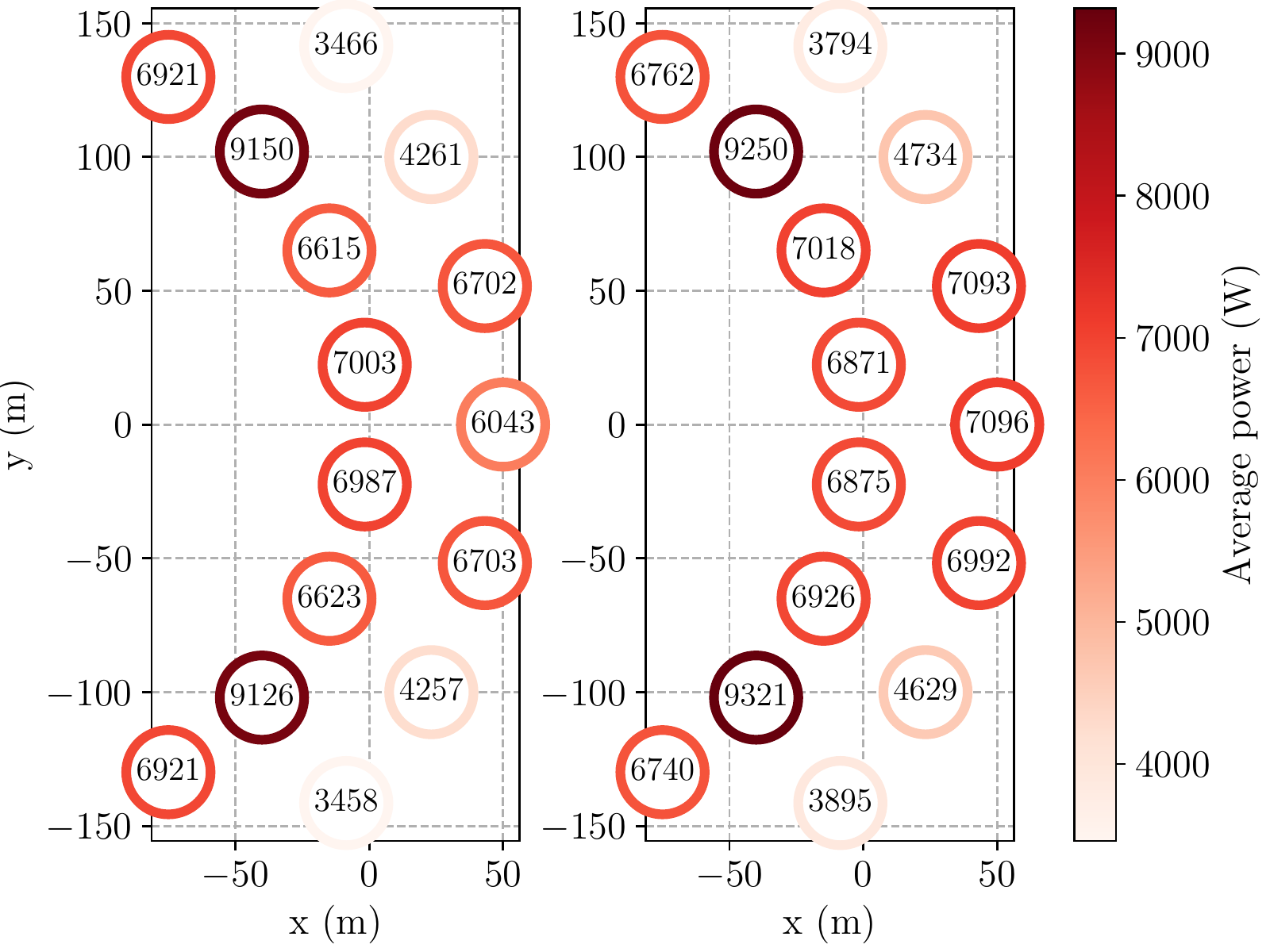}
\caption{Map of power before (left) and after (right) array optimization; case 2, 15 devices, unconstrained stiffness}
\label{fig:15cs2pow}
\end{figure}

\begin{figure}[h!]
\centering
\subfloat{\includegraphics[scale=0.45]{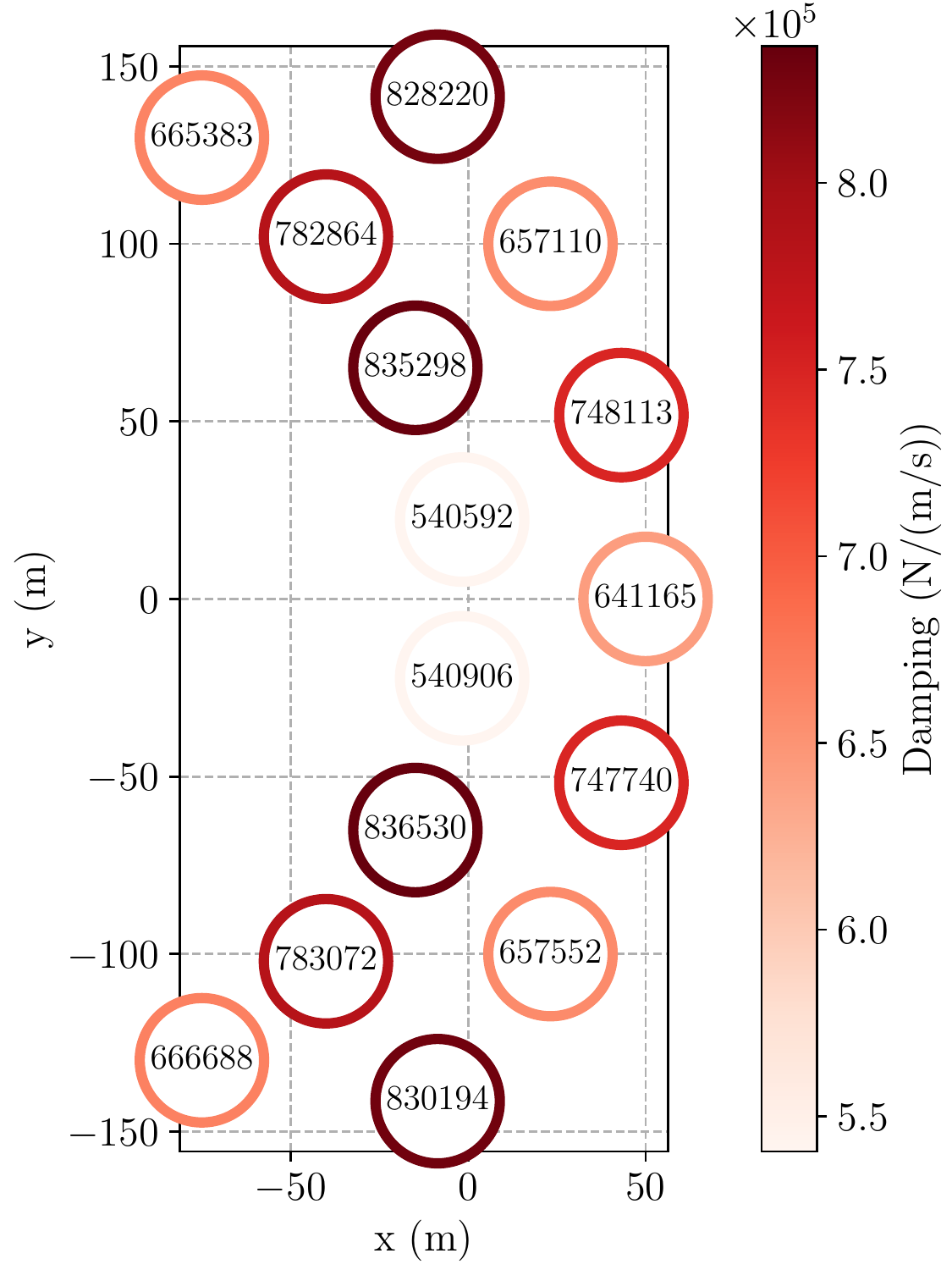}}
\subfloat{\includegraphics[scale=0.45]{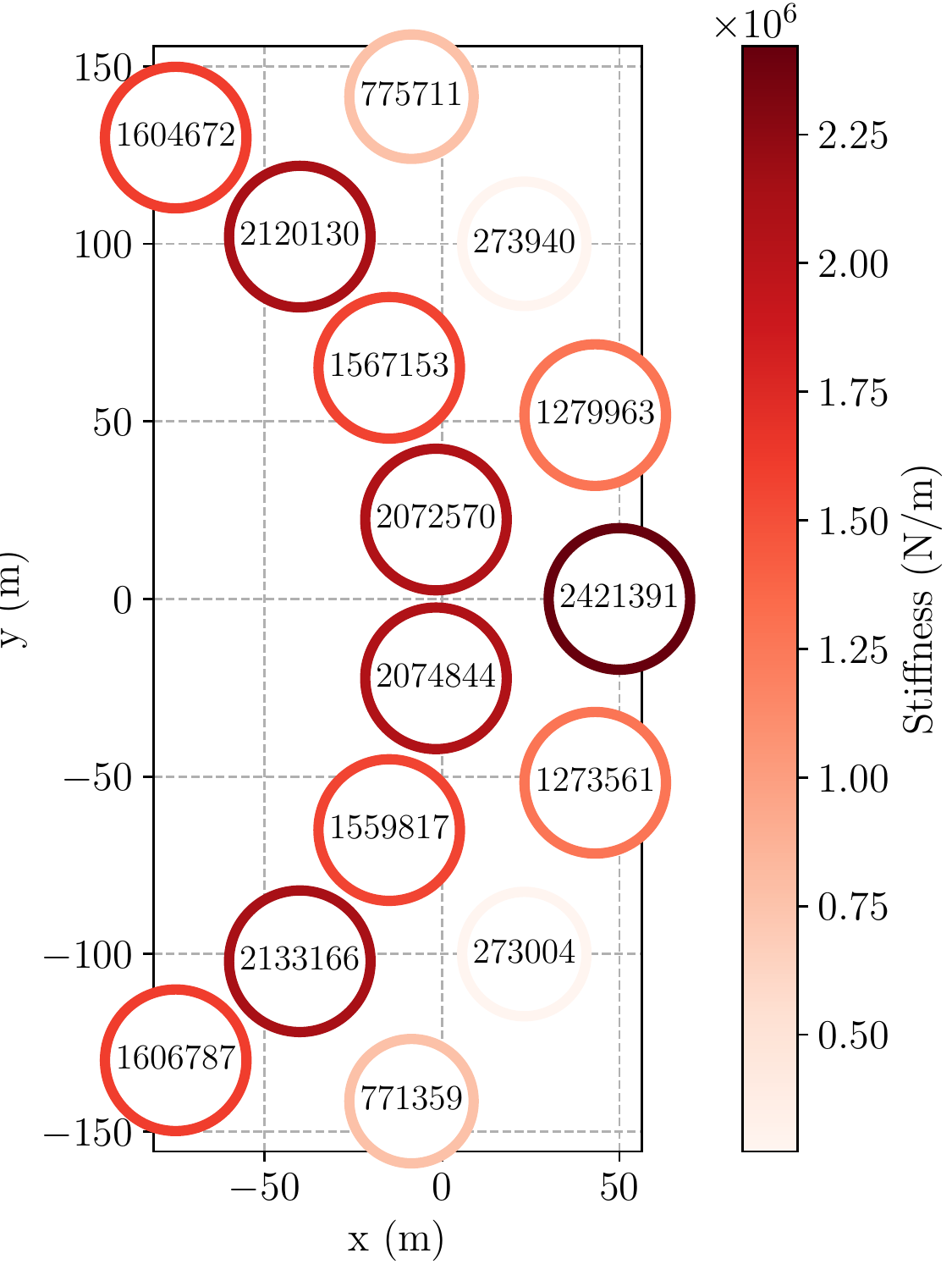}}
\caption{Map of control parameters; case 2, 15 devices, unconstrained stiffness}
\label{fig:15cs2cont}
\end{figure}

\subsection{Multi-body optimization: $4\times 4$ square arrangement}
Figures \ref{fig:16cs1pow} and \ref{fig:16cs1cont} refer to the case of a square grid of 16 bodies. For case 1, without the constraint of positive stiffnesses, we obtain all negative stiffness, the smallest ones being located in the most downwave rows. Maximum dampings are assigned to the devices at the front and center of the array. Bodies in the upwave rows are subject to a power increase compared to the initial guess; in the downwave rows, devices at the edges see a moderate power increase, while devices at the center see a power decrease, that is particularly intense for the bodies in the last row. The global power increase is of $3.11\%$, and the interaction factor is 0.766.

\begin{figure}[h!]
\centering
\includegraphics[scale=0.5]{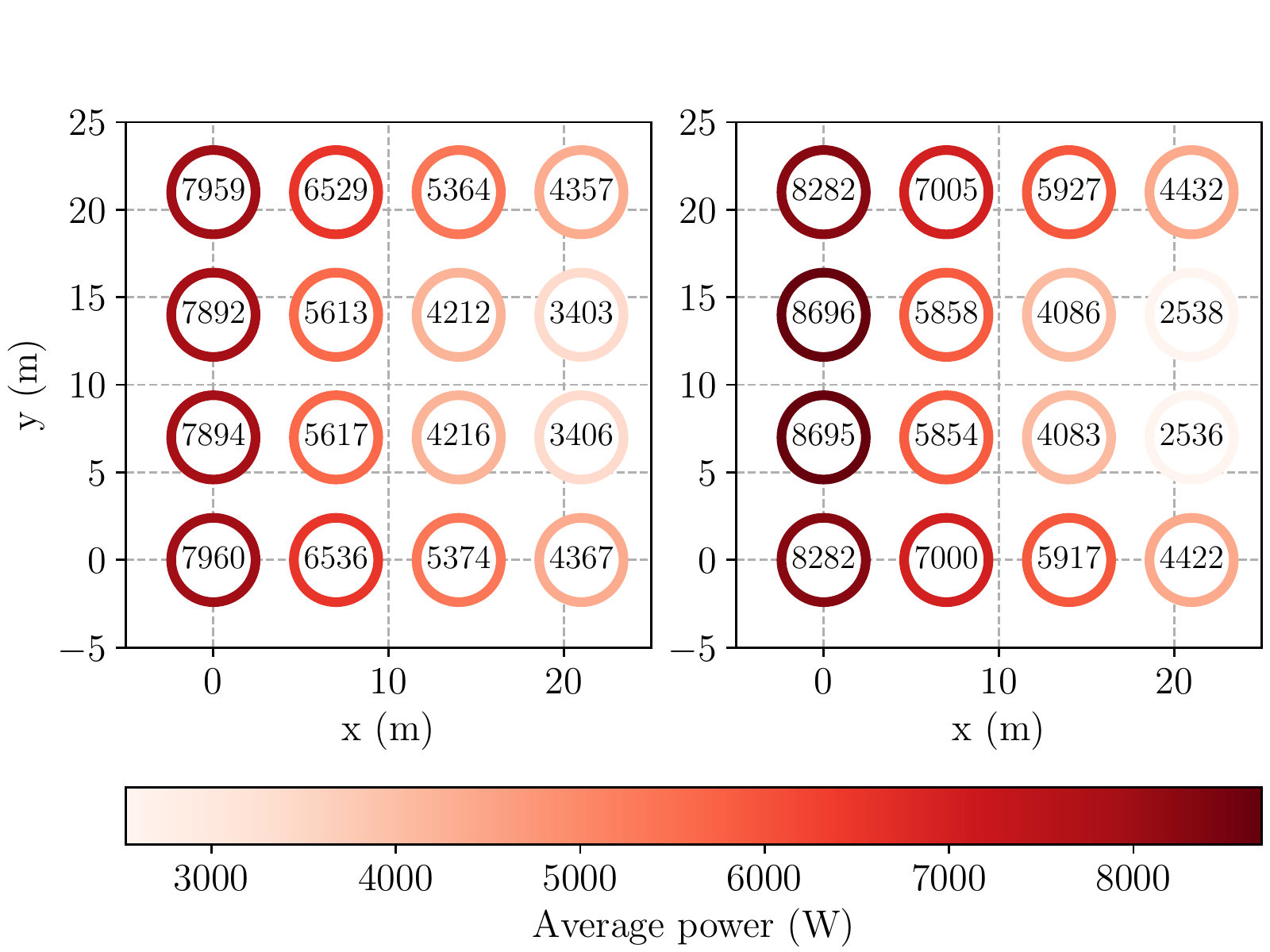}
\caption{Map of power before (left) and after (right) array optimization; case 1, 16 devices, unconstrained stiffness}
\label{fig:16cs1pow}
\end{figure}

\begin{figure}[h!]
\centering
\subfloat{\includegraphics[scale=0.45]{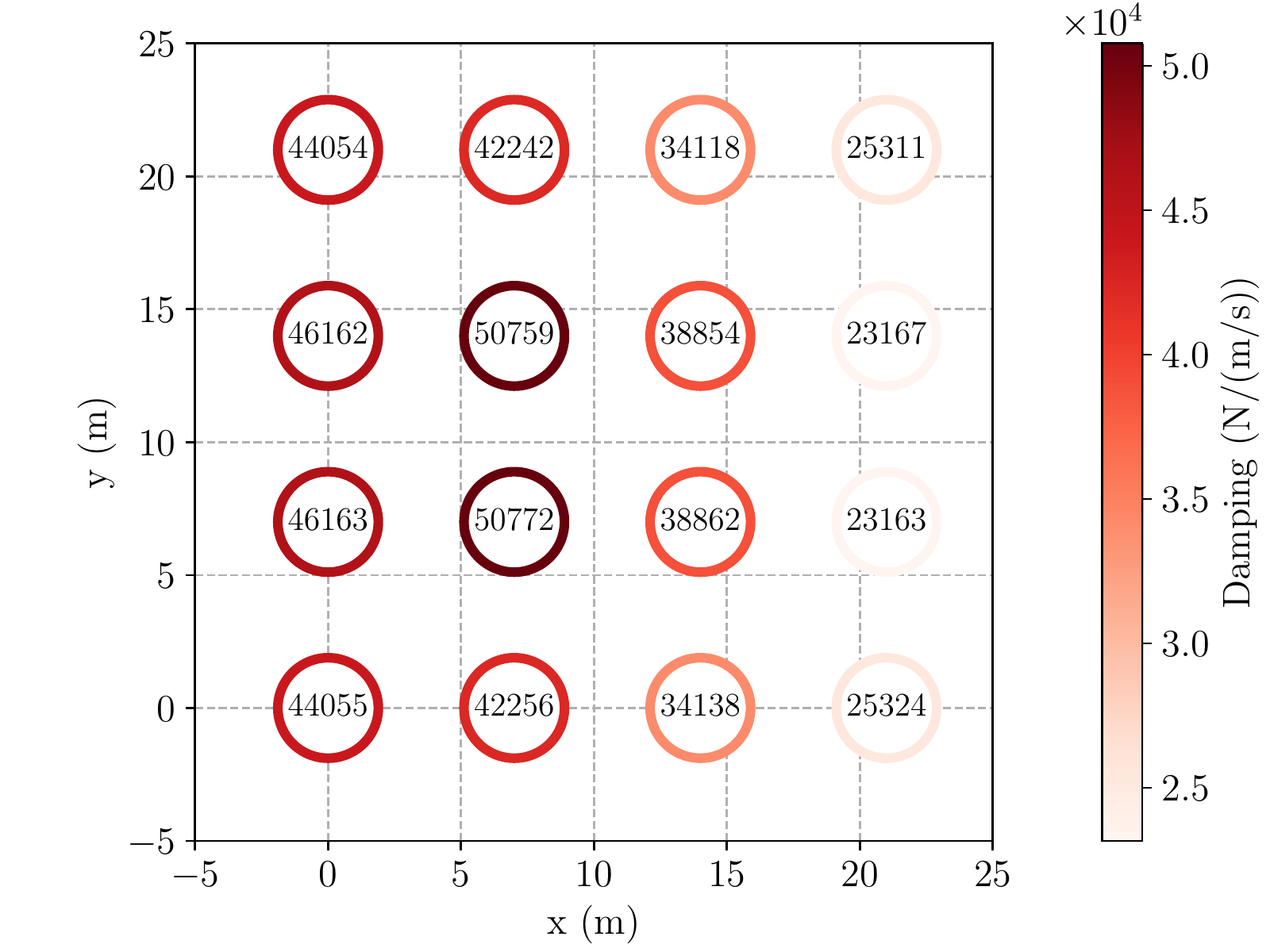}}
\subfloat{\includegraphics[scale=0.45]{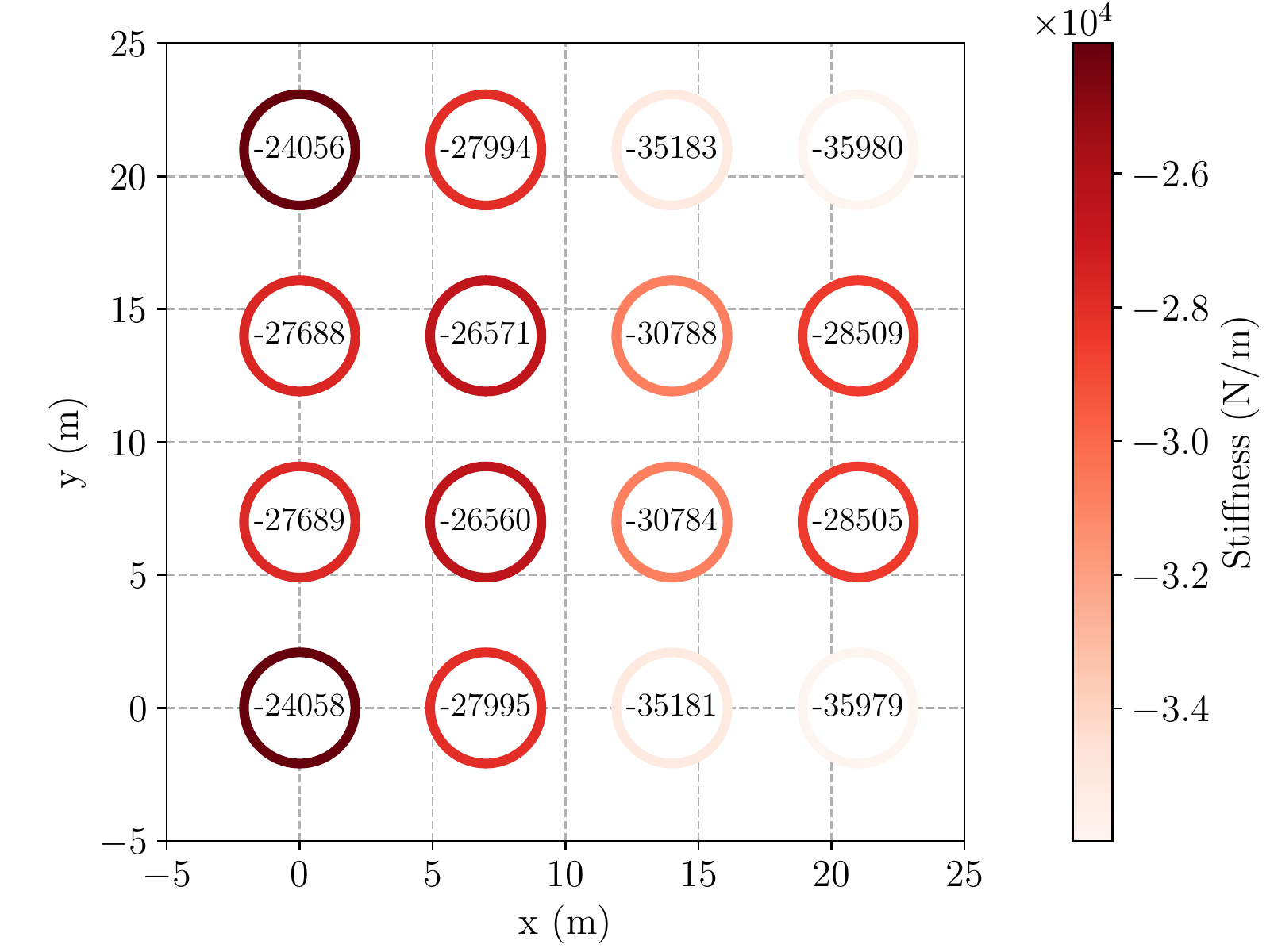}}
\caption{Map of control parameters; case 1, 16 devices, unconstrained stiffness}
\label{fig:16cs1cont}
\end{figure}

If stiffnesses are constrained to positive values, the optimization drives them all to zero. Powers and control dampings are shown in Figure \ref{fig:16cs1proj}. This time, the largest power increases are for the devices at the center of the array, which are also the ones with the largest control stiffnesses.
The power increase is of $0.41 \%$, and the obtained interaction factor is 0.816.

\begin{figure}[h!]
\centering
\subfloat{\includegraphics[scale=0.45]{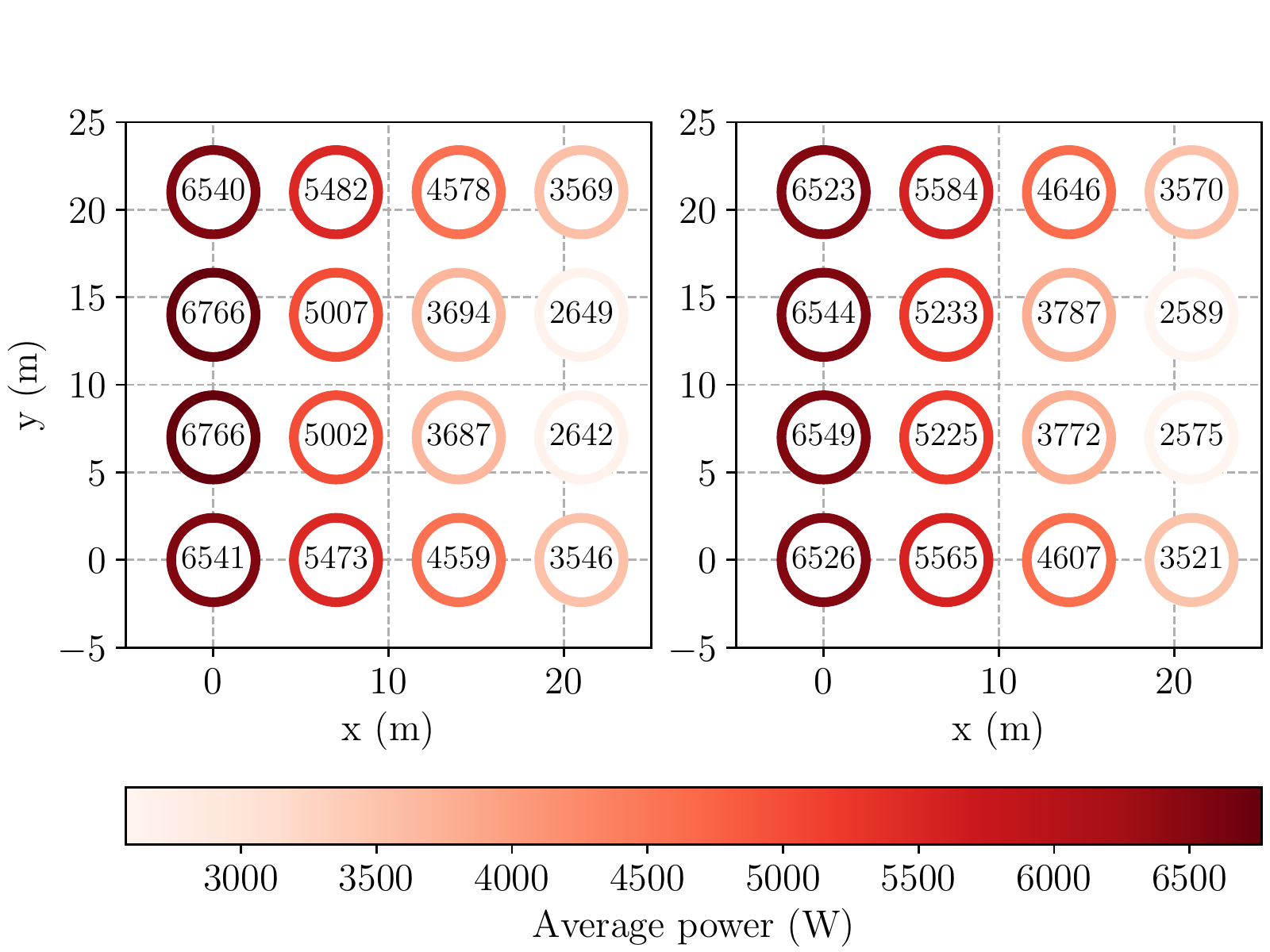}}
\subfloat{\includegraphics[scale=0.45]{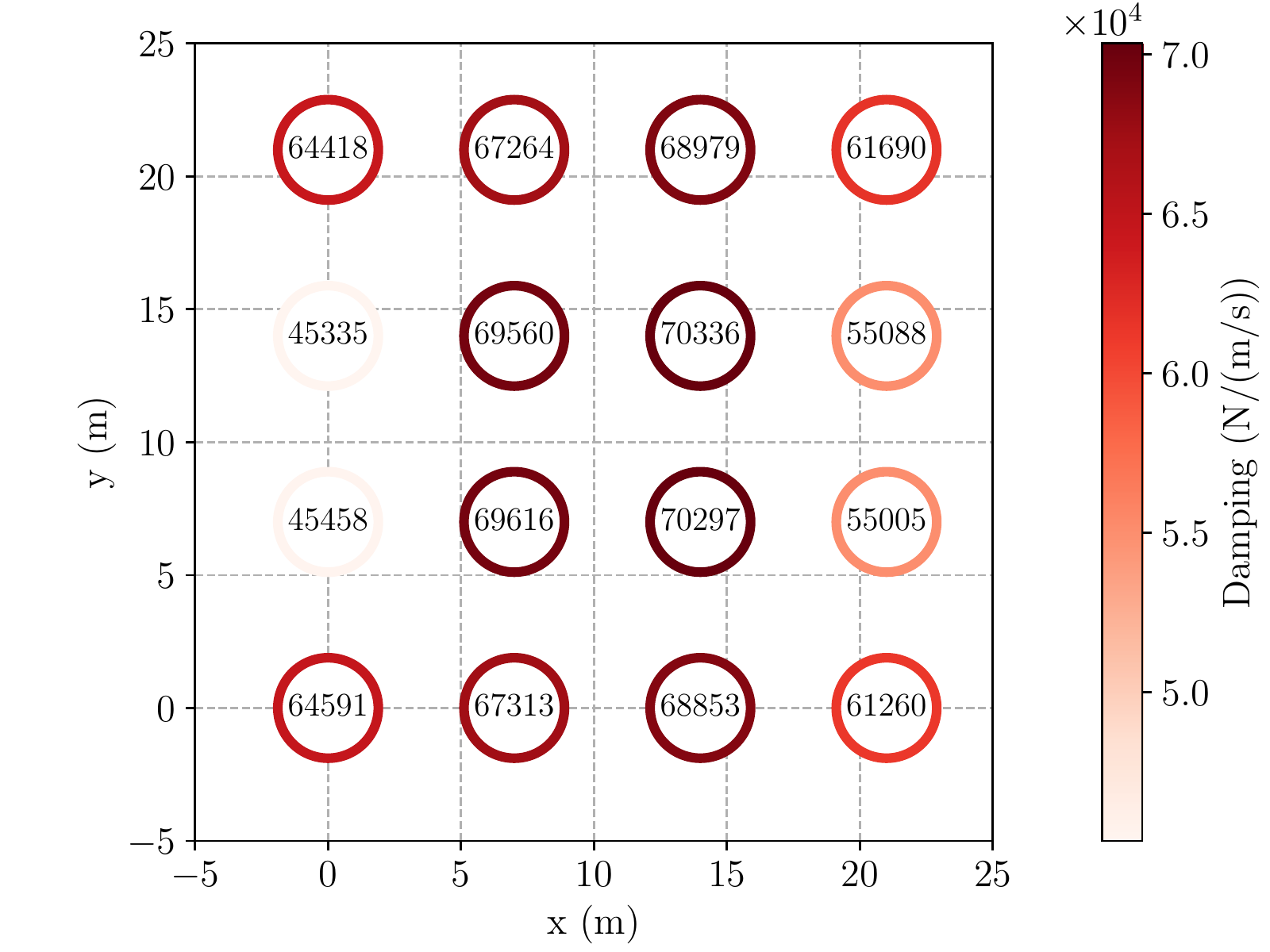}}
\caption{Maps of power (left) and control damping (right); case 1, 16 devices, positive stiffness constraint}
\label{fig:16cs1proj}
\end{figure}

A $4\times 4$ square arrangement has also been simulated for the case 2 considered in the previous section. The results are shown in Figures \ref{fig:16cs2pow} and \ref{fig:16cs2cont}. Devices in the most downwave row are subject to a power reduction, while all the others see a power increase. Bodies in the third row are assigned negative stiffnesses. The global power is increased by $9.16\%$ and the interaction factor is 0.835.

\begin{figure}[h!]
\centering
\includegraphics[scale=0.5]{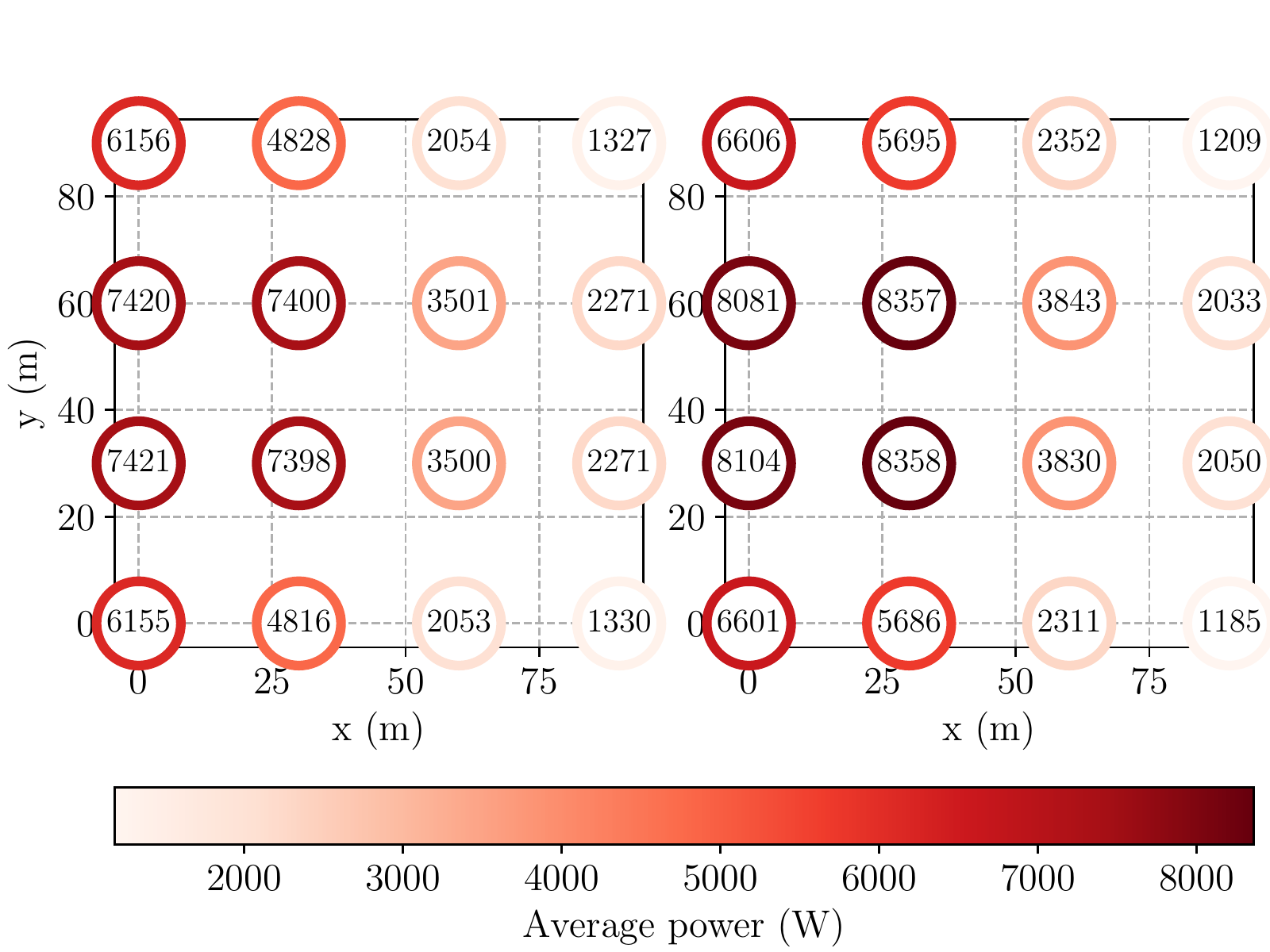}
\caption{Map of power before (left) and after (right) array optimization; case 2, 16 bodies, unconstrained stiffness}
\label{fig:16cs2pow}
\end{figure}

\begin{figure}[h!]
\centering
\subfloat{\includegraphics[scale=0.45]{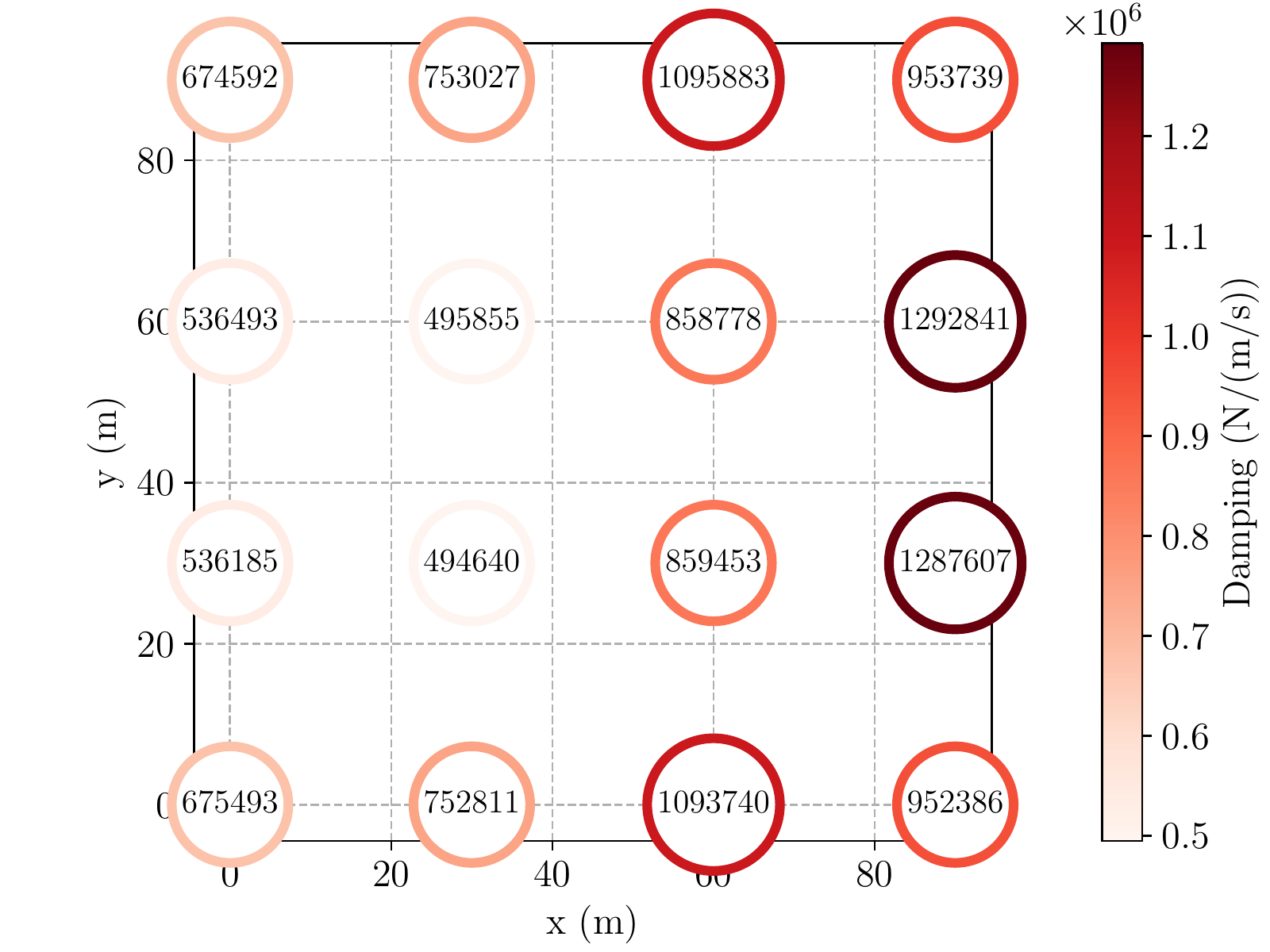}}
\subfloat{\includegraphics[scale=0.45]{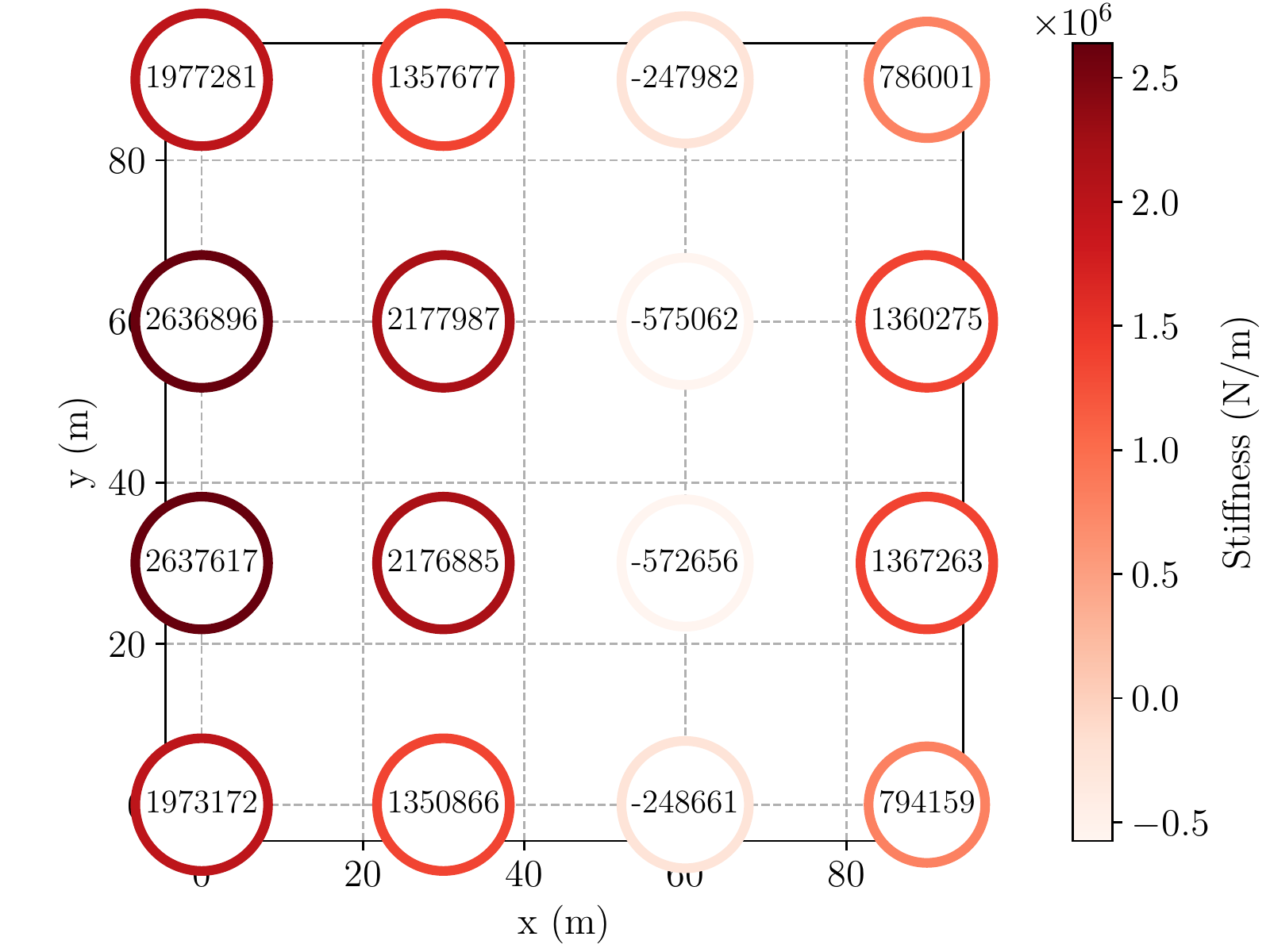}}
\caption{Map of control parameters; case 2, 16 devices, unconstrained stiffness}
\label{fig:16cs2cont}
\end{figure}

The process is repeated with the constraint of positive stiffnesses: see Fig. \ref{fig:16cs2projpow}, \ref{fig:16cs2projcont}. In terms of total power, the result is around 76.3 kW in both cases, but in this second one, almost all bodies in the downwave part of the array are assigned stiffnesses close to zero.

\begin{figure}[h!]
\centering
\includegraphics[scale=0.5]{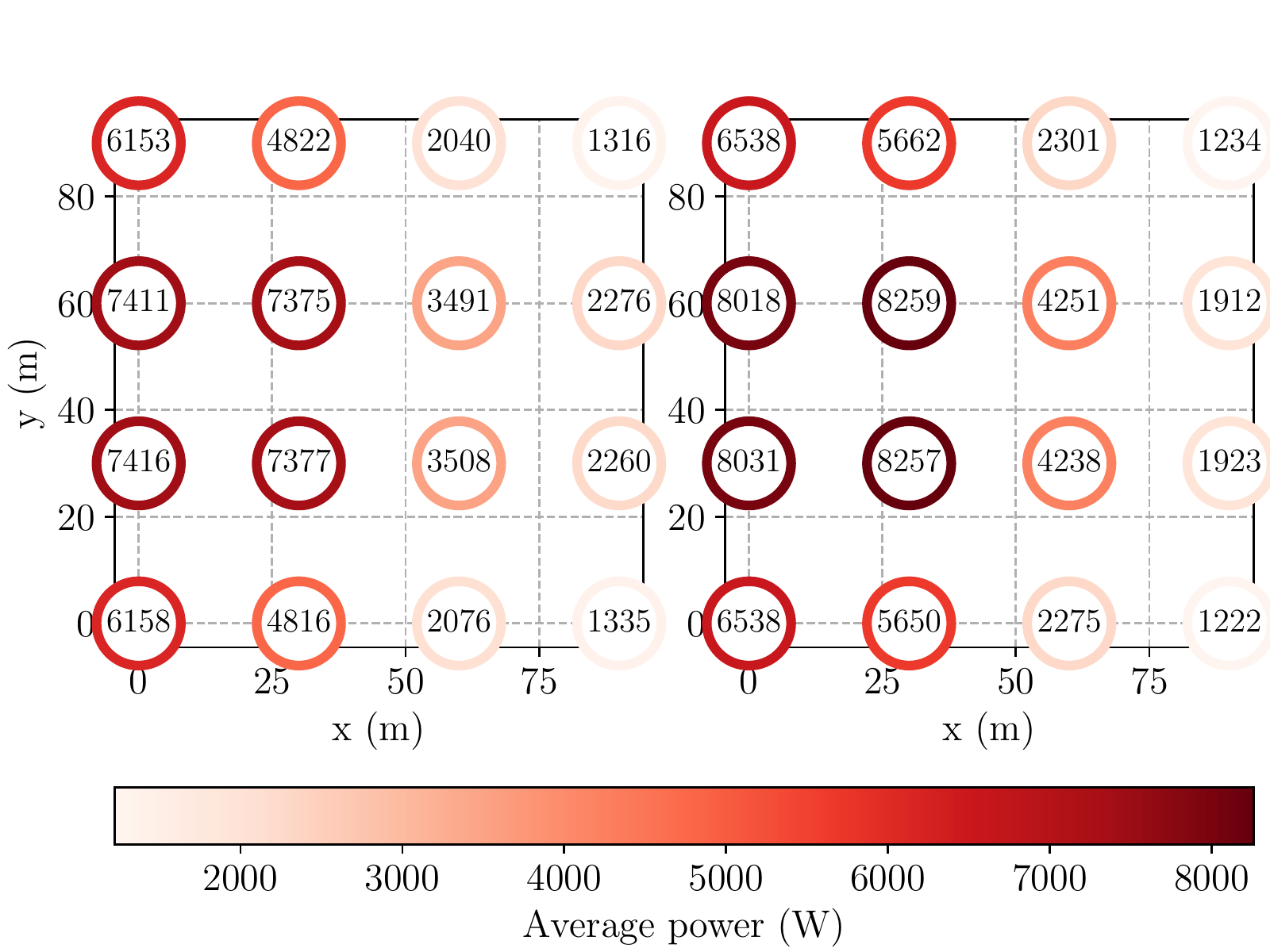}
\caption{Map of power before (left) and after (right) array optimization; case 2, 16 devices, positive stiffness constraint}
\label{fig:16cs2projpow}
\end{figure}

\begin{figure}[h!]
\centering
\subfloat{\includegraphics[scale=0.45]{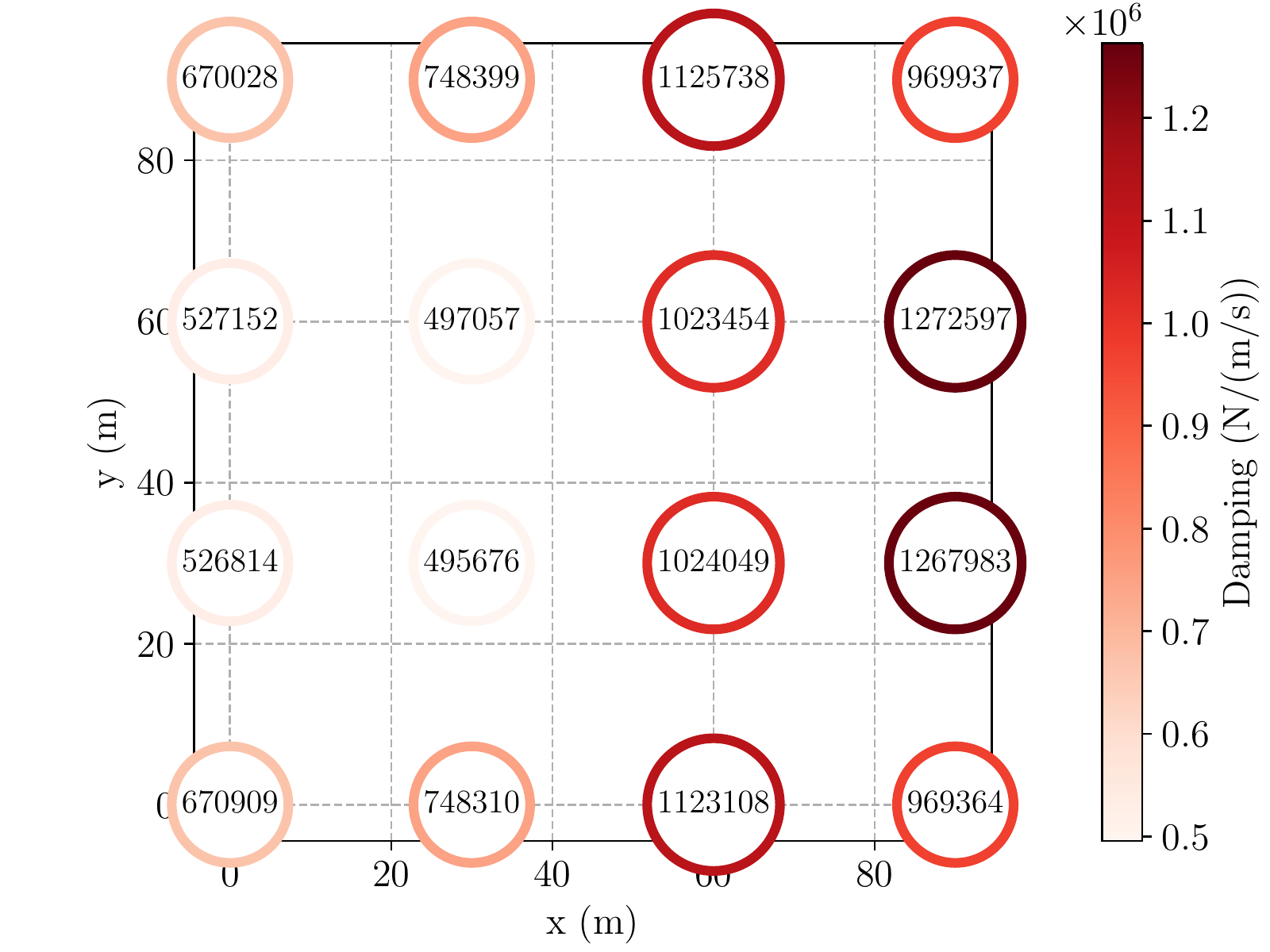}}
\subfloat{\includegraphics[scale=0.45]{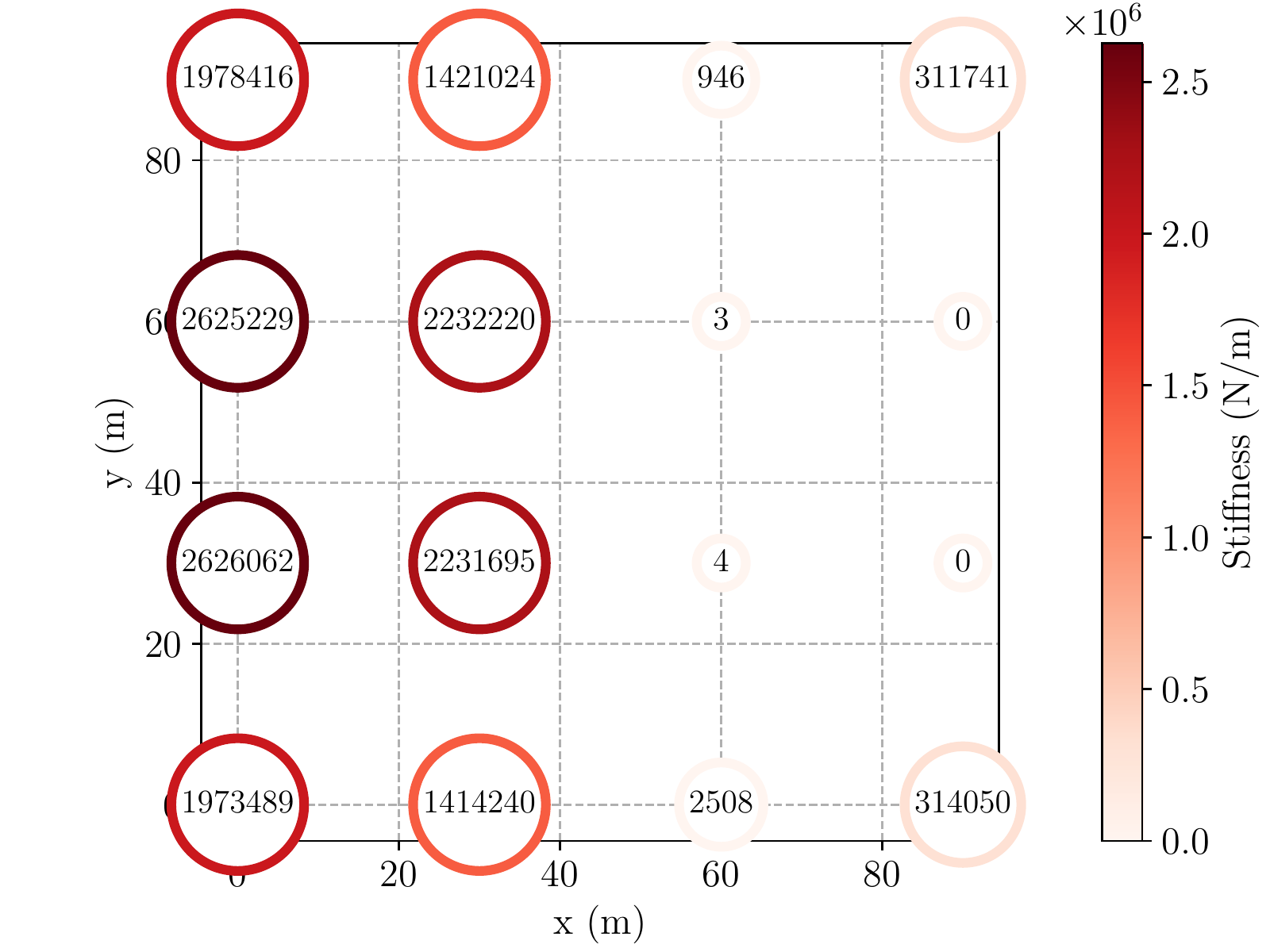}}
\caption{Map of control parameters; case 2, 16 devices, positive stiffness constraint}
\label{fig:16cs2projcont}
\end{figure}

\section{Extension to multiple random parameters}\label{sec:2peak}
The methods presented in the previous sections can be generalized to the case of multiple random variables. We present an application of practical interest involving two random parameters, based on a discussion presented in \cite{Boukhanovsky2009}. This allows us to discuss a case where our optimization framework computes an optimal solution that is, from a physical point of view, less intuitive than the ones presented in the previous sections.
Consider a sea state formed by a superposition of wind waves and swell waves. Wind waves are generated locally by the action of the wind on the free surface, while swell waves originate from storms occurring far from the measurement point. The frequency spectrum of such sea state typically presents two peaks: a relatively narrow, low-frequency peak corresponding to the swell component, and a relatively wide, high-frequency peak corresponding to the wind component. Assuming linearity of the interaction between the two wave systems, the spectrum can be written as $S(f, \theta) = S_w(f, \theta) + S_s(f, \theta)$, the subscripts $_w$ and $_s$ being referred to the wind contribution and the swell contribution, respectively. Consistently with the approximation adopted for the case of a single random direction, we write $S_w(f, \theta) = S_w(f)D_w(\theta)$, and the analogous expression for swell. The frequency part is given by the Pierson-Moskovitz spectrum \eqref{eq:PMspectrum}, while the directional part is given by the Donelan distribution function \eqref{eq:donelan}.

To formulate the robust optimization problem in this case, we assume that peak frequencies $f_{p,w}$, $f_{p,s}$ and significant wave heights $H_{s,w}$, $H_{s,s}$ are known and fixed values. We define a realization as the sum of two irregular waves, each one unidirectional with directions $\theta_w$ and $\theta_s$, sampled from $D_w(\theta)$ and $D_s(\theta)$, respectively. The excitation force acting on each device is, accordingly, the sum of the excitation forces due to the wind and swell contributions.

We present a case with the 15-device configuration of type 2 (see Tab.~\ref{tab:bodies}). The parameters of the wind component are $H_s = 2 \text{ m}$, $T_p = 5 \text{ s}$, $\beta=5$, $\theta_0=30\degree$. The parameters of the swell component are $H_s = 1 \text{ m}$, $T_p = 10 \text{ s}$, $\beta=20$, $\theta_0=0\degree$.
The corresponding spectrum is depicted in Fig.~\ref{fig:2peak-spectrum}.
\begin{figure}[h!]
    \centering
    \subfloat{
        \includegraphics[width=0.46\textwidth]{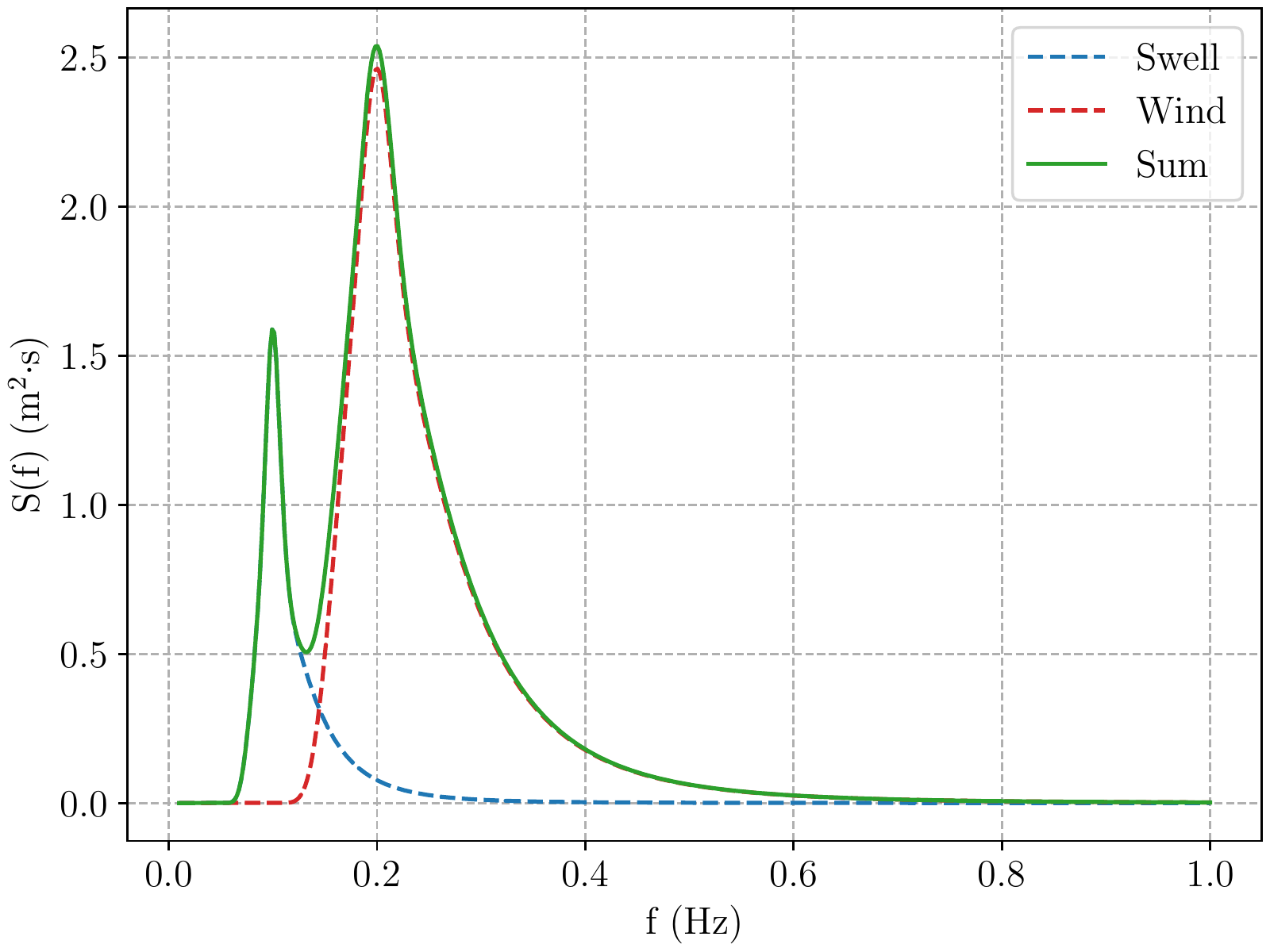}
    }
    \hfill
    \subfloat{
        \includegraphics[width=0.46\textwidth]{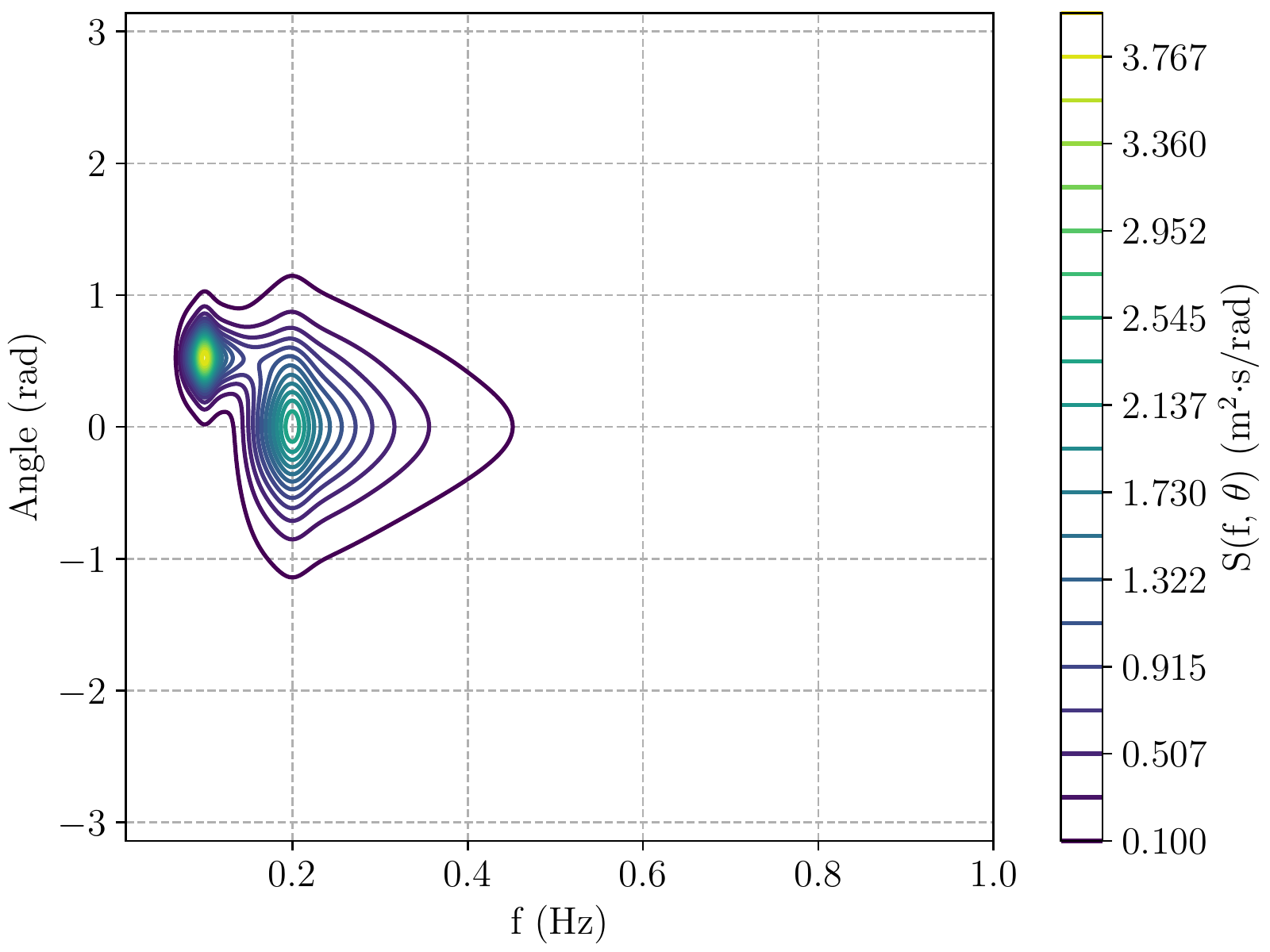}
    }
    \caption{Double-peaked wave spectrum. Left: frequency spectrum; right: frequency-directional spectrum}
    \label{fig:2peak-spectrum}
\end{figure}
The results of the optimization run are reported in Fig.~\ref{fig:15-2peak}.
\begin{figure}[h!]
\centering
\subfloat{\includegraphics[scale=0.37]{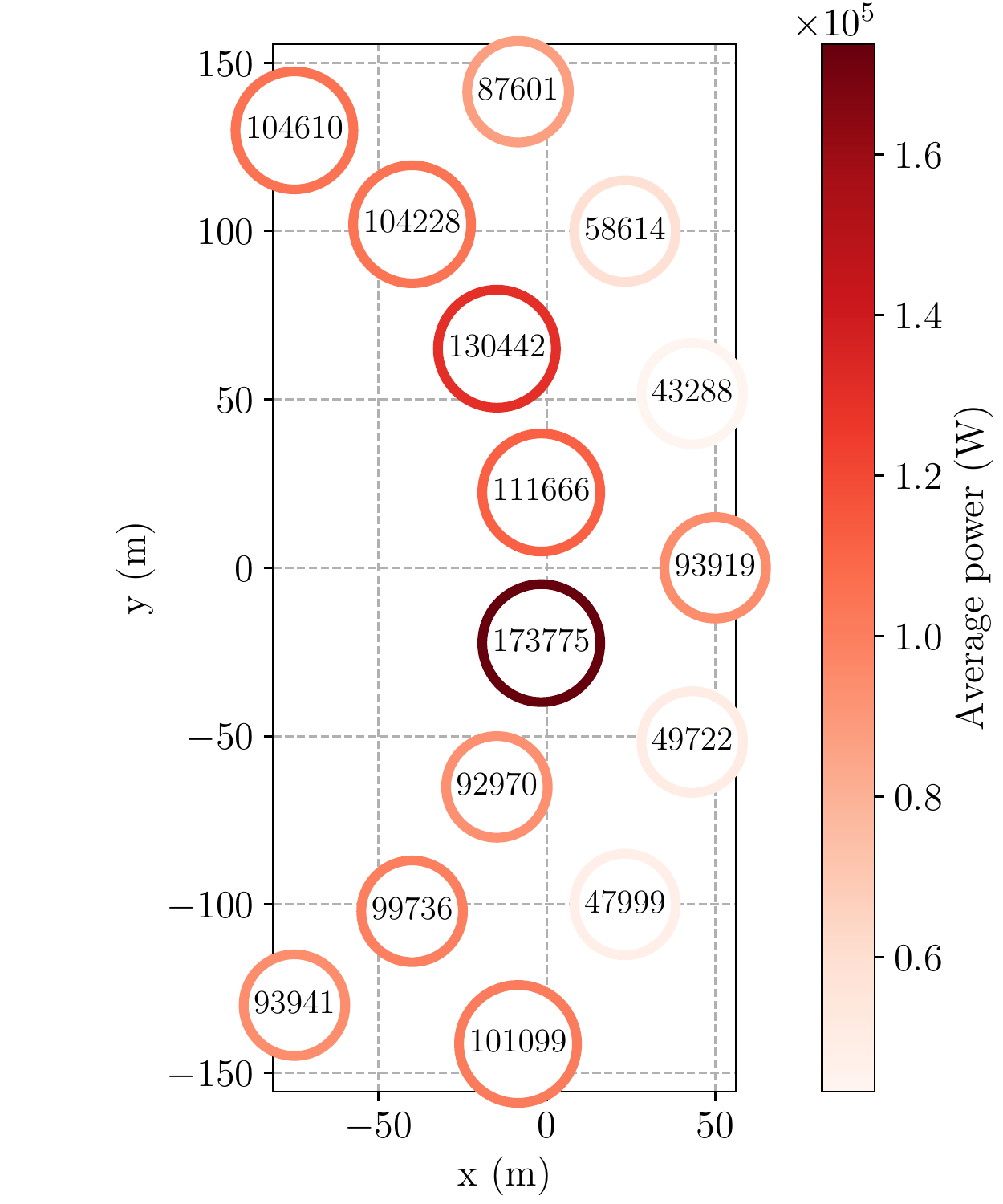}}
\subfloat{\includegraphics[scale=0.37]{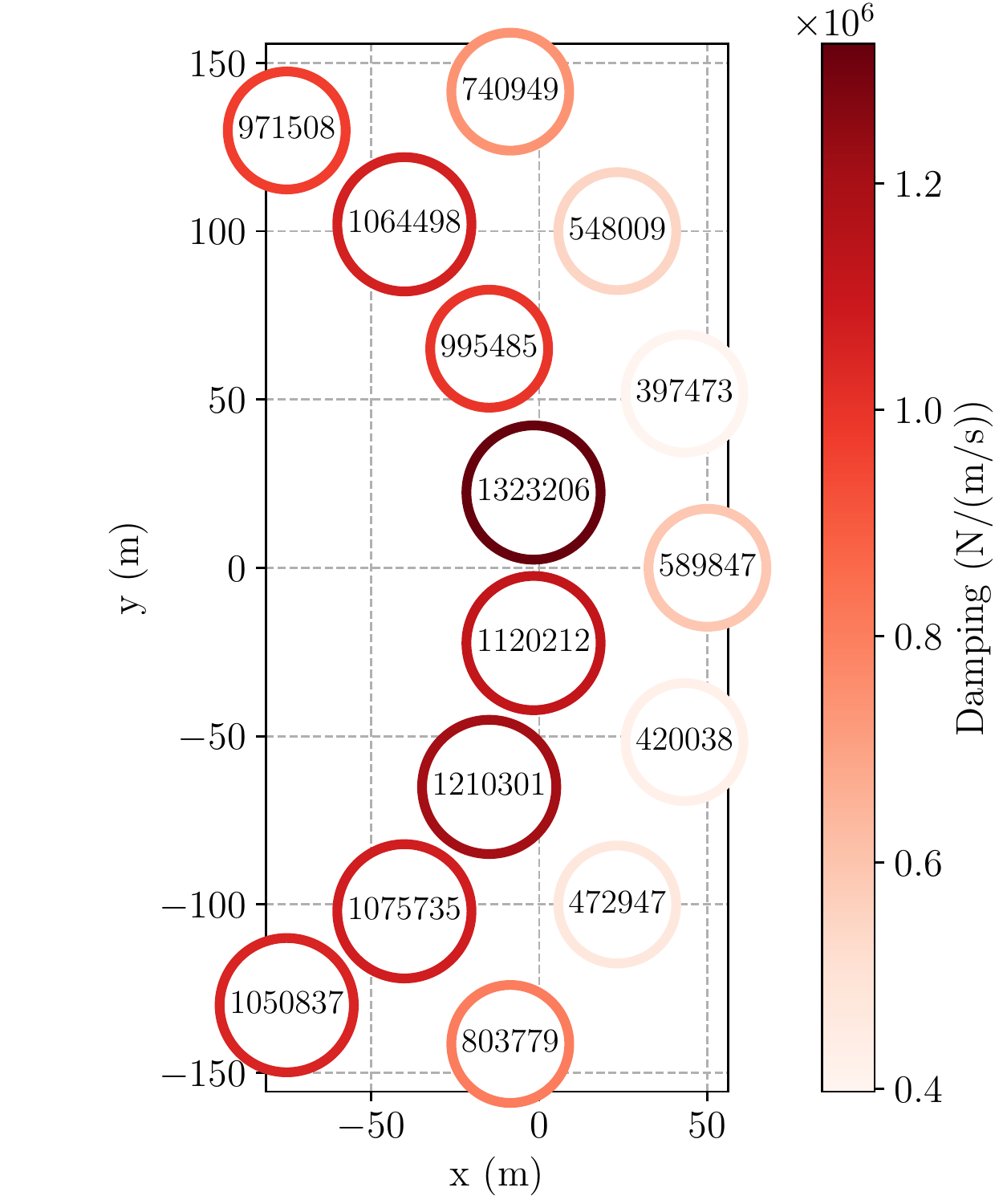}}
\subfloat{\includegraphics[scale=0.37]{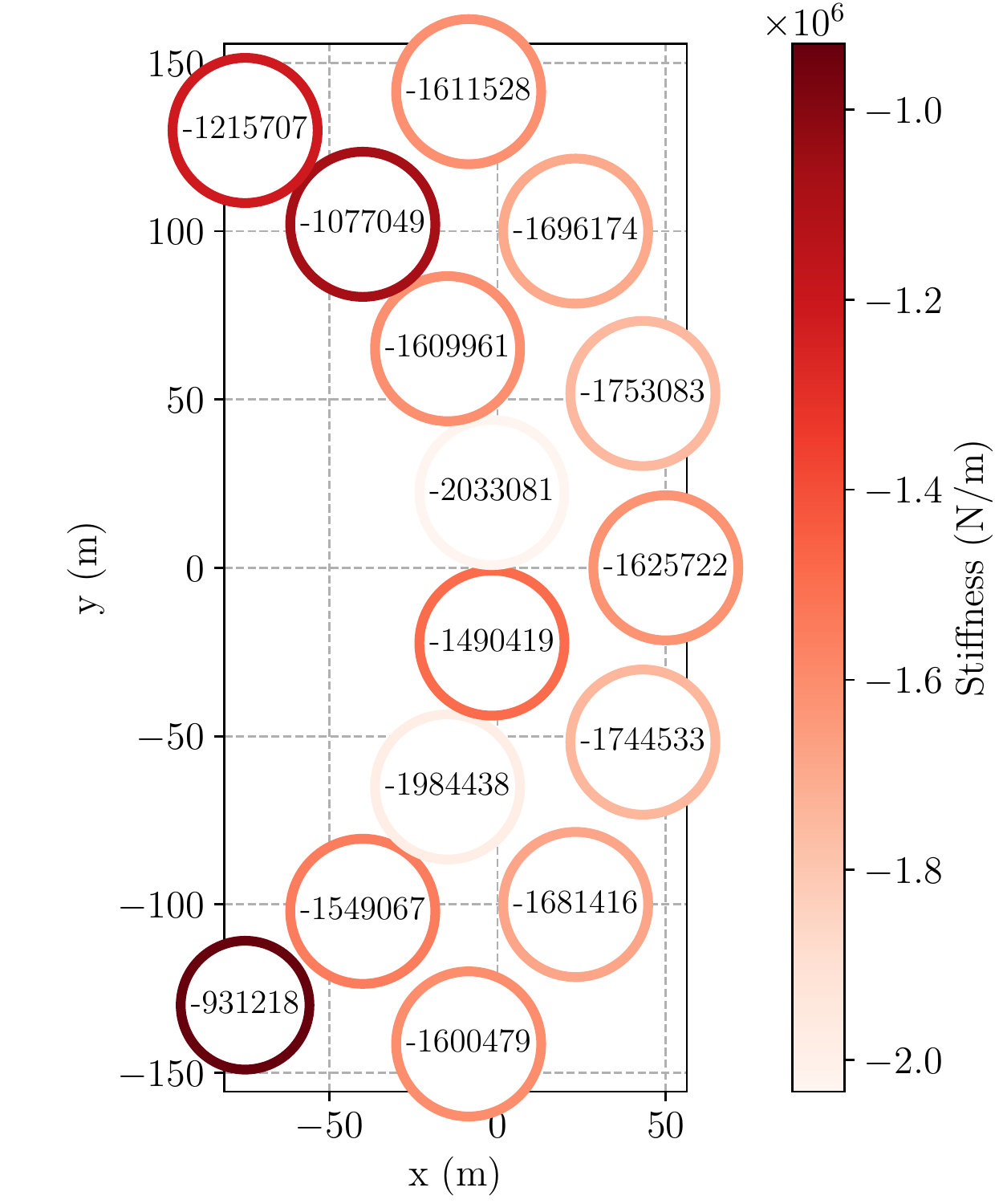}}
\caption{Map of power (left), damping (center) and stiffness (right) for 15 bodies, double-peaked spectrum}
\label{fig:15-2peak}
\end{figure}
Let us discuss the obtained results from a physical point of view. First, it is possible to see that the distributions of damping and stiffness parameters is not symmetric with respect to the $x$ axis. This is due to the asymmetry in the wave directions. 
More interestingly, one can observe that the parameters are not clearly clustered: even though one may expect to obtain two clusters, dealing with the two peaks in the spectrum independently, the optimal solution presents a more intricate configuration, with a large spread in the array of the values of stiffness and damping coefficients.

\section{Discussion and conclusions}\label{sec:concl}
In this work, a new computational framework for the robust optimization of arrays of WEC was proposed and tested.
In particular, we modelled an optimization problem for the maximization of the power of WEC parks with respect to the individual control damping and stiffness coefficients of each device. The results are robust with respect to the incident wave direction, which is treated as a random variable.
The numerical solution to this stochastic optimization problem is obtained as a combination of penalty iterations and stochastic approaches, namely robust SA and SAA based on Monte Carlo or Gauss-Legendre quadrature integration.

The results of our numerical simulations are in agreement with the ones presented (for deterministic problems and unidirectional waves) in classical references in the field of marine engineering, like, e.g., \cite{Babarit2013}, where it was observed that downwave rows are in general less productive than upwave rows. This behaviour is made even more evident by optimization, especially in square arrangements of bodies, where the difference in power between the first and the last row generally increases. Upwave bodies extract more power than they would in isolated conditions, so they benefit from the interactions with downwave bodies. These positive interactions might be impeded or reduced by the slamming constraint. In particular, it is possible that bodies tuned to have admissible oscillations individually (i.e., in isolated conditions) will violate the constraint when installed in an array. In our tests, we observed that when the initial guess is feasible, then the proposed optimization framework leads to performance increases of a few percentage points. When the initial guess is unfeasible, the method leads to a condition that is feasible  and that might have a smaller power than the unfeasible initial guess. In most of the explored cases, allowing the stiffness to have negative values provided larger power that enforcing positive stiffness. In this case, technological feasibility considerations should be made about whether such solution is realizable. Finally, by proposing an example with a double-peaked wave spectrum, we showed that our computational framework can be used to control and optimize more complex configurations, so that it is capable to provide useful indications for the preliminary control design of realistic WEC arrays.

\section*{Acknowledgments}
Funding sources:
all authors are members of the GNCS Indam group. The present research is part of the activities of “Dipartimento di Eccellenza 2023-2027.”

\bibliography{refers}{}

\begin{thebibliography}{10}

\bibitem{oceaneuroreport}
2030 ocean energy vision: Industry analysis of future deployments, costs and
  supply chains.
\newblock Technical report, {Ocean Energy Europe}, 2020.

\bibitem{eurocomm2020}
An {EU Strategy} to harness the potential of offshore renewable energy for a
  climate neutral future.
\newblock Technical report, {European Commission}, 2020.

\bibitem{windeurope2022}
Wind energy in {Europe} 2021. {Statistics} and the outlook for 2022-2026.
\newblock Technical report, Wind Europe, 2022.

\bibitem{Ancellin2019}
M.~Ancellin and F.~Dias.
\newblock Capytaine: a {Python}-based linear potential flow solver.
\newblock {\em Journal of Open Source Software}, 4(36):1341, Apr. 2019.

\bibitem{Babarit2013}
A.~Babarit.
\newblock On the park effect in arrays of oscillating wave energy converters.
\newblock {\em Renewable Energy}, 58:68--78, 2013.

\bibitem{Babarit2012}
A.~Babarit, J.~Hals, M.~Muliawan, A.~Kurniawan, T.~Moan, and J.~Krokstad.
\newblock Numerical benchmarking study of a selection of wave energy
  converters.
\newblock {\em Renewable Energy}, 41:44--63, may 2012.

\bibitem{Backer2010}
G.~D. Backer, M.~Vantorre, C.~Beels, J.~D. Rouck, and P.~Frigaard.
\newblock Power absorption by closely spaced point absorbers in constrained
  conditions.
\newblock {\em {IET} Renewable Power Generation}, 4(6):579, 2010.

\bibitem{Boukhanovsky2009}
A.~Boukhanovsky and C.~G. Soares.
\newblock Modelling of multipeaked directional wave spectra.
\newblock {\em Applied Ocean Research}, 31(2):132--141, apr 2009.

\bibitem{Chakrabarti1987}
S.~K. Chakrabarti.
\newblock {\em {Hydrodynamics of Offshore Structures}}.
\newblock Computational Mechanics Publications. Springer-Verlag, 1987.

\bibitem{Eriksson2007}
M.~Eriksson, R.~Waters, O.~Svensson, J.~Isberg, and M.~Leijon.
\newblock Wave power absorption: Experiments in open sea and simulation.
\newblock {\em Journal of Applied Physics}, 102(8):084910, oct 2007.

\bibitem{Falnes2020}
J.~Falnes and A.~Kurniawan.
\newblock {\em Ocean Waves and Oscillating Systems: Linear Interactions
  Including Wave-Energy Extraction}.
\newblock Cambridge University Press, 2 edition, 2020.

\bibitem{Goda1999}
Y.~Goda.
\newblock A comparative review on the functional forms of directional wave
  spectrum.
\newblock {\em Coastal Engineering Journal}, 41(1):1--20, Mar. 1999.

\bibitem{Goda2010}
Y.~Goda.
\newblock {\em Random Seas and Design of Maritime Structures}.
\newblock World Scientific, June 2010.

\bibitem{Goeteman2014}
M.~Göteman, J.~Engström, M.~Eriksson, J.~Isberg, and M.~Leijon.
\newblock Methods of reducing power fluctuations in wave energy parks.
\newblock {\em Journal of Renewable and Sustainable Energy}, 6:043103, 07 2014.

\bibitem{Goeteman2020}
M.~Göteman, M.~Giassi, J.~Engström, and J.~Isberg.
\newblock Advances and challenges in wave energy park optimization—a review.
\newblock {\em Frontiers in Energy Research}, 8:26, 2020.

\bibitem{Haldar2000}
A.~Haldar and S.~Mahadevan.
\newblock {\em {Probability, Reliability and Statistical Methods in Engineering
  Design}}.
\newblock John Wiley \& S., 2000.

\bibitem{Korde2016}
U.~A. Korde and J.~Ringwood.
\newblock {\em Hydrodynamic Control of Wave Energy Devices}.
\newblock Cambridge University Press, sep 2016.

\bibitem{Kundu2012}
P.~Kundu, I.~Cohen, and D.~Dowling.
\newblock {\em Fluid Mechanics}.
\newblock Elsevier Science, 2012.

\bibitem{Lalanne1999}
C.~Lalanne.
\newblock {\em Random Vibration}.
\newblock {CRC} Press, Feb. 1999.

\bibitem{arenas2019}
A.~Maria-Arenas, A.~J. Garrido, E.~Rusu, and I.~Garrido.
\newblock Control strategies applied to wave energy converters: State of the
  art.
\newblock {\em Energies}, 12(16), 2019.

\bibitem{Nemirovski2009}
A.~Nemirovski, A.~Juditsky, G.~Lan, and A.~Shapiro.
\newblock Robust stochastic approximation approach to stochastic programming.
\newblock {\em {SIAM} Journal on Optimization}, 19(4):1574--1609, Jan. 2009.

\bibitem{Newman1977}
J.~Newman.
\newblock {\em Marine Hydrodynamics}.
\newblock Mit Press. 1977.

\bibitem{Nocedal2006}
J.~Nocedal and S.~J. Wright.
\newblock {\em {Numerical Optimization}}.
\newblock Springer New York, 2006.

\bibitem{Patel2021}
V.~Patel.
\newblock Stopping criteria for, and strong convergence of, stochastic gradient
  descent on {B}ottou-{C}urtis-{N}ocedal functions.
\newblock {\em Mathematical Programming}, Oct. 2021.

\bibitem{Quarteroni2006}
A.~Quarteroni, R.~Sacco, and F.~Saleri.
\newblock {\em Numerical Mathematics (Texts in Applied Mathematics)}.
\newblock Springer-Verlag, Berlin, Heidelberg, 2006.

\bibitem{Rychlik1997}
I.~Rychlik, P.~Johannesson, and M.~Leadbetter.
\newblock Modelling and statistical analysis of ocean-wave data using
  transformed gaussian processes.
\newblock {\em Marine Structures}, 10(1):13--47, jan 1997.

\bibitem{Shapiro2014}
A.~Shapiro, D.~Dentcheva, and A.~Ruszczy{\'{n}}ski.
\newblock {\em Lectures on Stochastic Programming: Modeling and Theory, Second
  Edition}.
\newblock Society for Industrial and Applied Mathematics, jan 2014.

\bibitem{Sinha2016}
A.~Sinha, D.~Karmakar, and C.~G. Soares.
\newblock Performance of optimally tuned arrays of heaving point absorbers.
\newblock {\em Renewable Energy}, 92:517--531, jul 2016.

\bibitem{Thornhill2019}
E.~Thornhill.
\newblock Interpolation of complex-valued response amplitude operators.
\newblock Technical report, DRDC – Atlantic Research Centre, Nov. 2019.

\bibitem{Todalshaug2016}
J.~H. Todalshaug, G.~S. {\'{A}}sgeirsson, E.~Hj{\'{a}}lmarsson, J.~Maillet,
  P.~Möller, P.~Pires, M.~Gu{\'{e}}rinel, and M.~Lopes.
\newblock Tank testing of an inherently phase-controlled wave energy converter.
\newblock {\em International Journal of Marine Energy}, 15:68--84, sep 2016.

\bibitem{Trefethen2019}
L.~N. Trefethen.
\newblock {\em Approximation Theory and Approximation Practice, Extended
  Edition}.
\newblock Society for Industrial and Applied Mathematics, jan 2019.

\bibitem{Tucker1984}
M.~Tucker, P.~Challenor, and D.~Carter.
\newblock Numerical simulation of a random sea: a common error and its effect
  upon wave group statistics.
\newblock {\em Applied Ocean Research}, 6(2):118--122, apr 1984.

\bibitem{Yang2022}
B.~Yang, S.~Wu, H.~Zhang, B.~Liu, H.~Shu, J.~Shan, Y.~Ren, and W.~Yao.
\newblock Wave energy converter array layout optimization: A critical and
  comprehensive overview.
\newblock {\em Renewable and Sustainable Energy Reviews}, 167:112668, oct 2022.

\bibitem{Zhang2019}
H.~Zhang, R.~Xi, D.~Xu, K.~Wang, Q.~Shi, H.~Zhao, and B.~Wu.
\newblock Efficiency enhancement of a point wave energy converter with a
  magnetic bistable mechanism.
\newblock {\em Energy}, 181:1152--1165, aug 2019.

\end{thebibliography}
\bibliographystyle{abbrv}
\end{document}